\documentclass[11pt,reqno]{amsart}
\usepackage{amsbsy,amsfonts,amsmath,amsopn,amssymb,amsthm,amsxtra}
\usepackage{bm,bbm,color,enumerate,float,fullpage,graphicx,ifthen}
\usepackage{mathrsfs,multirow,pbox,pgf,tikz,tikz-cd,tikz-3dplot,upref,url,xcolor}
\usepackage[utf8]{inputenc}
\usepackage[english]{babel}
\usepackage[all,color]{xy}
\usepackage[backref=page]{hyperref}
\hypersetup{
	colorlinks,
	final,
	pdfauthor={Jose Bastidas},
	pdftitle={The polytope algebra of generalized permutahedra},
	pdfkeywords={Polytope algebra, Tits algebra, Hopf monoid, Coxeter arrangement, Eulerian numbers},
	linkcolor = cyan,
	citecolor = magenta,
}
\usepackage[square,sort,comma,numbers]{natbib}
\usepackage{array} \newcolumntype{L}{>{\displaystyle}l} \newcolumntype{C}{>{\displaystyle}c} \newcolumntype{R}{>{\displaystyle}r}
\usetikzlibrary{shapes.multipart,positioning,decorations.pathreplacing,calc,intersections,folding}

\newtheorem{theorem}{Theorem}[section]
\newtheorem{proposition}[theorem]{Proposition}
\newtheorem{lemma}[theorem]{Lemma}
\newtheorem{corollary}[theorem]{Corollary}
\newtheorem{conjecture}[theorem]{Conjecture}
\newtheorem*{theorem*}{Theorem}
\theoremstyle{definition}

\newtheorem{example}[theorem]{Example}
\newtheorem{convention}[theorem]{Convention}
\theoremstyle{remark}
\newtheorem{remark}[theorem]{Remark}

\theoremstyle{plain}

\numberwithin{equation}{section}
\numberwithin{figure}{section}
\numberwithin{table}{section}

\definecolor{blue1}{HTML}{007CFC}
\definecolor{green1}{HTML}{7DC636}
\definecolor{red1}{HTML}{D20000}
\newcommand\blue[1]{{\color{blue1}{#1}}}

\DeclareMathOperator{\apode}{s}
\DeclareMathOperator{\Conv}{Conv}
\DeclareMathOperator{\GP}{GP}
\DeclareMathOperator{\Id}{Id}
\DeclareMathOperator{\relint}{relint}
\DeclareMathOperator{\supp}{s}
\DeclareMathOperator{\vol}{vol}
\newcommand{\arr}{\mathcal A}
\newcommand{\bdot}{\bm\cdot}
\newcommand{\B}[1]{\mathtt{#1}}

\newcommand{\cP}{{\mathcal{P}}}
\newcommand{\cS}{{\mathcal{S}}}
\newcommand{\carr}{{\mathcal{C}}}
\newcommand{\cpoly}{\chi}
\newcommand{\face}{\Sigma}
\newcommand{\frakc}{{\mathfrak{c}}}
\newcommand{\frakl}{{\mathfrak{l}}}
\newcommand{\frakz}{{\mathfrak{z}}}
\newcommand{\f}{{\mathfrak{f}}}
\newcommand{\maxflat}{\top}
\newcommand{\opp}[1]{\overline{#1}}
\newcommand{\p}{{\mathfrak{p}}}
\newcommand{\qand}{\quad\text{and}\quad}
\newcommand{\q}{{\mathfrak{q}}}
\newcommand{\rmH}{{\mathrm{H}}}
\newcommand{\RR}{\mathbb{R}}
\newcommand{\QQ}{\mathbb{Q}}
\newcommand{\rr}{{\mathfrak{r}}}
\newcommand{\sfh}{{\mathsf{h}}}
\newcommand{\sfp}{{\mathsf{p}}}
\newcommand{\sfq}{{\mathsf{q}}}
\newcommand{\tits}{\mathsf{\Sigma}}
\newcommand{\rmX}{\mathrm{X}}
\newcommand{\rmY}{\mathrm{Y}}
\newcommand{\ZZ}{\mathbb{Z}}
\newcommand{\fflat}{\mathcal{L}}
\newcommand{\qqand}{\qquad\text{and}\qquad}
\newcommand{\frakB}{{\mathfrak{B}}}
\newcommand{\frakC}{{\mathfrak{C}}}
\newcommand{\frakS}{{\mathfrak{S}}}

\DeclareMathOperator{\Des}{Des}
\DeclareMathOperator{\des}{des}
\DeclareMathOperator{\Exc}{Exc}
\DeclareMathOperator{\exc}{exc}
\DeclareMathOperator{\fexc}{fexc}
\DeclareMathOperator{\Neg}{Neg}
\DeclareMathOperator{\fneg}{neg}

\tikzset{
anticlockwise arc centered at/.style={
to path={
let \p1=(\tikztostart), \p2=(\tikztotarget), \p3=(#1) in
\pgfextra{
\pgfmathsetmacro{\anglestart}{atan2(\y1-\y3,\x1-\x3)}
\pgfmathsetmacro{\angletarget}{atan2(\y2-\y3,\x2-\x3)}
\pgfmathsetmacro{\angletarget}
{\angletarget < \anglestart ? \angletarget+360 : \angletarget}
\pgfmathsetmacro{\radius}{veclen(\y1-\y3,\x1-\x3)}}
arc(\anglestart:\angletarget:\radius pt) -- (\tikztotarget)},},
clockwise arc centered at/.style={
to path={
let \p1=(\tikztostart), \p2=(\tikztotarget), \p3=(#1) in
\pgfextra{
\pgfmathsetmacro{\anglestart}{atan2(\y1-\y3,\x1-\x3)}
\pgfmathsetmacro{\angletarget}{atan2(\y2-\y3,\x2-\x3)}
\pgfmathsetmacro{\angletarget}
{\angletarget > \anglestart ? \angletarget - 360 : \angletarget}
\pgfmathsetmacro{\radius}{veclen(\y1-\y3,\x1-\x3)}}
arc(\anglestart:\angletarget:\radius pt)  -- (\tikztotarget)},},}

\newcommand{\phantomarcThroughThreePoints}[4][]{
\coordinate (middle1) at ($(#2)!.5!(#3)$);
\coordinate (middle2) at ($(#3)!.5!(#4)$);
\coordinate (aux1) at ($(middle1)!1!90:(#3)$);
\coordinate (aux2) at ($(middle2)!1!90:(#4)$);
\coordinate (center#2) at ($(intersection of middle1--aux1 and middle2--aux2)$);
\draw[#1,draw=none,name path=kamaan#2]
let \p1=($(#2)-(center#2)$),
\p2=($(#4)-(center#2)$),
\n0={veclen(\p1)},
\n1={atan2(\y1,\x1)},
\n2={atan2(\y2,\x2)},
\n3={\n2>\n1?\n2:\n2+360}
in (#2) arc(\n1:\n3:\n0);}

\newcommand{\arcThroughThreePoints}[3]{
\phantomarcThroughThreePoints{#1}{#2}{#3};
\draw (#1) to[anticlockwise arc centered at=center#1] (#3);}

\begin{document}

\title{The polytope algebra of generalized permutahedra}

\author[J.~Bastidas]{Jose Bastidas}
\address{Department of Mathematics\\
Cornell University\\
Ithaca, NY 14853}
\email{\href{mailto:jdb394@cornell.edu}{jdb394@cornell.edu}}
\urladdr{\url{https://sites.google.com/view/bastidas}}

\begin{abstract}
The polytope subalgebra of deformations of a zonotope can be endowed with the structure of a module over the Tits algebra of the corresponding hyperplane arrangement. We explore this construction and find relations between statistics on (signed) permutations and the module structure in the case of (type B) generalized permutahedra. In type B, the module structure surprisingly reveals that any family of generators (via signed Minkowski sums) for generalized permutahedra of type B will contain at least $2^{d-1}$ full-dimensional polytopes. We find a generating family of simplices attaining this minimum. Finally, we prove that the relations defining the polytope algebra are compatible with the Hopf monoid structure of generalized permutahedra.
\end{abstract}
\maketitle

\section{Introduction}

Generalized permutahedra serve as a geometric model for many \emph{classical} (type A) combinatorial objects and have been extensively studied in resent years. Notably, Aguiar and Ardila~\citep{aa17} endowed generalized permutahedra with the structure of a Hopf monoid {\sf GP} in the category of species and, in so doing, gave a unified framework to study similar algebraic structures over many different families of combinatorial objects. At the same time, McMullen's polytope algebra~\citep{mcmullen89} offers a different algebraic perspective to study generalized permutahedra.

One of the goals of the present paper is to investigate the compatibility between both structures.
To achieve this goal, we take a more general approach and study deformations of an arbitrary zonotope.
We then specialize our results to deformations Coxeter permutahedra of type A and type B, revealing remarkable connections with (type B) Eulerian polynomials and statistics over (signed) permutations. The results in type B allow us to find a family of generalized permutahedra that generates all other deformations via signed Minkowski sums, thus solving a question posed by Ardila, Castillo, Eur, and Postnikov~\citep{acep20submodular}.

Let~$V$ be a finite-dimensional real vector space.
The polytope algebra~$\Pi(V)$ is generated by the classes~$[\p]$ of polytopes~$\p \subseteq V$.
These classes satisfy the following \emph{valuation} and \emph{translation invariance} relations:
$
[\p \cup \q] + [\p \cap \q] = [\p] + [\q]
$ and $
[\p + \{t\}] = [\p],
$
whenever $\p \cup \q$ is a polytope and for every $t \in V$.
The product structure of $\Pi(V)$ is given by the Minkowski sum of polytopes: $[\p] \cdot [\q] = [\p + \q]$.

\begin{figure}[ht]
\[
\left[
\begin{gathered}
\begin{tikzpicture}[scale=.7]
\draw [thick,fill=gray!30!white] (0,0) -- (1,0) -- (.5,.866) -- cycle;
\end{tikzpicture}
\end{gathered}
\right]
\cdot\left[
\begin{gathered}
\begin{tikzpicture}[scale=.7]
\draw [very thick,fill=gray!30!white] (0,0) -- (1,0);
\end{tikzpicture}
\end{gathered}
\right]
\quad = \quad 
\left[
\begin{gathered}
\begin{tikzpicture}[scale=.7]
\draw [thick,fill=gray!30!white] (0,0) -- (2,0) -- (1.5,.866) -- (.5,.866) -- cycle;
\end{tikzpicture}
\end{gathered}
\right]
\quad = \quad 
\left[
\begin{gathered}
\begin{tikzpicture}[scale=.7]
\draw [thick,fill=gray!30!white] (0,0) -- (1,0) -- (1.5,.866) -- (.5,.866) -- cycle;
\end{tikzpicture}
\end{gathered}
\right]
+
\left[
\begin{gathered}
\begin{tikzpicture}[scale=.7]
\draw [thick,fill=gray!30!white] (0,0) -- (1,0) -- (.5,.866) -- (-.5,.866) -- cycle;
\end{tikzpicture}
\end{gathered}
\right]
-
\left[
\begin{gathered}
\begin{tikzpicture}[scale=.7]
\draw [thick,fill=gray!30!white] (1,0) -- (1.5,.866) -- (.5,.866) -- cycle;
\end{tikzpicture}
\end{gathered}
\right]
\]
\caption{Different expressions for the class of the trapezoid above in the polytope algebra $\Pi(\RR^2)$.}
\end{figure}

For a fixed polytope $\p \subseteq V$, $\Pi(\p)$ denotes the subalgebra of $\Pi(V)$ generated by classes of \emph{deformations} of $\p$; see~\citep{mcmullen93simple}.
We are particularly interested in the case where $\p$ is a zonotope corresponding to a linear hyperplane arrangement $\arr$.
In this case, let $\tits[\arr]$ denote the Tits algebra of $\arr$, see Section~\ref{s:Tits}.
It is linearly generated by the elements $\B{H}_F$ as $F$ runs through the faces of the arrangement.
The following is a central result of this paper.

\newtheorem*{t:action}{Theorem~\ref{t:mod_str}}
\begin{t:action}
Let $\p$ and $\arr$ be as above.
The algebra $\Pi(\p)$ is a right $\tits[\arr]$-module under the following action:
\[
[\q] \cdot \B{H}_F := [\q_v],
\]
where $v \in \relint(F)$ and $\q_v$ denotes the face of $\q$ \emph{maximal in the direction} of $v$.
Moreover, the action of $\B{H}_F$ is an endomorphism of graded algebras.
\end{t:action}

In the particular case of the permutahedron and the braid arrangement, the compatibility between the algebra structure and the action of the Tits algebra is related to the Hopf monoid structure of Aguiar and Ardila.
Similar results for the Hopf monoid of extended generalized permutahedra were independently obtained by Ardila and Sanchez in~\citep{as20valGP}.

\newtheorem*{t:Hopf}{Theorem~\ref{t:McGP}}
\begin{t:Hopf}
The species~$\Pi$ defined by $\Pi[I] = \Pi(\pi_I)$, where $\pi_I \subseteq \RR^I$ is the standard permutahedron, is a Hopf monoid quotient of ${\sf GP}$.
\end{t:Hopf}

We embark on further understanding the module structure of $\Pi(\frakz)$ when $\frakz$ is a zonotope corresponding to a hyperplane arrangement $\arr$.
The decomposition of $\Pi(\frakz) = \bigoplus_r \Xi_r(\frakz)$ into its graded components will play an essential role here. McMullen characterized $\Xi_r(\frakz)$ as the eigenspace of the dilation morphism $\delta_\lambda$, defined by $\delta_\lambda[\p] = [\lambda \p]$, with eigenvalue $\lambda^r$ for any positive $\lambda \neq 1$.
Theorem~\ref{t:mod_str} then implies that each graded component is a $\tits[\arr]$-submodule.

The simple modules over~$\tits[\arr]$ are one-dimensional and indexed by the flats of the arrangement~\citep[Chapter 9]{am17}.
Given a module~$M$ over~$\tits[\arr]$, the number of copies of the simple module associated with the flat~$\rmX$ that appear as a composition factor of~$M$ is~$\eta_\rmX(M)$.
We propose that studying these algebraic invariants for the modules $\Xi_r(\frakz)$ yields important geometric and combinatorial information of the deformations of $\frakz$.
In Sections~\ref{s:braid} and \ref{s:typeB}, we do this for the algebra of deformations of the permutahedron and the type B permutahedron, respectively.
The main results in these sections relate the invariants $\eta_\rmX$ with statistics over (signed) permutations:

\newtheorem*{t:main}{Theorem~\ref{t:dims_simul_e-spaces_A}}
\begin{t:main}
For any flat~$\rmX$ of the braid arrangement in~$\RR^d$ and~$r=0,1,\dots,d-1$,
\[
\eta_\rmX(\Xi_r(\pi_d)) = \big|\{ \sigma \in \frakS_d \,:\, \supp(\sigma) = \rmX,\ \exc(\sigma) = r \}\big|.
\]
\end{t:main}

\newtheorem*{t:mainB}{Theorem~\ref{t:dims_simul_e-spaces_B}}
\begin{t:mainB}
For any flat~$\rmX$ of type B Coxeter arrangement in~$\RR^d$ and~$r=0,1,\dots, d$,
\[
\eta_\rmX(\Xi_r(\pi^B_d)) = \big|\{ \sigma \in \frakB_d \,:\, \supp(\sigma) = \rmX,\ \exc_B(\sigma) = r \}\big|.
\]
\end{t:mainB}

This surprising relation arises from the remarkable results of McMullen and Brenti that we explain now. McMullen~\citep{mcmullen93simple} proved that when $\p$ is a simple polytope, just like the (type B) permutahedron, the dimension of the graded components $\Xi_r(\p)$ are determined by the $h$-numbers of $\p$. On the other hand, building on top of work by Björner~\citep{bjorner84}, Brenti proved that the $h$-polynomial of the Coxeter permutahedron associated with any Coxeter group is the corresponding Eulerian polynomial~\citep[Theorem 2.3]{brenti94cox-eul}.

Work of Postnikov~\citep{postnikov09}, and of Ardila, Benedetti, and Doker~\citep{abd10volume} show that any generalized permutahedron in $\RR^d$ can be written as the \emph{signed Minkowski sum} of the faces of the standard simplex $\Delta_{[d]} = \Conv\{ e_j \,:\, j \in [d]\}$. As a consequence of Theorem~\ref{t:dims_simul_e-spaces_B}, we show that any such set of generators for type B generalized permutahedra contains at least $2^{d-1}$ full dimensional polytopes (see Proposition~\ref{p:min_num_gens}). We manage to obtain a family of generators that attains this minimum.

\newtheorem*{t:gensB}{Theorem~\ref{t:B-gens}}
\begin{t:gensB}
Every type B generalized permutahedron in $\RR^d$ can be written uniquely as a signed Minkowski sum of the simplices $\Delta_S$ and $\Delta^0_S$ for \emph{special involution-exclusive} subsets $S \subseteq [\pm d]$.
\end{t:gensB}

Unlike the standard simplex and its faces in the type A case, this collection of generators is not invariant under the action of the corresponding Coxeter group. However, using the set of generators in the previous theorem, we are able to find a different collection of generators that is invariant under the action of $\frakB_d$, at the cost of including twice as many full-dimensional polytopes.

\newtheorem*{t:gensBsym}{Theorem~\ref{t:B-sym-gens}}
\begin{t:gensBsym}
Every type B generalized permutahedron in $\RR^d$ can be written uniquely as a signed Minkowski sum of the simplices $\Delta^0_S$ for involution-exclusive subsets $S$.
\end{t:gensBsym}

This document is organized as follows. We review McMullen's construction in Section~\ref{s:polytope}.
In Section~\ref{s:Tits}, we recall the definition of the Tits algebra of a hyperplane arrangement and some of its representation theoretic properties. Characteristic elements and Eulerian idempotents are also reviewed in this section. In Section~\ref{s:module}, we begin the study of the polytope algebra of deformations of a zonotope as a module over the Tits algebra of the corresponding hyperplane arrangement, which is the central construction for this paper. We study the polytope algebra of generalized permutahedra in Section~\ref{s:braid}. In particular, we provide a conjectural basis of simultaneous eigenvectors for the action of the \emph{Adams element} and positive dilations on the module~$\Pi(\pi_d)$. Section~\ref{s:typeB} contains the analogous results for type B. We also give explicit sets of \emph{signed Minkowski generators} for type B generalized permutahedra; one shows that the lower bound on the number of full-dimensional polytopes in such a collection is tight, and the other is invariant under the action of $\frakB_d$. In Section~\ref{s:hopf} we prove that the valuation and translation invariance relations are compatible with the Hopf monoid structure of {\sf GP}. Section~\ref{s:remarks} contains some final remarks and questions. We begin with some preliminaries.

\subsection{Preliminaries}

Let~$V$ be a real vector space of dimension~$d$ endowed with an inner product~$\langle \,\cdot\, , \cdot\, \rangle$, and let ${\bf 0} \in V$ denote its zero vector.
For a polytope~$\p \subseteq V$ and a vector~$v \in V$, let~$\p_v$ denote the face of~$\p$ maximized in the direction~$v$.
That is,
\[
\p_v := \{ p \in \p \,:\, \langle p , v \rangle \geq \langle q , v \rangle \text{ for all } q \in \p\}.
\]
The (outer) \textbf{normal cone} of a face~$\f$ of~$\p$ is the polyhedral cone
\[
N(\f,\p) := \{ v \in V  \,:\,  \f \leq \p_v\} = \overline{\{ v \in V \,:\, \f = \p_v\}},
\]
and the \textbf{normal fan} of~$\p$ is the collection
$
\face_\p = \{ N(\f,\p) \,:\, \f \leq \p\}
$
of all normal cones of faces of $\p$.
There is a natural order-reversing correspondence between faces of~$\p$ and cones in~$\face_\p$.
For~$F \in \face_\p$, we let~$\p_F \leq \p$ denote the face whose normal cone is~$F$.
That is $\p_F = \p_v$ for any $v$ in the relative interior of $F$.

\begin{figure}[ht]
\[
\begin{gathered}
\begin{tikzpicture}[->]
\draw [blue1] (0,0) -- (-1,0) node [left] {\small$w$};
\draw [red1] (0,0) -- (0,1) node [above] {\small$v$};
\end{tikzpicture}
\end{gathered}
\hspace*{.1\linewidth}
\begin{gathered}
\begin{tikzpicture}
\draw [thick, fill=gray!30!white] (-1,0) -- (1,0) -- (1.5,.866) -- (1,1.732) -- (0,1.732) -- (-1,0);
\draw [line width=0.7mm, red1] (1,1.732) -- (0,1.732);
\node at (.5,.866) {\small$\p$};
\node [red1] at (.5,2) {\small$\p_v$};
\node [circle, inner sep = 1pt, fill = blue1, color = blue1, draw] at (-1,0) {};
\node [blue1] at (-1.3,0) {\small$\p_w$};
\end{tikzpicture}
\end{gathered}
\hspace*{.1\linewidth}
\begin{gathered}
\begin{tikzpicture}
\newdimen\R
\R=1.2cm \draw [fill, color = blue1, opacity = .3] (0,0) -- (150:\R) arc (150:270:\R);
\draw [fill, color = gray!20!white] (0,0) -- (-90:\R) arc (-90:150:\R);
\draw [<->] (30:\R) -- (0,0) -- (-90:\R);
\draw [<->] (150:\R) -- (-30:\R);
\draw [line width=0.4mm, red1, ->] (0,0) -- (90:\R) node [above] {\small $N(\p_v,\p)$};
\node [blue1] at (200:1.5\R) {\small $N(\p_w,\p)$};
\end{tikzpicture}
\end{gathered}
\]
\caption{A 2-dimensional polytope~$\p$ and two of its faces~$\p_v,\p_w$ maximized in directions~$v,w$, respectively.
On the right, the normal fan~$\face_\p$ and the normal cones corresponding to the faces~$\p_v,\p_w$ of~$\p$.}
\end{figure}
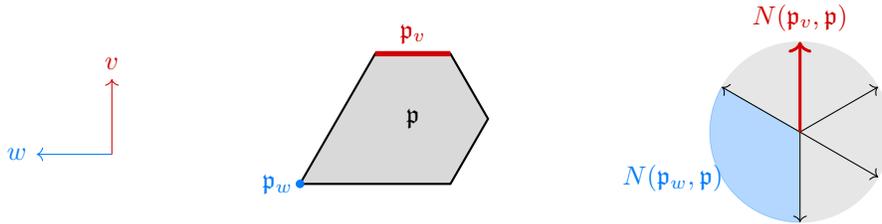

Recall that a fan~$\face$ \textbf{refines}~$\face'$ if every cone in~$\face'$ is a union of cones in~$\face$. We say that a polytope~$\q$ is a \textbf{deformation} of~$\p$ if~$\face_\p$ refines~$\face_\q$.
The \textbf{Minkowski sum} of two polytopes~$\p,\q \subseteq V$ is the polytope~$\p + \q := \{ p + q \,:\, p \in \p,\, q \in \q \}$.
We say that a polytope~$\q$ is a \textbf{Minkowski summand} of~$\p$ if~$\p = \q + \q'$ for some polytope~$\q'$.
The normal fan of~$\p + \q$ is the common refinement of~$\face_\p$ and~$\face_\q$. 
Hence,~$\face_\p$ refines the normal fan of any of its Minkowski summands, and the Minkowski sum of deformations of $\p$ is again a deformation of $\p$.

The~$f$-polynomial of a~$d$-dimensional polytope~$\p$ is
\[
f(\p,z) = \sum_{i=0}^d f_i(\p) z^i,
\]
where~$f_i(\p)$ is the number of~$i$-dimensional faces of~$\p$.
The~$h$-polynomial of~$\p$ is defined by
\[
h(\p,z) = \sum_{i=0}^d h_i(\p) z^i = f(\p,z-1).
\]
The sequences~$(f_0(\p),\dots,f_d(\p))$ and~$(h_0(\p),\dots,h_d(\p))$ are the~$f$-vector and~$h$-vector of~$\p$, respectively.
These polynomials behave nicely with respect to the Cartesian product of polytopes:
\[
f(\p \times \q,z) = f(\p,z)f(\q,z)
\qquad
h(\p \times \q,z) = h(\p,z)h(\q,z),
\]
where $\p \times \q = \{ (p,q) \,:\, p \in \p,\, q \in \q \}$.

\section{The polytope algebra}\label{s:polytope}

We briefly review the construction of McMullen's polytope algebra~\citep{mcmullen89} and its main properties. A hands on introduction to this topic can be found in the survey \citep{castillo19}.
The subalgebra relative to a fixed polytope~\citep{mcmullen93simple} is studied at the end of this section.

\subsection{Definition and structure theorem}

As an abelian group, the \textbf{polytope algebra}~$\Pi(V)$ is generated by elements~$[\p]$, one for each polytope~$\p \subseteq V$.
These generators satisfy the relations
\begin{equation}\label{eq:valuation_property}
[\p \cup \q] + [\p \cap \q] = [\p] + [\q],
\end{equation}
whenever~$\p$,~$\q$ and~$\p \cup \q$ are polytopes; and
\begin{equation}\label{eq:translation_property}
[\p + \{t\}] = [\p],
\end{equation}
for any polytope~$\p$ and translation vector~$t \in V$.
These relations are referred as the \textbf{valuation property} and the \textbf{translation invariance property}, respectively.
The group~$\Pi(V)$ is endowed with a commutative product defined on generators by means of the Minkowski sum
\[
[\p] \cdot [\q] := [ \p + \q ].
\]
It readily follows from~\eqref{eq:translation_property} that the class of a point~$1 := [\{{\bf 0}\}]$ is the unit of~$\Pi(V)$.

For each scalar $\lambda \in \RR$, the \textbf{dilation morphism}~$\delta_\lambda:\Pi(V) \rightarrow \Pi(V)$ is defined on generators by
$
\delta_\lambda [\p] = [\lambda \p].
$
Recall that for any subset~$S \subseteq V$ and~$\lambda \in \RR$, $\lambda S := \{\lambda v \,:\, v \in S\}$.
One can easily verify that~$\delta_\lambda$ preserves the valuation~\eqref{eq:valuation_property} and translation invariance~\eqref{eq:translation_property} relations, and that it defines a morphism of rings.
The main structural result for the polytope algebra is the following.

\begin{theorem}[{\citep[Theorem 1]{mcmullen89}}]
The commutative ring~$\Pi(V)$ is almost a graded~$\RR$-algebra, in the following sense:
\begin{enumerate}[i.]
\item as an abelian group,~$\Pi(V)$ admits a direct sum decomposition
\[
\Pi(V) = \Xi_0(V) \oplus \Xi_1(V) \oplus \cdots \oplus \Xi_d(V);
\]
\item under multiplication,
$
\Xi_r(V) \cdot \Xi_s(V) = \Xi_{r+s}(V),
$
with~$\Xi_r(V) = 0$ for~$r > d$;
\item~$\Xi_0(V) \cong \ZZ$, and for~$r=1,\dots,d$,~$\Xi_r(V)$ is a real vector space;
\item the product of elements in~$\bigoplus_{r \geq 1} \Xi_r(V)$ is bilinear.
\item the dilations~$\delta_\lambda$ are algebra endomorphisms, and for~$r = 0,1,\dots,d$, if~$x \in \Xi_r(V)$ and~$\lambda \geq 0$, then
$
\delta_\lambda x = \lambda^r x.
$
\end{enumerate}
\end{theorem}

We discuss the definition of the graded components $\Xi_r(V)$ below. The component $\Xi_0(V)$ of degree $0$ is simply the subring of~$\Pi(V)$ generated by~$1$, and thus $\Xi_0(V) \cong \ZZ$. 
Let~$Z_1$ be the subgroup of~$\Pi(V)$ generated by all elements of the form~$[\p] - 1$.
Observe that $\delta_0[\p] = 1$ for every polytope $\p$, so $Z_1 = \ker\delta_0$ is an ideal.
As an abelian group,~$\Pi(V)$ has a direct sum decomposition
$
\Pi(V) = \Xi_0(V) \oplus Z_1.
$
Moreover, $Z_1$ is a nil ideal, since for any~$k$-dimensional polytope $\p$,
\[
([\p] - 1)^r = 0
\quad \text{for } r > k.
\]
This is {\citep[Lemma 13]{mcmullen89}}. McMullen also shows that $Z_1$ has the structure of a vector space (first over $\QQ$ and then over $\RR$). Therefore, we can define inverse maps
\[
\begin{tikzcd}
1 + Z_1 \arrow[rr, "\log" description, bend left] &  & Z_1 \arrow[ll, "\exp" description, bend left]
\end{tikzcd}
\]
by means of their usual power series.
In particular, we can define the \emph{log-class} of a $k$-dimensional polytope $\p$ by
\begin{equation}\label{eq:def-log}
\log[\p] := \log(1 + ([\p]-1)) = \sum_{r = 1}^k \dfrac{(-1)^{r-1}}{r}([\p]-1)^r.
\end{equation}
Using the exponential map, we recover $[\p]$ from $\log[\p]$:
\begin{equation}\label{eq:def-exp}
[\p] = \sum_{r = 0}^k \dfrac{1}{r!} (\log[\p])^r.
\end{equation}
The $\log$ and $\exp$ maps satisfy the standard properties of logarithms and exponentials. In particular,
\begin{equation}\label{eq:log_add}
\log[\p + \q] = \log([\p] \cdot [\q]) = \log[\p] + \log[\q].
\end{equation}

\begin{example}\label{ex:prod_intervals}
Let~$v_1,\dots,v_k \in V$ be nonzero vectors and let~$\frakl_i$ denote the line segment~$\Conv\{ {\bf 0} ,v_i\}$.
Then, $\log[\frakl_i] = [\frakl_i] - 1$.
We will see that the product $\prod_{i=1}^k \log[\frakl_i] \in \Xi_k(V)$ is nonzero if and only if the collection $\{v_1,\dots,v_k\}$ is linearly independent. See Figure~\ref{f:ex-prod-log} for an example with $k = 3$.

Consider the polytope~$\frakz = \frakl_1 + \frakl_1 + \dots + \frakl_k$, a Minkowski sum of segments.
Using the logarithm property~\eqref{eq:log_add}, we get
$
(\log[\frakz])^k 
= ( \sum_{i=1}^k \log[\frakl_i] )^k 
= k! \prod_{i=1}^k \log[\frakl_i].
$
The last equality follows since~$k! \prod \log[\frakl_i]$ is the only square-free term in the expansion of~$( \sum\log[\frakl_i])^k$, and~$(\log[\frakl])^2 = ([\frakl] - 1)^2 = 0$ for any line segment~$\frakl$.
Finally,~$(\log[\frakz])^k \neq 0$ if and only if~$k \leq \dim(\frakz)$, and~$\frakz$ being the sum of~$k$ line segments has dimension at most~$k$, with equality precisely when the vectors~$v_1,\dots,v_k$ are linearly independent.
\end{example}

\begin{figure}[ht]
\[
\begin{gathered}
\begin{tikzpicture}[->,black]
\draw (0,0) -- (0:1) node [right] {\small$v_1$};
\draw (0,0) -- (60:1) node [above right] {\small$v_2$};
\draw (0,0) -- (120:1) node [above left] {\small$v_3$};
\end{tikzpicture}
\end{gathered}
\]
\[
([\frakl_1]-1)([\frakl_2]-1)
= \quad
\begin{gathered}
\begin{tikzpicture}
\draw [fill=gray!30!white] (0,0) -- (1,0) -- (1.5,.866) -- (.5,.866) -- (0,0);
\node [] at (.75,.433) {\small$[\frakl_1][\frakl_2]$};
\end{tikzpicture}
\end{gathered}
\quad - \quad
\begin{gathered}
\begin{tikzpicture}
\node [above] at (-.5,0) {\small$[\frakl_1]$};
\draw [thick] (0,0) -- (-1,0);
\end{tikzpicture}
\end{gathered}
\quad - \quad
\begin{gathered}
\begin{tikzpicture}
\node [above] at (.6,.1) {\small$[\frakl_2]$};
\draw [shift={(.05,.05)}, thick] (0,0) -- (.5,.866);
\end{tikzpicture}
\end{gathered}
\quad + \quad
\bullet
\quad = \quad
\begin{gathered}
\begin{tikzpicture}
\draw [color = white, fill=gray!30!white] (0,0) -- (1,0) -- (1.5,.866) -- (.5,.866) -- (0,0);
\draw [] (.5,.866) -- (0,0) -- (1,0);
\draw [dashed] (.5,.866) -- (1.5,.866) -- (1,0);
\end{tikzpicture}
\end{gathered}
\]
\begin{multline*}
([\frakl_1]-1)([\frakl_2]-1)([\frakl_3]-1)
= \\
\begin{gathered}
\begin{tikzpicture}
\draw [fill=gray!30!white] (0,0) -- (1,0) -- (1.5,.866) -- (1,1.732) -- (0,1.732) -- (-.5,.866) -- (0,0);
\node at (.5,.866) {\small$[\frakl_1][\frakl_2][\frakl_3]$};
\draw [dotted] (-.5,.866) -- (.5,.866) -- (1,1.732);
\draw [dotted] (.5,.866) -- (1,0);
\end{tikzpicture}
\end{gathered}
\quad - \quad
\begin{gathered}
\begin{tikzpicture}
\draw [fill=gray!30!white] (0,0) -- (1,0) -- (.5,.866) -- (-.5,.866) -- (0,0);
\node at (.25,.433) {\small$[\frakl_1][\frakl_3]$};
\draw [shift={(1.1,.05)}, fill=gray!30!white] (0,0) -- (.5,.866) -- (0,1.732) -- (-.5,.866) -- (0,0);
\node [shift={(1.1,.05)}] at (0,.866) {\small$[\frakl_2][\frakl_3]$};
\draw [shift={(-.5,.966)}, fill=gray!30!white] (0,0) -- (1,0) -- (1.5,.866) -- (.5,.866) -- (0,0);
\node [shift={(-.5,.966)}] at (.75,.433) {\small$[\frakl_1][\frakl_2]$};
\end{tikzpicture}
\end{gathered}
\quad + \quad
\begin{gathered}
\begin{tikzpicture}
\node [above] at (-.5,0) {\small$[\frakl_1]$};
\node [above] at (.6,.1) {\small$[\frakl_2]$};
\node [below] at (.6,-.1) {\small$[\frakl_3]$};
\draw [thick] (0,0) -- (-1,0);
\draw [shift={(.05,.05)}, thick] (0,0) -- (.5,.866);
\draw [shift={(.05,-.05)}, thick] (0,0) -- (.5,-.866);
\end{tikzpicture}
\end{gathered}
\quad - \quad
\bullet
\quad = 0
\end{multline*}
\label{f:ex-prod-log}
\caption{Vectors $v_1,v_2,v_3$ lie in the same plane. The vectors $v_1,v_2$ are linearly independent and $\log[\frakl_1]\log[\frakl_2]$ represents \emph{the class of a half-open parallelogram}. In contrast, the product $\log[\frakl_1]\log[\frakl_2]\log[\frakl_3]$ is zero.}
\end{figure}
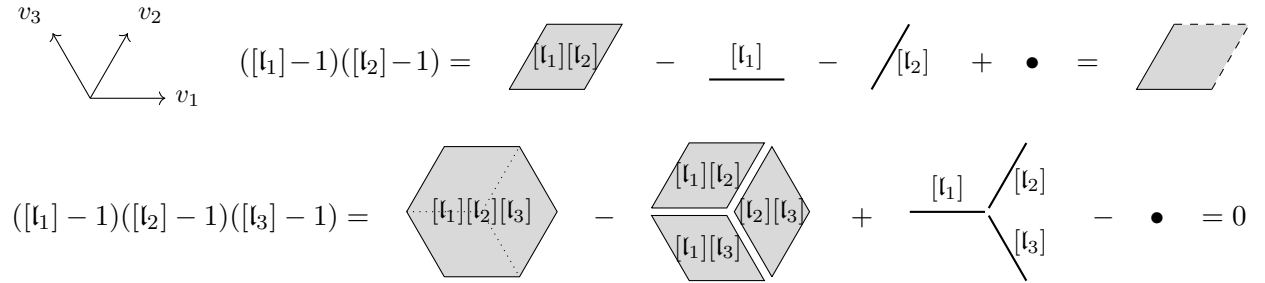

For~$r \geq 1$ let~$\Xi_r(V)$ be the subgroup (or subspace) of~$\Pi(V)$ generated by elements of the form~$(\log[\p])^r$. The following result implies that the sum $\Pi(V) = \Xi_0(V) + \Xi_1(V) + \dots + \Xi_d(V)$ is direct, and characterizes each graded component as the space of eigenvectors for the positive dilations $\delta_\lambda$.

\begin{lemma}[{\citep[Lemma 20]{mcmullen89}}]
Let~$x \in \Pi(V)$ and ~$\lambda > 0$, with~$\lambda \neq 1$.
Then,
\begin{equation}\label{eq:Xi_dilation}
x \in \Xi_r(V)
\qquad\text{if and only if}\qquad
\delta_\lambda x = \lambda^r x.
\end{equation}
\end{lemma}

It is clear by the definition of the graded components $\Xi_r(V)$ that the class of a half-open parallelogram in Figure~\ref{f:ex-prod-log} is in $\Xi_2(V)$. We can also verify this using~\eqref{eq:Xi_dilation} with $\lambda = 2$:
\[
\delta_2 \left(
\begin{gathered}
\begin{tikzpicture}
\draw [color = white, fill=gray!30!white] (0,0) -- (1,0) -- (1.5,.866) -- (.5,.866) -- (0,0);
\draw [] (.5,.866) -- (0,0) -- (1,0);
\draw [dashed] (.5,.866) -- (1.5,.866) -- (1,0);
\end{tikzpicture}
\end{gathered}
\right)
=
\begin{gathered}
\begin{tikzpicture}[scale = 2]
\draw [color = white, fill=gray!30!white] (0,0) -- (1,0) -- (1.5,.866) -- (.5,.866) -- (0,0);
\draw [] (.5,.866) -- (0,0) -- (1,0);
\draw [dashed] (.5,.866) -- (1.5,.866) -- (1,0);
\end{tikzpicture}
\end{gathered}
=
\begin{gathered}
\begin{tikzpicture}
\draw [color = white, fill=gray!30!white] (0,0) -- (1,0) -- (1.5,.866) -- (.5,.866) -- (0,0);
\draw [] (.5,.866) -- (0,0) -- (1,0);
\draw [dashed] (.5,.866) -- (1.5,.866) -- (1,0);

\draw [shift={(1.15,0)}, color = white, fill=gray!30!white] (0,0) -- (1,0) -- (1.5,.866) -- (.5,.866) -- (0,0);
\draw [shift={(1.15,0)}] (.5,.866) -- (0,0) -- (1,0);
\draw [shift={(1.15,0)}, dashed] (.5,.866) -- (1.5,.866) -- (1,0);

\draw [shift={(.6,1.039)}, color = white, fill=gray!30!white] (0,0) -- (1,0) -- (1.5,.866) -- (.5,.866) -- (0,0);
\draw [shift={(.6,1.039)}] (.5,.866) -- (0,0) -- (1,0);
\draw [shift={(.6,1.039)}, dashed] (.5,.866) -- (1.5,.866) -- (1,0);

\draw [shift={(1.75,1.039)}, color = white, fill=gray!30!white] (0,0) -- (1,0) -- (1.5,.866) -- (.5,.866) -- (0,0);
\draw [shift={(1.75,1.039)}] (.5,.866) -- (0,0) -- (1,0);
\draw [shift={(1.75,1.039)}, dashed] (.5,.866) -- (1.5,.866) -- (1,0);
\end{tikzpicture}
\end{gathered}
=
2^2 
\begin{gathered}
\begin{tikzpicture}
\draw [color = white, fill=gray!30!white] (0,0) -- (1,0) -- (1.5,.866) -- (.5,.866) -- (0,0);
\draw [] (.5,.866) -- (0,0) -- (1,0);
\draw [dashed] (.5,.866) -- (1.5,.866) -- (1,0);
\end{tikzpicture}
\end{gathered}
\]

\begin{convention}
As in later work of McMullen~\citep{mcmullen93separation,mcmullen93simple}, we replace~$\Xi_0(V) \cong \ZZ$ with the tensor product~$\Xi_0(V)_\RR := \RR \otimes \Xi_0(V) \cong \RR$ to get a genuine graded~$\RR$-algebra $\Pi(V)_\RR := \Xi_0(V)_\RR \oplus Z_1$.
To simplify notation, we drop the subscript $\RR$ and sometimes wirte $\Xi_r$ and $\Pi$ instead of $\Xi_r(V)$ and~$\Pi(V)_\RR$.
\end{convention}

For an arbitrary vector~$v \in V$, we can define a \textbf{maximization operator}~$\p \mapsto \p_v$ on the space of all polytopes~$\p \subseteq V$.
The next result shows that it induces a well-defined map on~$\Pi(V)$.

\begin{theorem}[{\citep[Theorem 7]{mcmullen89}}]\label{t:direction_morphism}
The map~$\p \mapsto \p_v$ induces an endomorphism $x \mapsto x_v$ on $\Pi(V)$,
defined on generators by
\[
[\p] \mapsto [\p]_v := [\p_v].
\]
This endomorphism commutes with nonnegative dilations.
\end{theorem}

In particular, the morphism~$x \mapsto x_v$ restricts to each graded component~$\Xi_r$.

Lastly, the \textbf{Euler map}~$x \mapsto x^*$ is the linear operator defined on generators by
\begin{equation}\label{eq:Euler_map}
[\p]^* = \sum_{\q \leq \p}(-1)^{\dim(\q)} [\q].
\end{equation}
The sum runs over all nonempty faces~$\q$ of~$\p$.
Up to a sign, the element~$[\p]^*$ corresponds to \emph{the class of the interior} of~$\p$.

\begin{theorem}[{\citep[Theorem 2]{mcmullen89}}]\label{t:euler_dilations}
The Euler map is an involutory automorphism of~$\Pi(V)$.
Moreover, for~$x \in \Xi_r(V)$ and~$\lambda < 0$,
\[
\delta_\lambda x = \lambda^r x^*.
\]
\end{theorem}

\subsection{Subalgebra relative to a fixed polytope}\label{ss:subalgebra_polytope}

Fix a polytope~$\p \subseteq V$.
The \textbf{subalgebra relative to~$\p$}, denoted~$\Pi(\p)$, is the subalgebra of~$\Pi(V)$ generated by classes~$[\q]$ of deformations~$\q$ of~$\p$. 
It is worth pointing out that if $\q,\q'$ are deformations of $\p$ such that $\q \cup \q'$ is a polytope, then both $\q \cup \q'$ and $\q \cap \q'$ are deformations of $\p$.
This follows since, in this case, $\q \cup \q' + \q \cap \q' = \q + \q'$ and Minkowski summands of a deformation of $\p$ are again deformations of $\p$.
Thus, the valuation property~\eqref{eq:valuation_property} is not introducing classes of new polytopes to $\Pi(\p)$.

\begin{remark}
McMullen~\cite{mcmullen93simple} originally defined~$\Pi(\p)$ as the subalgebra generated by the classes of Minkowski summands of~$\p$.
The following result of Shephard~\cite[Section 15.2.7]{grunbaum03convex} implies that both definitions are equivalent:
a polytope $\q$ is a deformation of $\p$ if and only if for small enough $\lambda > 0$, $\lambda \q$ is a Minkowski summand of $\p$.
\end{remark}

The relations between $[\q]$ and $\log[\q]$ in~\eqref{eq:def-log} and~\eqref{eq:def-exp} show that $\Pi(\p)$ is generated by homogeneous elements.
Thus, the grading of~$\Pi(V)$ induces a grading of~$\Pi(\p)$.
We let~$\Xi_r(\p) = \Pi(\p) \cap \Xi_r(V)$ denote the component of $\Pi(\p)$ in degree $r$.
Unlike the full algebra $\Pi(V)$, the subalgebra $\Pi(\p)$ has finite dimension.
McMullen described the dimension of the graded components of $\Pi(\p)$ when $\p$ is a simple polytope.

\begin{theorem}[{\citep[Theorem 6.1]{mcmullen93simple}}]\label{t:simple_h_dimension}
Let~$\p$ be a~$d$-dimensional simple polytope.
Then,
\[\dim_\RR(\Xi_r(\p)) = h_r(\p)\]
for~$r = 0,1,\dots,d$.
\end{theorem}

Let~$\f$ be a face of~$\p$ and~$v \in \relint\big(N(\f,\p)\big)$.
The maximization operator~$x \mapsto x_v$ defines a morphism of graded algebras
\begin{equation}\label{eq:morphism_to_face}
\psi_\f : \Pi(\p) \rightarrow \Pi(\f)
\end{equation}
that only depends on the face~$\f$ and not on the particular choice of~$v \in \relint\big(N(\f,\p)\big)$.

First observe that this map is well defined; that is,~$[\q_v] \in \Pi(\f)$ for every generator~$[\q]$ of~$\Pi(\p)$.
Indeed, if~$\q$ is a summand of~$\p$, say~$\p = \q + \q'$, then
$
\f = \p_v = \q_v + \q'_v,
$
so~$\q_v$ is a Minkowski summand of~$\f$.
Moreover, since the normal fan of~$\p$ refines that of~$\q$, then~$\q_w = \q_v$ for any other~$w \in \relint\big(N(\f,\p)\big)$.
Therefore the morphism~\eqref{eq:morphism_to_face} only depends on~$\f$.

\begin{theorem}[{\citep[Theorem 2.4]{mcmullen93simple}}]\label{t:simple_surjective}
Let~$\p$ be a simple polytope and~$\f$ a face of~$\p$.
Then, the morphism~$\psi_\f$ is surjective.
\end{theorem}

\section{The Tits algebra of a linear hyperplane arrangement}\label{s:Tits}

\subsection{Basic definitions}

Let~$\arr$ be a linear hyperplane arrangement in~$V$.
An arbitrary intersection of hyperplanes in $\arr$ is a \textbf{flat} of the arrangement.
The set~$\fflat[\arr]$ of flats of $\arr$ is a lattice with maximum~$\top:= V$ and minimum~$\perp$, the intersection of all hyperplanes in~$\arr$.
For a flat $\rmX \in \fflat[\arr]$, the \textbf{arrangement under~$\rmX$} is the following hyperplane arrangement in ambient space~$\rmX$:
\[
\arr^\rmX = \{\rmH\cap\rmX \,:\, \rmH\in\arr,\, \rmX\not\subseteq\rmH \}.
\]
The \textbf{characteristic polynomial} of~$\arr$ is
\[
\cpoly(\arr,t) := \sum_{\rmX \in \fflat[\arr]} \,\mu(\rmX,\maxflat) \,t^{\dim(\rmX)},
\]
where~$\mu(\rmX,\maxflat)$ denotes the Möbius function of~$\fflat[\arr]$.
It is a monic polynomial of degree~$\dim(V)$.

The hyperplanes in~$\arr$ split~$V$ into a collection~$\face[\arr]$ of polyhedral cones called \textbf{faces} of $\arr$.
Explicitly, the complement in~$V$ of the union of hyperplanes in~$\arr$ is the disjoint union of open subsets of~$V$; and~$\face[\arr]$ is the collection of the closures of these regions together with all their faces.
$\face[\arr]$ is a poset under containment, its maximal elements are called \textbf{chambers}.
The \textbf{central face}~$O$ is the minimum element of $\face[\arr]$, it coincides with the minimal flat~$\perp$ of the arrangement.
See Figure~\ref{f:ex_arr} for an example.

The \textbf{support} of a face~$F$ is the smallest flat~$\supp(F)$ containing it.
Equivalently, it is the linear span of~$F$.
The support map
\begin{equation}\label{eq:supp-map}
\supp:\face[\arr]\to\fflat[\arr]
\end{equation}
is surjective and order preserving.

\begin{figure}[ht]
\[
\begin{array}{C}
\begin{gathered}
\begin{tikzpicture}[<->,black]
\node [] at (-2.2,1.3) {$\arr$};
\draw (-90:1.5) -- (90:1.5) node [above] {\small$\rmH_1$};
\draw (150:1.5) -- (-30:1.5) node [below right] {\small$\rmH_3$};
\draw (210:1.5) -- (30:1.5) node [above right] {\small$\rmH_2$};
\draw [fill] (0,0) circle (.1);
\end{tikzpicture}
\end{gathered}
\\ \phantom{a}\\
\begin{gathered}
\begin{tikzpicture}[scale=.9]
\node [fill=white] at (-3.5,2) {$\face[\arr]$};
\draw (-2.5,1) -- (2.5,2);
\foreach \x in {-2.5 , -1.5 , -.5, .5 , 1.5} {
\draw[preaction={draw=white, -,line width=6pt}] (\x+1,1) -- (\x,2);
}
\foreach \x in {-2.5 , -1.5 , -.5, .5 , 1.5 , 2.5} {
\draw (0,0) -- (\x,1) -- (\x,2);
\draw[preaction={draw=white, -,line width=6pt}] (\x,1) -- (\x,2);
\foreach \y in {1,2} {
\node [fill=white] at (\x,\y) {};
\node [circle, inner sep=1.5pt,fill,draw] at (\x,\y) {};
}
}
\node [fill=white] at (0,0) {\small~$O$};
\end{tikzpicture}
\end{gathered}
\hspace*{.1\linewidth}
\begin{gathered}
\begin{tikzpicture}[scale=.9]
\draw (0,0) -- (-1.4,1) -- (0,2);
\draw (0,0) -- (0,1) -- (0,2);
\draw (0,0) -- (1.4,1) -- (0,2);
\node [fill=white] at (-2,2) {$\fflat[\arr]$};
\node [color=blue1,fill=white] at (2,2) {\small~$\mu(\bdot,\top)$};
\node [fill=white] at (0,0) {\small~$\perp^{\blue{2}}$};
\node [fill=white] at (-1.2,1) {\small~${\rmH_1}^{\blue{-1}}$};
\node [fill=white] at (0.05,1) {\small~${\rmH_2}^{\blue{-1}}$};
\node [fill=white] at (1.2,1) {\small~${\rmH_3}^{\blue{-1}}$};
\node [fill=white] at (0,2) {\small~$\top^{\blue{1}}$};
\end{tikzpicture}
\end{gathered}
\end{array}
\]
\caption{A 2-dimensional arrangement~$\arr$ together with
its poset of faces (left) and lattice of flats (right).
The Möbius function of the lattice of flats is shown in blue.
The characteristic polynomial of~$\arr$ is~$\cpoly(\arr,t) = t^2 - 3t + 2$.
}
\label{f:ex_arr}
\end{figure}
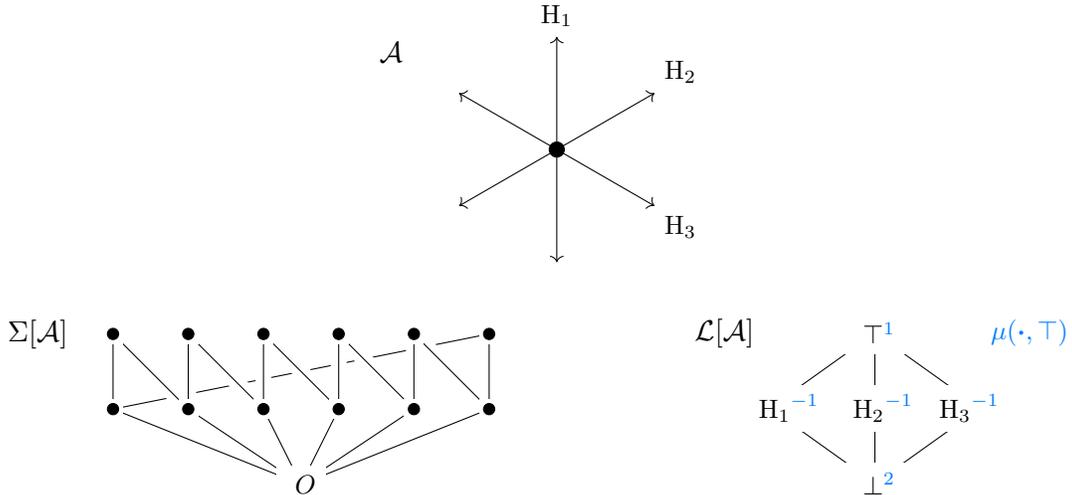

\subsection{The Tits algebra and characteristic elements}

The set~$\face[\arr]$ has the structure of a monoid under the \textbf{Tits product} illustrated in Figure~\ref{f:ex_arr_prod}.
The product of two faces~$F$ and~$G$, denoted~$FG$, is the first face you encounter after moving a small positive distance from an interior point of~$F$ to an interior point of~$G$.
The unit of this product is the central face~$O \in \face[\arr]$.
The \textbf{Tits algebra} of~$\arr$ is the monoid algebra $\tits[\arr] := \RR \Sigma[\arr]$.
See~\cite[Chapters 1 and 9]{am17} for more details.
We let~$\B{H}_F$ denote the basis element of~$\tits[\arr]$ associated with the face~$F$ of~$\arr$.

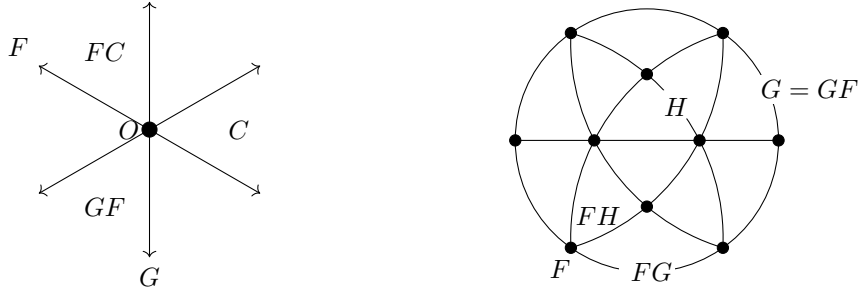
\begin{figure}[ht]
\[
\begin{gathered}
\begin{tikzpicture}[<->,black]
\newdimen\R
\R=1.7cm \draw (-90:\R) node [below] {\small$G$} -- (0,0) node [left] {\small$O$} -- (90:\R);
\draw (150:\R) node [above left] {\small$F$} -- (-30:\R);
\draw (210:\R) -- (30:\R);
\draw [fill] (0,0) circle (.1);
\node [] at (0:.7\R) {\small$C$};
\node [] at (120:.7\R) {\small$FC$};
\node [] at (240:.7\R) {\small$GF$};
\end{tikzpicture}
\end{gathered}
\hspace*{.2\linewidth}
\begin{gathered}
\begin{tikzpicture}[scale=.7,black]
\newdimen\R
\R=2.5cm \coordinate (madhya) at (0,0);
\draw (madhya) circle (\R); \coordinate (P) at (-1,0); \coordinate (Q) at (1,0); \coordinate (A1) at (54.735:\R);
\coordinate (B1) at (180+54.735:\R);
\arcThroughThreePoints{A1}{P}{B1};
\arcThroughThreePoints{B1}{Q}{A1};
\coordinate (A2) at (180-54.735:\R);
\coordinate (B2) at (-54.735:\R);
\arcThroughThreePoints{A2}{P}{B2};
\arcThroughThreePoints{B2}{Q}{A2};
\coordinate (A3) at (0:\R);
\coordinate (B3) at (180:\R);
\draw (A3) -- (B3);
\path[name intersections={of=kamaanA1 and kamaanB2, by=a1b2}];
\path[name intersections={of=kamaanB1 and kamaanA2, by=b1a2}];
\node [circle, inner sep=1.5pt,fill,draw] at (a1b2) {};
\node [circle, inner sep=1.5pt,fill,draw] at (b1a2) {};
\node [circle, inner sep=1.5pt,fill,draw] at (P) {};
\node [circle, inner sep=1.5pt,fill,draw] at (Q) {};
\node [circle, inner sep=1.5pt,fill,draw] at (A1) {};
\node [circle, inner sep=1.5pt,fill,draw] at (B1) {};
\node [circle, inner sep=1.5pt,fill,draw] at (A2) {};
\node [circle, inner sep=1.5pt,fill,draw] at (B2) {};
\node [circle, inner sep=1.5pt,fill,draw] at (A3) {};
\node [circle, inner sep=1.5pt,fill,draw] at (B3) {};
\node [black] at (180+54.735:1.2\R) {\small~$F$};
\node [black] at (180+54.735:0.7\R) {\small~$FH$};
\node [black, circle, inner sep = 0pt, minimum size = 1pt, fill=white] at (0.47,.63) {\small~$H$};
\node [black, inner sep = 2pt, fill=white, right] at (27.367:.85\R) {\small~$G=GF$};
\node [black, inner sep = 1pt, fill=white] at (270:1\R) {\small~$FG$};
\end{tikzpicture}
\end{gathered}
\]
\caption{Product of faces in arrangements in $\RR^2$ (left) and $\RR^3$ (right).
The second arrangement has been intersected with a sphere around the origin, its central face is not visible.
}
\label{f:ex_arr_prod}
\end{figure}

We view~$\fflat[\arr]$ as a commutative monoid with the join operation $\vee$ for the product.
This makes the support map~\eqref{eq:supp-map} a morphism of monoids.
We let~$\B{H}_\rmX$ denote the basis element of~$\RR\fflat[\arr]$ associated to the flat~$\rmX$ of~$\arr$, so that~$\B{H}_\rmX \cdot \B{H}_\rmY = \B{H}_{\rmX \vee \rmY}$.

A result of Solomon~\cite[Theorem 1]{solomon67} shows that the monoid algebra~$\RR \fflat[\arr]$ is split-semisimple.
This rests on the fact that the unique complete system of orthogonal idempotents for~$\RR \fflat[\arr]$ consists of elements~$\B{Q}_\rmX$ uniquely determined by
\begin{equation}\label{eq:H_and_Q}
\B{H}_\rmX = \sum_{\rmY:\,\rmY\geq\rmX}\B{Q}_\rmY
\quad\text{or equivalently}\quad
\B{Q}_\rmX = \sum_{\rmY:\,\rmY \geq \rmX} \mu(\rmX,\rmY) \B{H}_\rmY.
\end{equation}
In particular, $\RR \fflat[\arr]$ is the maximal split-semisimple quotient of~$\tits[\arr]$ via the support map and the simple modules of~$\tits[\arr]$ are indexed by flats.
The character~$\chi_\rmX$ of the simple module associated with the flat~$\rmX$ evaluated on an element
\[
w = \sum_F w^F \B{H}_F
\]
of~$\tits[\arr]$ yields
\begin{equation}\label{eq:elt-char}
\chi_\rmX(w)
= \sum_{F:\,\supp(F)\leq\rmX} w^F.
\end{equation}

Let~$t$ be a fixed scalar.
An element~$w$ of the Tits algebra is \textbf{characteristic of parameter~$t$} if for each flat~$\rmX$
\[
\chi_\rmX(w) = t^{\dim(\rmX)},
\]
with~$\chi_\rmX(w)$ as in~\eqref{eq:elt-char}.
Characteristic elements determine the characteristic polynomial of the arrangement and also determine the characteristic polynomial of the arrangements under each flat.
See~\citep[Section 12.4]{am17} and~\citep{abm19} for more information.

\subsection{Eulerian idempotents and diagonalization}

An \textbf{Eulerian family} of~$\arr$ is a collection of idempotent and mutually orthogonal elements~$\{\B{E}_\rmX\}_{\rmX \in \fflat[\arr]} \subseteq \tits[\arr]$ of the form
\[
\B{E}_\rmX = \sum_{F \,:\, \supp(F) \geq \rmX} a^F \B{H}_F,
\]
with~$a^F \neq 0$ for at least one~$F$ with~$\supp(F) = \rmX$. It follows that~$\{\B{E}_\rmX\}_\rmX$ is a complete system of primitive orthogonal idempotents and that $\supp(\B{E}_\rmX) = \B{Q}_\rmX$ \cite[Theorem 11.20]{am17}.
That is,
\[
\B{E}_\rmX \B{E}_\rmY = \begin{cases}
\B{E}_\rmX & \text{if } \rmX = \rmY,\\
0 & \text{otherwise},
\end{cases}
\qquad\qquad\qquad
\sum_\rmX \B{E}_\rmX = \B{H}_O,
\]
and $\B{E}_\rmX$ cannot be written as the sum of two non-trivial idempotents.

\begin{example}\label{ex:Eul_C2}
Let $\carr_2$ be the coordinate arrangement in $\RR^2$.
The following is an Eulerian family of $\carr_2$.
Observe that in this example only faces in the first quadrant have non-zero coefficients.
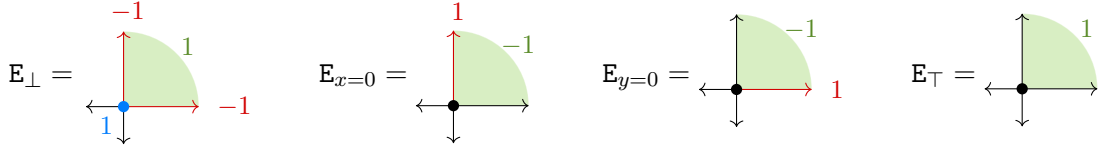
\[
\B{E}_\perp = 
\begin{gathered}
\begin{tikzpicture}
\filldraw [fill=green1, draw=white, opacity = .3] (0,0) -- (1,0) arc (0:90:1) -- cycle;
\draw [<->] (-.5,0) -- (0,0) node [blue1, below left] {\small~$1$} -- (0,-.5);
\draw [<->,red1] (1,0) node [right] {\small~$-1$} -- (0,0) -- (0,1) node [above] {\small~$-1$};
\node [circle, inner sep = 1.5pt, fill = blue1] at (0,0) {};
\node [color=green1!70!black] at (.8,.8) {\small~$1$};
\end{tikzpicture}
\end{gathered}
\quad
\B{E}_{x=0} = 
\begin{gathered}
\begin{tikzpicture}
\filldraw [fill=green1, draw=white, opacity = .3] (0,0) -- (1,0) arc (0:90:1) -- cycle;
\draw [<->] (-.5,0) -- (0,0) -- (0,-.5);
\draw [->] (0,0) -- (1,0);
\draw [->,red1] (0,0) -- (0,1) node [above] {\small~$1$};
\node [circle, inner sep = 1.5pt, fill = black] at (0,0) {};
\node [color=green1!70!black] at (.8,.8) {\small~$-1$};
\end{tikzpicture}
\end{gathered}
\quad
\B{E}_{y=0} = 
\begin{gathered}
\begin{tikzpicture}
\filldraw [fill=green1, draw=white, opacity = .3] (0,0) -- (1,0) arc (0:90:1) -- cycle;
\draw [<->] (-.5,0) -- (0,0) -- (0,-.5);
\draw [->,red1] (0,0) -- (1,0) node [right] {\small~$1$};
\draw [->] (0,0) -- (0,1);
\node [circle, inner sep = 1.5pt, fill = black] at (0,0) {};
\node [color=green1!70!black] at (.8,.8) {\small~$-1$};
\end{tikzpicture}
\end{gathered}
\quad
\B{E}_\top = 
\begin{gathered}
\begin{tikzpicture}
\filldraw [fill=green1, draw=white, opacity = .3] (0,0) -- (1,0) arc (0:90:1) -- cycle;
\draw [<->] (-.5,0) -- (0,0) -- (0,-.5);
\draw [<->] (0,1) -- (0,0) -- (1,0);
\node [circle, inner sep = 1.5pt, fill = black] at (0,0) {};
\node [color=green1!70!black] at (.8,.8) {\small~$1$};
\end{tikzpicture}
\end{gathered}
\]
\end{example}

A characteristic element~$w$ of \emph{non-critical}\footnote{$t \in \RR$ is non-critical if it is not a root of $\chi(\arr^\rmX,t)$ for any flat $\rmX \in \fflat[\arr]$.} parameter~$t$ uniquely determines an Eulerian family~$\B{E} = \{\B{E}_\rmX\}_\rmX$, which satisfies
\begin{equation}\label{eq:char_eul_idempotent}
w = \sum_\rmX t^{\dim(\rmX)} \rm E_\rmX.
\end{equation}
This is a consequence of~\citep[Propositions 11.9, 12.59]{am17}.
It follows that the action of such a characteristic elements~$w$ on any~$\tits[\arr]$-module~$M$ is diagonalizable.

Let~$M$ be a (right)~$\tits[\arr]$-module,~$w \in \tits[\arr]$ be a characteristic element of non-critical parameter~$t$ and~$\{\B{E}_\rmX\}_\rmX$ be the corresponding Eulerian family. Then, we have a decomposition
\begin{equation}\label{eq:M-decomposition}
M = \bigoplus_{\rmX} M \cdot \B{E}_\rmX
\end{equation}
of vector spaces. Expression~\eqref{eq:char_eul_idempotent} shows that~$w$ acts on~$M \cdot \B{E}_\rmX$ by multiplication by~$t^{\dim(\rmX)}$. We define
\[
\eta_\rmX(M) := \dim_\RR(M \cdot \B{E}_\rmX).
\]
The kernel of the support map $\supp$ is precisely the radical of $\tits[\arr]$ \citep[Proposition 9.22]{am17}. Consequently, the character~$\chi_M: \tits[\arr] \rightarrow \RR$ of~$M$ factors through~$\RR \fflat[\arr]$:
\[
\begin{tikzcd}[row sep=small]
{\tits[\arr]} \arrow[rd, "\chi_M"] \arrow[d, "\supp"'] & \\
{\RR\fflat[\arr]} \arrow[r, "\overline{\chi}_M"', dashed] & \RR
\end{tikzcd}
\]
Thus,~$\eta_\rmX(M) = \dim_\RR(M \cdot \B{E}_\rmX) = \chi_M(\B{E}_\rmX) = \opp{\chi}_M(\B{Q}_\rmX)$ is independent of the characteristic element~$w$. Furthermore, using relations~\eqref{eq:H_and_Q} and the linearity of $\opp{\chi}_M$ we deduce
\begin{multline}\label{eq:compute_eta}
\eta_\rmX(M) 
= \opp{\chi}_M(\B{Q}_\rmX) 
= \sum_{\rmY \geq \rmX} \mu(\rmX,\rmY) \opp{\chi}_M(\B{H}_\rmY)\\
= \sum_{\rmY \geq \rmX} \mu(\rmX,\rmY) \chi_M(\B{H}_{F_\rmY})
= \sum_{\rmY \geq \rmX} \mu(\rmX,\rmY) \dim_\RR(M \cdot \B{H}_{F_\rmY}),
\end{multline} 
where~$F_\rmY \in \face[\arr]$ is such that~$\supp(F_\rmY) = \rmY$. The last equality follows since~$\B{H}_{F_\rmY}$ is an idempotent element, and thus~$\chi_M(\B{H}_{F_\rmY}) = \dim_\RR(M \cdot \B{H}_{F_\rmY})$. 

Moreover, the number of composition factors~$M_{i+1}/M_i$ isomorphic to the simple module indexed by~$\rmX$ in a composition series
$
0 \subset M_1 \subset M_2 \subset \dots \subset M_k = M
$
of~$M$ is precisely~$\eta_\rmX(M)$.
See~\citep[Section 9.5 and Theorem D.37]{am17} for details.

\section{The polytope algebra as a module}\label{s:module}

Fix a hyperplane arrangement~$\arr$ in~$V$.
Take a normal vector~$v_\rmH$ for each hyperplane $\rmH \in \arr$, and consider the zonotope (Minkowski sum of segments):
\begin{equation}\label{eq:zon_A}
\frakz = \sum_{\rmH} \Conv\{ {\bf 0} , v_\rmH\}.
\end{equation}
Its normal fan~$\face_\frakz$ coincides with the collection of faces~$\face[\arr]$ of the arrangement $\arr$.
We say that a polytope~$\q$ is a \textbf{generalized zonotope} of~$\arr$ it is a deformation of~$\frakz$.

We now consider the algebra~$\Pi(\frakz)$ introduced in Section~\ref{ss:subalgebra_polytope}.
It is generated by the classes of generalized zonotopes of~$\arr$.
It only depends on the arrangement~$\arr$ and not on the particular choice of normal vectors~$v_\rmH$.
We start with a simple yet interesting result.
Recall the morphism $\psi_\f : \Pi(\frakz) \rightarrow \Pi(\f)$ defined for every face $\f \leq \frakz$ in~\eqref{eq:morphism_to_face}, it sends the class $[\q]$ of a generalized zonotope of $\arr$ to $[\q_v]$ where $v \in V$ is any vector such that $\f = \frakz_v$.

\begin{proposition}\label{p:zonotope_surjective}
Let~$\frakz$ be a zonotope, and~$\f$ a face of~$\frakz$. Then, the morphism~$\psi_\f$ is surjective.
\end{proposition}

\begin{proof}
Let $\f$ be a face of $\frakz$, $F \in \face[\arr]$ be its normal cone, and $\rmX = \supp(F)$ be the flat orthogonal to $\f$.
It follows from~\eqref{eq:zon_A} that~$\f = \frakz_F$ is a translate of
\[
\frakz_\rmX := \sum_{\rmH \,:\, \rmH \supseteq \rmX} \Conv\{ {\bf 0} , v_\rmH\}.
\]
In particular,~$\f$ is a Minkowski summand of~$\frakz$.
Being a Minkowski summand is a transitive relation.
Hence, any generator $[\q]$ of~$\Pi(\f)$ is also in~$\Pi(\frakz)$.
That is,~$\Pi(\f)$ is a subalgebra of~$\Pi(\frakz)$.
Moreover, if~$v \in \relint(F)$ and~$\q$ is a Minkowski summand of~$\f$, then~$\q_v = \q$.
Therefore, the composition
$
\Pi(\f) \hookrightarrow \Pi(\frakz) \xrightarrow{\psi_\f} \Pi(\f)
$
is the identity map.
Consequently, the morphism $\psi_\f$ is surjective.
\end{proof}

\begin{remark}
Compare with Theorem~\ref{t:simple_surjective} and note that we do not assume the zonotope~$\frakz$ to be simple.
For an arbitrary polytope~$\p$ and a face~$\f$ of~$\p$, there is no natural morphism~$\Pi(\f) \rightarrow \Pi(\p)$, unlike in the previous case.
This is a particular property of zonotopes. Indeed, a polytope $\p$ is a zonotope if and only if every face $\f \leq \p$ is a Minkowski summand of $\p$, see~\citep[Proposition 2.2.14]{blswz} for a proof.
\end{remark}

Let~$F$ be a face of~$\arr$ and~$\f = \frakz_F$ the corresponding face of~$\frakz$.
We define right multiplication by the basis element~$\B{H}_F \in \tits[\arr]$ on~$\Pi(\frakz)$
as the projection $\psi_\f : \Pi(\frakz) \rightarrow \Pi(\f) \subseteq \Pi(\frakz)$.

\begin{theorem}\label{t:mod_str}
The algebra~$\Pi(\frakz)$ is a right~$\tits[\arr]$-module under the action above.
Explicitly, for a generator~$[\q]$ of~$\Pi(\frakz)$ and a basis element~$\B{H}_F$ of~$\tits[\arr]$,
\[
[\q] \cdot \B{H}_F := [\q_v], \qquad
\text{where }v \in \relint(F).
\]
Moreover, each graded component~$\Xi_r(\frakz)$ is a~$\tits[\arr]$-submodule and the action of basis elements~$\{\B{H}_F\}_F$ on~$\Pi(\frakz)$ is by \textup{(}graded\textup{)} algebra endomorphisms.
\end{theorem}

\begin{proof}
The zero vector belongs to the central face~$O$, so the action is clearly unital.
Associativity follows from the following fact about polytopes~\citep[Section 3.1.5]{grunbaum03convex}.
If~$\q \subseteq V$ is a polytope and~$v,w \in V$, then 
$
(\q_v)_w = \q_{v + \lambda w}
$
for any small enough~$\lambda > 0$.
Similarly, the definition of the Tits product is such that if~$v \in \relint(F)$ and~$w \in \relint(G)$, then~$v + \lambda w \in \relint(FG)$ for any small enough~$\lambda > 0$.
Hence,
\[
([\q] \cdot \B{H}_F) \cdot \B{H}_G = [(\q_v)_w] = [\q_{v + \lambda w}] = [\q] \cdot \B{H}_{FG}.
\]
It follows that this product gives~$\Pi(\frakz)$ the structure of a right~$\tits[\arr]$-module.

The second statement follows directly from Theorem~\ref{t:direction_morphism} and the characterization of the graded components~$\Xi_r$ in~\eqref{eq:Xi_dilation}.
Indeed, for any~$x \in \Xi_r(\frakz)$ and $\lambda > 0$,
\[
\delta_\lambda(x \cdot \B{H}_F)
= \delta_\lambda(x) \cdot \B{H}_F
= \lambda^r x \cdot \B{H}_F
= \lambda^r (x \cdot \B{H}_F),
\]
thus~$x \cdot \B{H}_F \in \Xi_r(\frakz)$.
\end{proof}

\subsection{Simultaneous diagonalization}

Let $\lambda > 0$ and let $w \in \tits[\arr]$ be a characteristic element of non-critical parameter $t$.
We know that the dilation morphism $\delta_\lambda$ and the action of $w$ are diagonalizable.
Moreover, since $\delta_\lambda$ and the action of $w$ commute, they are simultaneously diagonalizable.
A natural question is to determine the eigenvalues of $\delta_\lambda$ and of the action of $w$ in their simultaneous eigenspaces. 
We completely answer this question in the case of the Coxeter arrangements of type A and B in the next two sections.
The following result holds in the general case.

\begin{proposition}
Let $x \in \Pi(\frakz)$ be a \textup{(}nonzero\textup{)} simultaneous eigenvector for $\delta_\lambda$ and $w$ with eigenvalues $\lambda^r$ and $t^k$, respectively.
Then, $r + k \leq d$.
\end{proposition}

\begin{proof}
Let $\{\B{E}_\rmX\}_\rmX$ be the Eulerian family associated to $w$.
Using the characterization of the graded components $\Xi_r$ as the eigenspaces of $\delta_\lambda$, and the decomposition of $\tits[\arr]$-modules in \eqref{eq:M-decomposition}, we deduce that the common eigenspace for $\delta_\lambda$ and $w$ with the given eigenvalues is $\bigoplus_\rmX \Xi_r(\frakz) \cdot \B{E}_\rmX$, where the sum is over all $k$-dimensional flats of $\arr$. 
Without loss of generality, we assume that $x \in \Xi_r(\frakz) \cdot \B{E}_\rmX$ for a single $k$-dimensional flat $\rmX$.

Proposition~\ref{p:zonotope_surjective} implies that $\Xi_r(\frakz) \cdot \B{H}_F = \Xi_r(\frakz_\rmY)$, where~$F \in \face[\arr]$ is any face of support~$\rmY$.
Hence, formula~\eqref{eq:compute_eta} yields
\begin{equation}\label{eq:eta_general_arr}
\eta_\rmX(\Xi_r(\frakz)) = \sum_{\rmY \geq \rmX} \mu(\rmX,\rmY) \dim_\RR(\Xi_r(\frakz_\rmY)).
\end{equation}
If $k > d-r$, then $\dim(\frakz_\rmY) = d - \dim(\rmY) \leq d - k < r$ and $\dim_\RR(\Xi_r(\frakz_\rmY)) = 0$ for any $\rmY$ in the sum.
So, in this case, $\dim_\RR(\Xi_r(\frakz) \cdot \B{E}_\rmX) = \eta_\rmX(\Xi_r(\frakz)) = 0$.
This contradicts that $x \in \Xi_r(\frakz) \cdot \B{E}_\rmX$ is a nonzero element.
Therefore, $r + k \leq d$.
\end{proof}

If in addition~$\arr$ is a simplicial arrangement, like in the case of reflection arrangements, then $\frakz$ and each of its faces are simple polytopes.
In that case, Theorem~\ref{t:simple_h_dimension} allows us to replace $\dim_\RR(\Xi_r(\frakz_\rmY))$ by $h_r(\frakz_\rmY)$ in expression \eqref{eq:eta_general_arr}.
Multiplying by $z^r$ and taking the sum over all values of $r$, we obtain
\begin{equation}\label{eq:eta_simplicial_arr}
\sum_r \eta_\rmX(\Xi_r(\frakz)) z^r = \sum_{\rmY \geq \rmX} \mu(\rmX,\rmY) h(\frakz_\rmY,z).
\end{equation}

\subsection{First example: the cube and the coordinate arrangement}

Let~$\carr_d$ be the coordinate arrangement in~$\RR^d$.
It consists of the~$d$ coordinate hyperplanes~$x_i = 0$ for~$i=1,\dots,d$.
We identify the lattice of flats~$\fflat[\carr_d]$ with the (opposite) boolean lattice~$2^{[d]}$ in the following manner:
\[
S \subseteq [d] \longleftrightarrow \rmX_S := \bigcap_{i \in S} \{x \,:\, x_i = 0\}.
\]
Observe that~$\rmX_S \leq \rmX_T$ if and only if~$T \subseteq S$, and in this case $\mu_\fflat(\rmX_S,\rmX_T) = \mu_{2^{[d]}}(T,S) = (-1)^{|S \setminus T|}$.

The~$d$-cube~$\frakc_d = [0,1]^d$ a zonotope of~$\carr_d$.
It is the Minkowski sum of the~$d$ line segments $\frakl_i := \Conv\{{\bf 0},e_i\}$ for~$i=1,\dots,d$.
It is a simple polytope with~$h$-polynomial~$h(\frakc_d,z) = (1+z)^d$.
Furthermore, for any~$S \subseteq [d]$ we have
\[
(\frakc_d)_{\rmX_S} = \sum_{i \in S} \frakl_i \cong \frakc_{|S|}.
\]

Let us consider the right~$\tits[\carr_d]$-module~$\Pi(\frakc_d)$.
For a flat~$\rmX_S$, formula~\eqref{eq:eta_simplicial_arr} yields
\[
\sum_r \eta_{\rmX_S}(\Xi_r(\frakc_d)) z^r 
=  \sum_{T \subseteq S} \mu(\rmX_S,\rmX_T) h(\frakc_{|T|},z)
=  \sum_{T \subseteq S} (-1)^{|S \setminus T|} (1+z)^{|T|}
= z^{|S|}.
\]
Hence,
\[
\eta_{\rmX_S}(\Xi_r(\frakc_d)) = \begin{cases}
1 & \text{if } |S| = r,\\
0 & \text{otherwise}.
\end{cases}
\]
In particular, a series decomposition of~$\Xi_r(\frakc_d)$ contains exactly one copy of the simple module indexed by~$\rmX_S$ for every~$S \in {[d] \choose r}$.

Let us now consider the characteristic element~$\gamma_t \in \tits[\carr_d]$ introduced in~\citep[Section 5.3]{abm19} for $t \neq 1$.
It is defined by
\[
\gamma_t = \sum_F \gamma_t^F \B{H}_F,
\quad \text{where} \quad
\gamma_t^F = \begin{cases}
(t-1)^{\dim(F)} & \text{ if~$F$ lies in the first orthant,} \\
0& \text{ otherwise.}
\end{cases}
\]
For each $S \subseteq [d]$, let $F_S$ be the intersection of the first orthant with~$\rmX_S$, it is a face of $\carr_d$.
We have~$\supp(F_S) = \rmX_S$ and~$T \subseteq S$ if and only if~$F_S \leq F_T$.
A simple computation shows that the Eulerian family corresponding to the characteristic element $\gamma_t$ is determined by
\[
\B{E}_{\rmX_S} = \sum_{T \subseteq S} (-1)^{|S \setminus T|} \B{H}_{F_T}.
\]
In dimension 2, this is the Eulerian family in Example \ref{ex:Eul_C2}.
For each~$S \subseteq [d]$, define
\[
y_S = \prod_{i \in S} \log[\frakl_i] \in \Pi(\frakc_d).
\]
Example~\ref{ex:prod_intervals} shows that~$y_S$ is a nonzero element of~$\Pi(\frakc_d)$.

We claim that~$\{y_S\}_{S\subseteq [d]}$ is a basis of simultaneous eigenvectors of~$\Pi(\frakc_d)$.
Explicitly,~$y_S$ is an eigenvector for the action of~$\gamma_t$ of eigenvalue~$t^{d-|S|}$, and for the action of~$\delta_\lambda$ of eigenvalue~$\lambda^{|S|}$ ($\lambda > 0$).
The second statement is clear, since~$\log[\frakl_i] \in \Xi_1(\frakc_d)$.
Moreover, using that~$\log[\frakl_i] = [\frakl_i] - 1$, we have
\[
y_S = \prod_{i \in S} ([\frakl_i] - 1) = \sum_{T \subseteq S} (-1)^{|S \setminus T|} [\frakc_{\rmX_T}].
\]
On the other hand, observe that
\[
[\frakc_{\rmX_S}] \cdot \B{E}_{\rmX_S} = \sum_{T \subseteq S} (-1)^{|S \setminus T|} [\frakc_{\rmX_S}] \cdot \B{H}_{F_T} = \sum_{T \subseteq S} (-1)^{|S \setminus T|} [\frakc_{\rmX_T}] = y_S.
\]
Therefore,
\[
y_S \in \Xi_r(\frakc_d) \cap (\Pi(\frakc_d) \cdot \B{E}_{\rmX_S}) = \Xi_r(\frakc_d) \cdot \B{E}_{\rmX_S}.
\]
The claim follows since~$\dim(\rmX_S) = d - |S|$.

\subsection{The zonotope module of a product of arrangements}

The Cartesian product of two arrangements~$\arr$ in~$V$ and~$\arr'$ in~$W$ is the following collection of hyperplanes in~$V \oplus W$:
\[
\arr \times \arr' =
\{ \rmH \oplus W \,:\, \rmH \in \arr \} \cup
\{ V \oplus \rmH \,:\, \rmH \in \arr' \}.
\]
One can easily verify that~$\face[\arr \times \arr'] \cong \face[\arr] \times \face[\arr']$ as monoids.
Hence, ~$\tits[\arr \times \arr'] \cong \tits[\arr] \otimes \tits[\arr']$.
In fact, it is also true that
\[
\Pi(\frakz \times \frakz') \cong \Pi(\frakz) \otimes \Pi(\frakz'),
\]
where~$\frakz$ and~$\frakz'$ are zonotopes of~$\arr$ and~$\arr'$, respectively, and therefore~$\frakz \times \frakz'$ is a zonotope of~$\arr \times \arr'$.
Indeed, every generalized zonotope of~$\arr \times \arr'$ is the Cartesian product of generalized zonotopes of~$\arr$ and~$\arr'$.
The corresponding isomorphism is induced by
\[
\begin{array}{CCC}
\Pi(\frakz) \otimes \Pi(\frakz') & \rightarrow & \Pi(\frakz \times \frakz')\\
{[\p]} \otimes [\q] & \mapsto & [\p \times \q]
\end{array}
\]
The fact that this map is well-defined and a morphism of~$\tits[\arr \times \arr']$-modules follows from the ideas in Section~\ref{ss:mc}.

\section{The module of generalized permutahedra}\label{s:braid}

Generalized permutahedra are the deformations of the standard permutahedron $\pi_d \subseteq \RR^d$.
Edmonds first introduced them under a different name in~\citep{edmonds}, where he studied their relation to submodular functions and optimization.
For a thorough study of the combinatorics of these polytopes, see~\citep{aa17,postnikov09,prw08faces}.

In this section, we study the algebra $\Pi(\pi_d)$ of generalized permutahedra and its structure as a module over the Tits algebra of the braid arrangement $\arr_d$.
We begin with a brief review of the braid arrangement, its relation with the symmetric group, and some statistics on permutations.

\subsection{The braid arrangement}

The {braid arrangement}~$\arr_d$ in~$\RR^d$ consists of the 
diagonal 
hyperplanes $x_i=x_j$ for~$1\leq i<j\leq d$.
Its central face is the line perpendicular to the hyperplane~$x_1+\cdots+x_d=0$. Intersecting $\arr_d$ with this hyperplane and a sphere around the origin we obtain the \emph{Coxeter complex of type~$A_{d-1}$}. The pictures below show the cases~$d=3$ and~$4$.
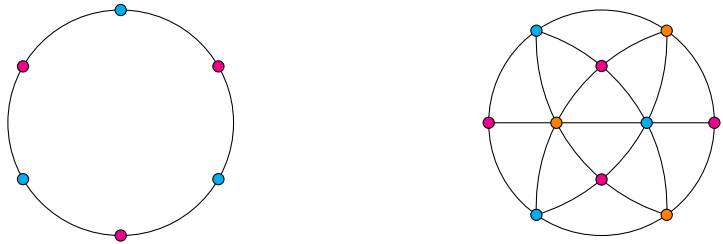
\[
\begin{gathered}
\begin{tikzpicture}[scale=.6]
\newdimen\R
\R=2.5cm
\draw (0,0) circle (\R);
\node [circle, inner sep=1.5pt,fill=magenta,draw] at (30:\R) {};
\node [circle, inner sep=1.5pt,fill=cyan,draw] at (90:\R) {};
\node [circle, inner sep=1.5pt,fill=magenta,draw] at (150:\R) {};
\node [circle, inner sep=1.5pt,fill=cyan,draw] at (210:\R) {};
\node [circle, inner sep=1.5pt,fill=magenta,draw] at (270:\R) {};
\node [circle, inner sep=1.5pt,fill=cyan,draw] at (330:\R) {};
\end{tikzpicture}
\end{gathered}
\hspace*{.2\linewidth}
\begin{gathered}
\begin{tikzpicture}[scale=.6]
\newdimen\R
\R=2.5cm
\draw (0,0) circle (\R);
\coordinate (P) at (-1,0); \coordinate (Q) at (1,0); \coordinate (A1) at (54.735:\R);
\coordinate (B1) at (180+54.735:\R);
\arcThroughThreePoints{A1}{P}{B1};
\arcThroughThreePoints{B1}{Q}{A1};
\coordinate (A2) at (180-54.735:\R);
\coordinate (B2) at (-54.735:\R);
\arcThroughThreePoints{A2}{P}{B2};
\arcThroughThreePoints{B2}{Q}{A2};
\coordinate (A3) at (0:\R);
\coordinate (B3) at (180:\R);
\draw (A3) -- (B3);
\path[name intersections={of=kamaanA1 and kamaanB2, by=a1b2}];
\path[name intersections={of=kamaanB1 and kamaanA2, by=b1a2}];
\node [circle, inner sep=1.5pt,fill=magenta,draw] at (a1b2) {};
\node [circle, inner sep=1.5pt,fill=magenta,draw] at (b1a2) {};
\node [circle, inner sep=1.5pt,fill=orange,draw] at (P) {};
\node [circle, inner sep=1.5pt,fill=cyan,draw] at (Q) {};
\node [circle, inner sep=1.5pt,fill=orange,draw] at (A1) {};
\node [circle, inner sep=1.5pt,fill=cyan,draw] at (B1) {};
\node [circle, inner sep=1.5pt,fill=cyan,draw] at (A2) {};
\node [circle, inner sep=1.5pt,fill=orange,draw] at (B2) {};
\node [circle, inner sep=1.5pt,fill=magenta,draw] at (A3) {};
\node [circle, inner sep=1.5pt,fill=magenta,draw] at (B3) {};
\end{tikzpicture}
\end{gathered}
\]
Flats and faces of $\arr_d$ are in one-to-one correspondence with set partitions and set compositions of~$[d] : = \{1,2,\dots,d\}$, respectively. 
We proceed to review this correspondence.

A \textbf{weak set partition} of a finite set~$I$ is a collection~$\rmX = \{S_1,\dots,S_k\}$ of pairwise disjoint subsets $S_i \subseteq I$ such that $I = S_1 \cup \dots \cup S_k$.
The subsets~$S_i$ are the \textbf{blocks} of~$\rmX$.
A \textbf{set partition} is a weak set partition with no empty blocks.
We write~$\rmX \vdash I$ to denote that~$\rmX$ is a set partition of~$I$.
Given a partition~$\rmX \vdash [d]$, the corresponding flat of $\arr_d$ is the intersection of the hyperplanes~$x_a = x_b$ for all~$a,b$ that belong to the same block of $\rmX$, as illustrated in the following example for $d=8$:
\[
x_1 = x_3, \quad x_2 = x_5 = x_6 = x_8
\qquad \longleftrightarrow \qquad
\{13,2568,4,7\},
\]
where we write $13$ to abbreviate the set $\{1,3\}$, $2568$ to abbreviate the set $\{2,5,6,8\}$ and so on.
We use $\rmX$ to denote both a flat of $\arr_d$ and the corresponding set partition of $[d]$.
Observe that $\dim(\rmX)$ is precisely the number of blocks of $\rmX$ as a partition.
The partial order relation of $\fflat[\arr_d]$ becomes the ordering by refinement of set partitions. That is, $\rmX \leq \rmY$ if the set partition $\rmX$ is \textbf{refined by}~$\rmY$. Recall that $\rmX$ is refined by $\rmY$ if every block of $\rmX$ is the union of some blocks in $\rmY$. For instance, $\{12345678\} \leq \{13,2568,4,7\} \leq \{1,28,3,4,56,7\}$.

If~$S \subseteq I$ is a union of blocks of a partition~$\rmX \vdash I$, we let~$\rmX|_S \vdash S$ denote the partition of~$S$ formed by the blocks of $\rmX$ whose union is $S$.
Let~$\rmX = \{S_1,\dots,S_k\} \vdash [d]$ be a partition.
Then, the choice of $\rmY \geq \rmX$ is equivalent to the choice of partitions $\rmY|_{S_i} \vdash S_i$ for each block of $\rmX$.
With $\rmX$ and $\rmY$ as above, the Möbius function of~$\fflat[\arr_d]$ is determined by
\begin{equation}\label{eq:mu_prod_A}
\mu(\perp,\rmX) = (-1)^{k-1} (k-1)!
\qqand
\mu(\rmX,\rmY) = \mu(\perp,\rmY|_{S_1}) \dots \mu(\perp,\rmY|_{S_k}),
\end{equation}
where in each factor, $\perp$ denotes the minimum partition of $S_i$.

A \textbf{set composition} of~$I$ is an ordered set partition~$F = (S_1,\dots,S_k)$.
We write~$F \vDash I$ to denote that~$F$ is a composition of~$I$, and let~$\supp(F) \vdash I$ be the underlying (unordered) set partition.
Given a set composition~$F \vDash [d]$, the corresponding face of $\arr_d$ is obtained by intersecting the hyperplanes $x_a = x_b$ whenever $a,b$ are in the same block of $F$, and the halfspaces~$x_a \geq x_b$ whenever the block containing~$a$ precedes the block containing~$b$. For example,
\[
x_1 = x_3 \geq x_4 \geq x_2 = x_5 = x_6 = x_8 \geq x_7
\qquad \longleftrightarrow \qquad
(13,4,2568,7).
\]

\subsection{The symmetric group and the Eulerian polynomial}

The symmetric group~$\frakS_d$ is the group of \emph{permutations}~$\sigma : [d] \rightarrow [d]$ under composition.
It is the Coxeter group corresponding to the braid arrangement. It acts on~$\RR^d$ by permuting coordinates:
\[
\sigma(x_1,x_2,\dots,x_d) = (x_{\sigma(1)},x_{\sigma(2)},\dots,x_{\sigma(d)}).
\]
For a permutation $\sigma \in \frakS_d$, we let~$\supp(\sigma)$ denote the \textbf{subspace of points fixed by the action} of~$\sigma$; it is a flat of~$\arr_d$.
In view of the identification between flats of $\arr_d$ and partitions of $[d]$,~$\supp(\sigma)$ can equivalently be defined as the partition of~$[d]$ into the disjoint cycles of~$\sigma$.
For example, if in cycle notation $\sigma = ({\bf 1}3)({\bf 2}6{\bf 5}8)(4)(7)$, then $\supp(\sigma) = \{13,2568,4,7\}$.

Recall that $i \in [d-1]$ is a \textbf{descent} of $\sigma \in \frakS_d$ if $\sigma(i) > \sigma(i+1)$, 
and $i \in [d-1]$ is an \textbf{excedance} of $\sigma \in \frakS_d$ if $\sigma(i) > i$. Let $\des(\sigma)$ and $\exc(\sigma)$ denote the number of descents and excedances of $\sigma$, respectively. In the example above, $1,2,5$ (in bold) are the excedances of $\sigma$, and we have $\exc(\sigma) = 3$. We can similarly define descents and excedances for permutations of any set $S$ with a total order~$\prec$, we denote the corresponding statistics by $\des_\prec$ and $\exc_\prec$.

It is a classical result that descents and excedances are \emph{equidistributed} in~$\frakS_d$.
That is,
\[
A_{d,k} := \big|\{ \sigma \in \frakS_d \,:\, \des(\sigma) = k \}\big|
=
\big|\{ \sigma \in \frakS_d \,:\, \exc(\sigma) = k \}\big|,
\]
for all possible values of~$k$.
Foata's \emph{fundamental transformation} provides a simple proof of this result.
The numbers $A_{d,k}$ are the classical \textbf{Eulerian numbers} (OEIS: \href{https://oeis.org/A008292}{A008292}).
The \textbf{Eulerian polynomial}~$A_d(z)$ is:
\[
A_d(z) := \sum_{k=0}^{d-1} A_{d,k} z^k = \sum_{\sigma \in \frakS_d} z^{\exc(\sigma)}.
\]
The exponential generating function for these polynomials was originally given by Euler himself:
\begin{equation}\label{eq:generating_eulerian}
A(z,x) = 1 + \sum_{d \geq 1} A_d(z) \dfrac{x^d}{d!} = \dfrac{z-1}{z-e^{x(z-1)}}.
\end{equation}
See~\citep[Section 3]{foata10eulerian} for a derivation of this formula.

Let~$\frakC(S)$ the collection of \textbf{cyclic permutations} on a finite set~$S$, and $\frakC(d) = \frakC([d])$.
Given a permutation $\sigma \in \frakS_d$ and a block $S \in \supp(\sigma)$, the restriction $\sigma|_S$ of $\sigma$ to $S$ is a cyclic permutation. For example, with $\sigma$ as before and $S = \{2,5,6\} \in \supp(\sigma)$, we have $\sigma|_S = (265) \in \frakC(\{2,5,6\})$. A very simple but important observation is that the number of excedances of $\sigma$ can be computed by adding up the excedances in each cycle in its cycle decomposition. That is, $\exc(\sigma) = \sum_{S \in \supp(\sigma)} \exc(\sigma|_S)$. The number of excedances in each cycle $\sigma|_S$ is computed with respect to the natural order in $S \subseteq [d]$.

\subsection{The module Generalized permutahedra}

The permutahedron~$\pi_d \subseteq \RR^d$ is the convex hull of the~$\frakS_d$-orbit the point~$(1,2,\dots,d)$.
It is a zonotope of the braid arrangement~$\arr_d$ and has dimension~$d-1$.
Deformations of~$\pi_d$ are called \textbf{generalized permutahedra}.
We consider the module~$\Pi(\pi_d)$ as in Section~\ref{s:module}. The main goal of this section will be to prove the following result.

\begin{theorem}\label{t:dims_simul_e-spaces_A}
For any flat~$\rmX \in \fflat[\arr_d]$ and~$r=0,1,\dots,d-1$,
\[
\eta_\rmX(\Xi_r(\pi_d)) = \big| \{ \sigma \in \frakS_d \,:\, \supp(\sigma) = \rmX,\ \exc(\sigma) = r \}\big|.
\]
\end{theorem}

\begin{figure}[ht]
\begin{center}
\[
\begin{gathered}
\begin{tikzpicture}
\newdimen\R
\R=1.3cm
\draw [very thick, fill=gray!30!white] (0:\R) -- (60:\R) -- (120:\R) -- (180:\R) -- (240:\R) -- (300:\R) -- cycle;
\node [right] at (0:\R) {\scriptsize$(1,2,3)$};
\node [] at (60:1.15\R) {\scriptsize$(1,3,2)$};
\node [] at (120:1.15\R) {\scriptsize$(2,3,1)$};
\node [left] at (180:\R) {\scriptsize$(3,2,1)$};
\node [] at (240:1.15\R) {\scriptsize$(3,1,2)$};
\node [] at (300:1.15\R) {\scriptsize$(2,1,3)$};
\end{tikzpicture}
\end{gathered}
\hspace*{.2\linewidth}
\begin{gathered}
\tdplotsetmaincoords{60}{120}
\begin{tikzpicture}[tdplot_main_coords,scale=.7]
\coordinate (A) at (-1.414, 0.816, 2.309);
\coordinate (B) at (-1.414, 1.632, 1.154);
\coordinate (C) at (-0.707, -0.408, 2.309);
\coordinate (D) at (-0.707, 1.224, 0);
\coordinate (E) at (-0.707, 1.224, 3.464);
\coordinate (F) at (-0.707, 2.857, 1.154);
\coordinate (G) at (0, -0.816, 1.154);
\coordinate (H) at (0, 0, 0);
\coordinate (I) at (0, 0, 3.464);
\coordinate (J) at (0, 2.449, 0);
\coordinate (K) at (0, 2.449, 3.464);
\coordinate (L) at (0, 3.265, 2.309);
\coordinate (M) at (1.414, -0.816, 1.154);
\coordinate (N) at (1.414, 0, 0);
\coordinate (O) at (1.414, 0, 3.464);
\coordinate (P) at (1.414, 2.449, 0);
\coordinate (Q) at (1.414, 2.449, 3.464);
\coordinate (R) at (1.414, 3.265, 2.309);
\coordinate (S) at (2.121, -0.408, 2.309);
\coordinate (T) at (2.121, 1.224, 0);
\coordinate (U) at (2.121, 1.224, 3.464);
\coordinate (V) at (2.121, 2.857, 1.154);
\coordinate (W) at (2.828, 0.816, 2.309);
\coordinate (X) at (2.828, 1.632, 1.154);
\draw [dashed] (A) -- (B);
\draw [dashed] (A) -- (C);
\draw [dashed] (A) -- (E);
\draw [dashed] (B) -- (D);
\draw [dashed] (B) -- (F);
\draw [dashed] (C) -- (G);
\draw [dashed] (C) -- (I);
\draw [dashed] (D) -- (H);
\draw [dashed] (D) -- (J);
\draw [very thick] (E) -- (I);
\draw [very thick] (E) -- (K);
\draw [very thick] (F) -- (J);
\draw [very thick] (F) -- (L);
\draw [dashed] (G) -- (H);
\draw [dashed] (G) -- (M);
\draw [dashed] (H) -- (N);
\draw [very thick] (I) -- (O);
\draw [very thick] (J) -- (P);
\draw [very thick] (K) -- (L);
\draw [very thick] (L) -- (R);
\draw [very thick] (M) -- (N);
\draw [very thick] (M) -- (S);
\draw [very thick] (M) -- (S);
\draw [very thick] (N) -- (T);
\draw [very thick] (O) -- (S);
\draw [very thick] (O) -- (U);
\draw [very thick] (P) -- (T);
\draw [very thick] (P) -- (V);
\draw [very thick] (Q) -- (R);
\draw [very thick] (Q) -- (U);
\draw [very thick] (R) -- (V);
\draw [very thick] (S) -- (W);
\draw [very thick] (T) -- (X);
\draw [very thick] (U) -- (W);
\draw [very thick] (V) -- (X);
\draw [very thick] (W) -- (X);
\draw [very thick] (K) -- (Q);
\end{tikzpicture}
\end{gathered}
\]
\caption{The permutahedron in $\RR^3$ and $\RR^4$.}
\end{center}
\end{figure}

The relation between~$\Pi(\pi_d)$ and statistics on $\frakS_d$ is via the $h$-polynomial of $\pi_d$.
Brenti~\citep[Theorem 2.3]{brenti94cox-eul} showed that $h(\pi_d,z) = A_d(z)$. Moreover, for a flat/partition~$\rmX =\{S_1,\dots,S_k\}$ of~$\arr_d$, the face $(\pi_d)_\rmX$ is a translate of $\pi_{|S_1|} \times \dots \times \pi_{|S_k|}$, a product of lower-dimensional permutahedra. Thus,
\begin{equation}\label{eq:prod_perm}
h((\pi_d)_\rmX,z) = A_{|S_1|}(z) \cdot \ldots \cdot A_{|S_k|}(z).
\end{equation}

Lemma~\ref{l:cyc_exc} below is an essential ingredient in the proof of Theorem~\ref{t:dims_simul_e-spaces_A}. Its proof uses the Compositional Formula, for which a type B analog is proved in Proposition~\ref{p:typeB-compositional}.

\begin{theorem}[{The Compositional Formula \citep[Theorem 5.1.4]{stanley99EC2}}]\label{t:compositional_formula}
Let
\[
g(x) = 1 + \sum_{d \geq 1} g_d \dfrac{x^d}{d!} \qquad
a(x) = \sum_{d \geq 1} a_d \dfrac{x^d}{d!}.
\]
If
\[
h(x) = 1 + \sum_{d \geq 1} h_d \dfrac{x^d}{d!!}
\qquad \text{where} \qquad
h_d = \sum_{\{ S_1,S_2,\dots , S_k\} \vdash [d]} g_k a_{|S_1|} \dots a_{|S_k|},
\]
then
\[
h(x) = g(a(x)).
\]
\end{theorem}

\begin{lemma}\label{l:cyc_exc}
For every $d \geq 1$,
\begin{equation}\label{eq:lem_cyc_exc}
\sum_{\rmX = \{S_1,\dots,S_k\} \vdash [d]} \mu(\perp,\rmX) A_{|S_1|}(z) \cdot \ldots \cdot A_{|S_k|}(z) = \sum_{\sigma \in \frakC(d)} z^{\exc(\sigma)}.
\end{equation}
\end{lemma}

\begin{proof}
We will show that the exponential generating function of both sides of~\eqref{eq:lem_cyc_exc} are equal to~$\log (A(z,x))$, where~$A(z,x)$ is the generating function for the Eulerian polynomials in~\eqref{eq:generating_eulerian}.

First, recall that if $\rmX = \{S_1,\dots,S_k\}$, then~$\mu(\perp,\rmX) = (-1)^{k-1}(k-1)!$.
Thus, a direct application of the Compositional Formula shows that the exponential generating function of the LHS of~\eqref{eq:lem_cyc_exc} is the composition of
\[
\sum_{d \geq 1} (-1)^{d-1}(d-1)! \dfrac{x^d}{d!}
= \log(1+x)
\qquad \text{with} \qquad
\sum_{d \geq 1} A_d(z) \dfrac{x^d}{d!}
= A(z,x) - 1,
\]
which is precisely~$\log (A(z,x))$.

On the other hand, grouping permutations with the same underlying partition~$\supp(\sigma)$, we obtain
\[
A(z,x)
= 1 + \sum_{d \geq 1} \Big( \sum_{\sigma \in \frakS_d} z^{\exc(\sigma)} \Big) \dfrac{x^d}{d!}
= 1 + \sum_{d \geq 1} \Big( \sum_{\rmX \vdash [d]} \Big( \sum_{\substack{\sigma \in \frakS_d \\ \supp(\sigma) = \rmX}} z^{\exc(\sigma)} \Big) \Big) \dfrac{x^d}{d!}.
\]
Since a permutation $\sigma$ with $\supp(\sigma) = \rmX$ is the product of cyclic permutations $\sigma_S \in \frakC(S)$ for each block $S \in \rmX$, and in this case $\exc(\sigma) = \sum_{S \in \rmX} \exc(\sigma_S)$,
\begin{equation}\label{eq:exc_cycles}
\sum_{\substack{\sigma \in \frakS_d \\ \supp(\sigma) = \rmX}} z^{\exc(\sigma)}
= \prod_{S \in \rmX} \Big( \sum_{\sigma_S \in \frakC(S)} z^{\exc(\sigma_S)} \Big).
\end{equation}
Thus, the Exponential Formula~\citep[Corollary 5.1.6]{stanley99EC2} implies that
\[
A(z,x) = \exp\bigg( \sum_{d \geq 1} \Big( \sum_{\sigma \in \frakC(d)} z^{\exc(\sigma)} \Big) \dfrac{x^d}{d!} \bigg).
\]
Taking logarithms on both sides yields the result.
\end{proof}

A small modification in the proof of the previous Lemma immediately gives the following result, which was first discovered by Brenti.

\begin{corollary}[{\citep[Proposition 7.3]{brenti00qeulerian}}]
The following identity holds
\[
1 + \sum_{d \geq 1} \Big( \sum_{\sigma \in \frakS_d} t^{|\supp(\sigma)|} z^{\exc(\sigma)} \Big) \dfrac{x^n}{n!}
= \exp(t\log (A(z,x)))
= \left( \dfrac{z-1}{z-e^{x(z-1)}} \right)^t.
\]
\end{corollary}

An analogous formula for the type B Coxeter group is described in Proposition~\ref{p:bivariate-generating-B}.
We are now ready to prove the main result of this section.

\begin{proof}[Proof of Theorem~\ref{t:dims_simul_e-spaces_A}]
We will compute the values $\eta_\rmX(\Xi_r(\pi_d))$ using formula \eqref{eq:eta_simplicial_arr},
which in this case reads
\[
\sum_r \eta_\rmX(\Xi_r(\pi_d)) z^r
= \sum_{\rmY \,:\, \rmY \geq \rmX} \mu(\rmX,\rmY) h((\pi_d)_\rmY,z).
\]
Using~\eqref{eq:mu_prod_A} and~\eqref{eq:prod_perm}, we can rewrite the expression above as
\[
\sum_r \eta_\rmX(\Xi_r(\pi_d)) z^r
= \prod_{S \in \rmX} \Big( \sum_{\rmY=\{T_1,\dots,T_\ell\} \vdash S} \mu(\perp,\rmY) A_{|T_1|}(z) \cdot \ldots \cdot A_{|T_\ell|}(z) \Big).
\]
Now, an application of Lemma~\ref{l:cyc_exc} and relation~\eqref{eq:exc_cycles} gives
\[
\sum_r \eta_\rmX(\Xi_r(\pi_d)) z^r
= \prod_{S \in \rmX} \Big( \sum_{\sigma \in \frakC(S)} z^{\exc(\sigma)} \Big)
= \sum_{\substack{\sigma \in \frakS_d \\ \supp(\sigma) = \rmX}} z^{\exc(\sigma)}.
\]
Finally, taking the coefficient of~$z^r$ on both sides of the last equality yields the result.
\end{proof}

Adding over all flats with the same dimension in Theorem~\ref{t:dims_simul_e-spaces_A}, we conclude the following.

\begin{corollary}
Let~$w \in \tits[\arr_d]$ be a characteristic element of non-critical parameter~$t$ and $\lambda > 0$.
The dimension of the simultaneous eigenspace for $w$ and $\delta_\lambda$ with eigenvalues $t^k$ and $\lambda^r$ is
\[
\big| \{ \sigma \in \frakS_d \,:\, |\supp(\sigma)| = k,\ \exc(\sigma) = r \}\big|.
\]
\end{corollary}

\subsection{Simultaneous-eigenbasis for the Adams element}

Perhaps the most natural characteristic elements for the braid arrangement are the \textbf{Adams elements}, defined for any parameter $t$ by:
\[
\alpha_t = \sum_F {t \choose \dim(F)} \B{H}_F.
\]
It is invariant with respect to the action of~$\frakS_d$, and its action on $\tits[\arr_d]$-modules is closely related with the \emph{convolution powers of the identity map} of a Hopf monoid, see \citep[Section 14.4]{am13}.
The corresponding Eulerian idempotents are~\citep[Theorem 12.75]{am17}
\[
\B{E}_\rmX = \dfrac{1}{\dim(\rmX)!} \sum_{\supp(F) = \rmX} \sum_{G \geq F} \dfrac{(-1)^{\dim(G/F)}}{\deg(G/F)} \B{H}_G,
\]
where~$\dim(G/F) = \dim(G) - \dim(F)$ and,
$
\deg(G/F) = \prod_{S \in F} \big|G|_S\big|.
$

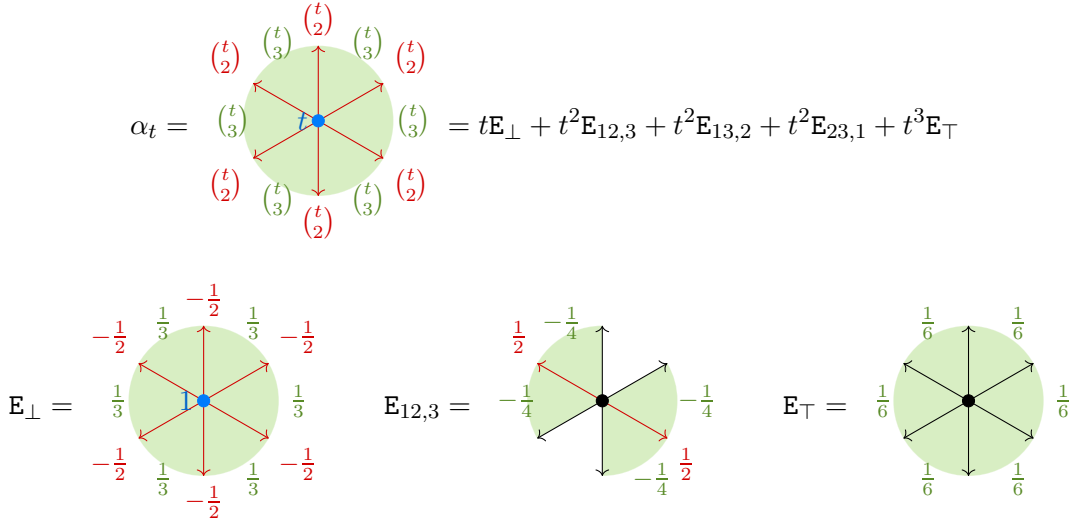
\begin{figure}[ht]
\[
\alpha_t = 
\begin{gathered}
\begin{tikzpicture}
\newdimen\R
\R=1cm \filldraw [fill=green1, draw=white, opacity = .3] (0:\R) arc (0:360:\R);
\draw [red1 , <->] (150:\R) node [above left] {\small${t \choose 2}$} -- (-30:\R) node [below right] {\small${t \choose 2}$};
\draw [red1 , <->] (210:\R) node [below left] {\small${t \choose 2}$} -- (30:\R) node [above right] {\small${t \choose 2}$};
\draw [red1 , <->] (-90:\R) node [below] {\small${t \choose 2}$} -- (0,0) node [blue1!80!black, left] {$t$} -- (90:\R) node [above] {\small${t \choose 2}$};
\draw [blue1, fill] (0,0) circle (.08);
\foreach \x in {0,60,120,180,240,300} {
\node [color=green1!70!black] at (\x:1.2\R) {\small~${t \choose 3}$};
}
\end{tikzpicture}
\end{gathered}
=
t \B{E}_\perp + 
t^2 \B{E}_{12,3} + t^2 \B{E}_{13,2} + t^2 \B{E}_{23,1} +
t^3 \B{E}_\top
\]
\[
\B{E}_\perp = 
\begin{gathered}
\begin{tikzpicture}
\newdimen\R
\R=1cm \filldraw [fill=green1, draw=white, opacity = .3] (0:\R) arc (0:360:\R);
\draw [red1 , <->] (150:\R) node [above left] {\small$-\tfrac{1}{2}$} -- (-30:\R) node [below right] {\small$-\tfrac{1}{2}$};
\draw [red1 , <->] (210:\R) node [below left] {\small$-\tfrac{1}{2}$} -- (30:\R) node [above right] {\small$-\tfrac{1}{2}$};
\draw [red1 , <->] (-90:\R) node [below] {\small$-\tfrac{1}{2}$} -- (0,0) node [blue1!80!black, left] {$1$} -- (90:\R) node [above] {\small$-\tfrac{1}{2}$};
\draw [blue1, fill] (0,0) circle (.08);
\foreach \x in {0,60,120,180,240,300} {
\node [color=green1!70!black] at (\x:1.2\R) {\small~$\tfrac{1}{3}$};
}
\end{tikzpicture}
\end{gathered}
\quad
\B{E}_{12,3} = 
\begin{gathered}
\begin{tikzpicture}
\newdimen\R
\R=1cm \filldraw [fill=green1, draw=white, opacity = .3] (0,0) -- (-90:\R) arc (-90:30:\R) -- cycle;
\filldraw [fill=green1, draw=white, opacity = .3] (0,0) -- (90:\R) arc (90:210:\R) -- cycle;
\draw [<->] (-90:\R) -- (90:\R);
\draw [red1 , <->] (150:\R) node [above left] {\small$\tfrac{1}{2}$} -- (-30:\R) node [below right] {\small$\tfrac{1}{2}$};
\draw [<->] (210:\R) -- (30:\R);
\draw [fill] (0,0) circle (.08);
\foreach \x in {0,120,180,300} {
\node [color=green1!70!black] at (\x:1.2\R) {\small~$-\tfrac{1}{4}$};
}
\end{tikzpicture}
\end{gathered}
\quad
\B{E}_{\top} = 
\begin{gathered}
\begin{tikzpicture}
\newdimen\R
\R=1cm \filldraw [fill=green1, draw=white, opacity = .3] (0:\R) arc (0:360:\R);
\draw [<->] (-90:\R) -- (90:\R);
\draw [<->] (150:\R) -- (-30:\R);
\draw [<->] (210:\R) -- (30:\R);
\draw [fill] (0,0) circle (.08);
\foreach \x in {0,60,120,180,240,300} {
\node [color=green1!70!black] at (\x:1.2\R) {\small~$\tfrac{1}{6}$};
}
\end{tikzpicture}
\end{gathered}
\]
\caption{The Adams element $\alpha_t$ and some of the associated Eulerian idempotents of the braid arrangement in $\RR^3$.}
\end{figure}

Theorem~\ref{t:dims_simul_e-spaces_A} suggest the existence of a natural basis for~$\Xi_r(\pi_d) \cdot \B{E}_\rmX$ indexed by permutations~$\sigma$ with~$r$ excedances and~$\supp(\sigma) = \rmX$.
In this section we will construct a candidate for such basis.

The standard simplex~$\Delta_{[d]} \subseteq \RR^d$ is the convex hull of the standard basis~$\{e_1,\dots,e_d\}$ of~$\RR^d$.
Similarly, for any nonempty subset~$S \subseteq [d]$, we let~$\Delta_S = \Conv\{e_i \,:\, i \in S\}$.

\begin{theorem}[{\citep{abd10volume,postnikov09}}]\label{t:typeA_gens}
Every generalized permutahedron~$\p$ can be written uniquely as a \emph{signed Minkowski sum} of the simplices $\{\Delta_S \,:\, S \subseteq [d] \}$.
That is, there is a unique choice of scalars $\{y_S\}_{S \subseteq [d]} \subseteq \RR$ such that
\[
\p + \sum_{y_S < 0} |y_S| \Delta_S = \sum_{y_S > 0} y_S \Delta_S.
\]
\end{theorem}

Up to translation, this is equivalent to the following identity in $\Xi_1(\pi_d)$:
\[
\log[\p]
= \sum_{S \subseteq [d]} y_S \log[\Delta_S].
\]
Given that~$\log[\Delta_S] = 0$ whenever~$S$ is a singleton, and that~$[\p] = [\q]$ if and only if~$\p$ is a translate of~$\q$,
we conclude that~$\{\log[\Delta_S] \,:\, S \subseteq [d],\, \big|S\big| \geq 2\}$ is a linear basis for~$\Xi_1(\pi_d)$ and therefore generate $\Pi(\pi_d)$ as an algebra.
This agrees with~$\dim_\RR(\Xi_1(\pi_d)) = h_1(\pi_d) = 2^d - d - 1$.

We will use a bijection between \emph{increasing rooted forests} on~$[d]$ and permutations in~$\frakS_d$.
An increasing rooted forest is a disjoint union of planar rooted trees where each child is larger than its parent and the children are in increasing order from left to right.
Given a rooted forest~$t$, the corresponding permutation~$\sigma(t)$ is read as follows.
Each connected component of~$t$ corresponds to a cycle of~$\sigma(t)$.
To form a cycle, traverse the corresponding tree counterclockwise and record a node the \emph{second} time you pass by it\footnote{A similar bijection is described by Peter Luschny in this \href{https://oeis.org/wiki/User:Peter_Luschny/PermutationTrees}{OEIS entry}.}.
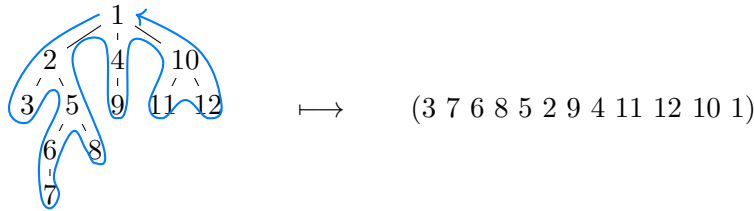
\[
\begin{gathered}
\begin{tikzpicture}[scale = .6]
\draw (-1.5,-4) -- (-1.5,-3) -- (-1,-2) -- (-.5,-3);
\draw (-2,-2)  -- (-1.5,-1) -- (-1,-2);
\draw (-1.5,-1) -- (0,0) -- (0,-2);
\draw (0,0) -- (1.5,-1) -- (1,-2);
\draw (1.5,-1) -- (2,-2);
\node [fill=white] at (0,0) {$1$};
\node [fill=white] at (-1.5,-1) {$2$};
\node [fill=white] at (0,-1) {$4$};
\node [fill=white] at (1.5,-1) {$10$};
\node [fill=white] at (-2,-2) {$3$};
\node [fill=white] at (-1,-2) {$5$};
\node [fill=white] at (0,-2) {$9$};
\node [fill=white] at (1,-2) {$11$};
\node [fill=white] at (2,-2) {$12$};
\node [fill=white] at (-1.5,-3) {$6$};
\node [fill=white] at (-.5,-3) {$8$};
\node [fill=white] at (-1.5,-4) {$7$};
\draw [thick, blue1, ->] plot[smooth, tension=.7] coordinates
{(-.4,0) (-1.7,-.8) (-2.4,-2) (-2,-2.4) (-1.5,-1.7) (-1.3,-2) (-1.7,-3) (-1.7,-4) (-1.5,-4.3) (-1.3,-4) (-1.4,-3.5) (-1,-2.6) (-.7,-3) (-.5,-3.3) (-.3,-3) (-1,-1) (-.3,-.6) (-.2,-2) (-0,-2.3) (.2,-2) (.3,-.6) (1,-1) (.7,-2) (1,-2.3) (1.5,-1.9) (2,-2.3) (2.3,-2) (1.8,-.8) (.4,0)};
\end{tikzpicture}
\end{gathered}
\qquad\longmapsto\qquad
(3\,\, 7\,\, 6\,\, 8\,\, 5\,\, 2\,\, 9\,\, 4\,\, 11\,\, 12\,\, 10\,\, 1)
\]
The inverse can be described inductively by writing each cycle with its minimum element in the last position, and using right to left minima.
We omit the details, but provide an example~$\sigma \mapsto t(\sigma)$ to illustrate the idea.
\begin{equation}\label{eq:permtree1}
(7\, 3\, 6\, 9\, 5\, 1)(4\, 10\, 8\, 2)
\longmapsto
\begin{gathered}
\begin{tikzpicture}[scale = .8]
\draw (0,0) -- (0.5,1) -- (1,0);
\draw (2,0) -- (2.5,1) -- (3,0);
\node [fill=white] at (0.5,1) {$1$};
\node [fill=white] at (-.1,0) {$(73)$};
\node [fill=white] at (1,0) {$(695)$};
\node [fill=white] at (2.5,1) {$2$};
\node [fill=white] at (2,0) {$4$};
\node [fill=white] at (3,0) {$(10\, 8)$};
\end{tikzpicture}
\end{gathered}
\longmapsto
\begin{gathered}
\begin{tikzpicture}[scale = .8]
\draw (0,0) -- (1,2) -- (2,0) -- (1.5,1) -- (1,0);
\draw (3,1) -- (3.5,2) -- (4,1) -- (4,0);
\node [fill=white] at (1,2) {$1$};
\node [fill=white] at (0.5,1) {$3$};
\node [fill=white] at (0,0) {$7$};
\node [fill=white] at (1.5,1) {$5$};
\node [fill=white] at (1,0) {$6$};
\node [fill=white] at (2,0) {$9$};
\node [fill=white] at (3.5,2) {$2$};
\node [fill=white] at (3,1) {$4$};
\node [fill=white] at (4,1) {$8$};
\node [fill=white] at (4,0) {$10$};
\end{tikzpicture}
\end{gathered}
\end{equation}
This bijection is such that the connected components of the forest~$t(\sigma)$ are the blocks of~$\supp(\sigma)$.
Moreover, the number of leaves of~$t(\sigma)$ in~$S \in \supp(\sigma)$ is~$\exc(\sigma|_S)$ (a tree consisting only of its root has zero leaves).
Consequently, the total number of leaves of~$t(\sigma)$ is~$\exc(\sigma)$.

Let~$\sigma \in \frakS_d$ be a permutation with~$r$ excedances and let~$\rmX = \supp(\sigma)$.
For~$1 \leq i \leq r$, let~$J_i$ be the elements on the path from the~$i^{th}$ leaf of~$t(\sigma)$ to the root of the corresponding tree.
Define the element
\begin{equation}\label{eq:def_2-e-basis}
x_\sigma = \Big( \prod_{i=1}^r \log[\Delta_{J_i}] \Big) \cdot \B{E}_\rmX.
\end{equation}
For instance, if~$\sigma$ is the permutation in~\eqref{eq:permtree1}, then
\[
x_\sigma = \Big( \log[\Delta_{\{7,3,1\}}]\log[\Delta_{\{6,5,1\}}]\log[\Delta_{\{9,5,1\}}]\log[\Delta_{\{4,2\}}]\log[\Delta_{\{10,8,2\}}] \Big) \cdot \B{E}_{\rmX},
\]
where $\rmX = \{1,3,5,6,7,9\},\{2,4,8,10\}$.

\begin{conjecture}
For fixed~$\rmX \vdash [d]$ and~$r \leq d - \big|\rmX\big|$, the collection
\[
\{x_\sigma \,:\, \supp(\sigma) = \rmX,\ \exc(\sigma) = r\}
\]
is a linear basis of~$\Xi_r(\pi_d)\cdot \B{E}_\rmX$.
\end{conjecture}

It follows from the definition \eqref{eq:def_2-e-basis} that~$x_\sigma \in \Xi_r(\pi_d)\cdot \B{E}_\rmX$.
The content of the conjecture is that these elements are linearly independent.
Explicit computations show that this is the case for~$d = 2,3,4$.
Propositions~\ref{p:basisAdams1} and~\ref{p:basisAdams2} below prove the extremal cases~$r = 1$ and~$r = d - |\rmX|$ of this conjecture,
respectively.

\begin{proposition}\label{p:basisAdams1}
For a subset~$J \subseteq [d]$ of cardinality at least~$2$, let~$\rmX_J \vdash [d]$ be the partition whose only non-singleton block is~$J$.
Then,
\[
\log[\Delta_J] \cdot \B{E}_{\rmX_J}
\]
is a nonzero element.
Furthermore,~$\{ \log[\Delta_J] \cdot \B{E}_{\rmX_J} \,:\, J \subseteq [d],\, |J| \geq 2\}$ is a basis of simultaneous eigenvectors for~$\Xi_1(\pi_d)$.
\end{proposition}

\begin{proof}
First, observe that any cyclic permutation on a set with more than one element has at least one excedance, and only one cyclic permutation attains this minimum. Namely, the only cyclic permutation in $\frakS_d$ having one excedance is $(d \, d-1 \, \dots \, 2 \, 1)$.
Hence, a permutation~$\sigma \in \frakS_d$ has at least as many excedances as non-singleton blocks in~$\supp(\sigma)$.
It then follows from Theorem~\ref{t:dims_simul_e-spaces_A} that
\[
\dim_\RR(\Xi_1(\pi_d)\cdot \B{E}_\rmX) = \begin{cases}
1 & \text{if~$\rmX \vdash [d]$ has exactly one non-singleton block},\\
0 & \text{otherwise}.
\end{cases}
\]
Thus, the second statement follows from the first.

Since~$\{\log[\Delta_J] \,:\, J \subseteq [d],\, \big|J\big| \geq 2\}$ is a linear basis for~$\Xi_1(\pi_d)$,
it is enough to write~$\log[\Delta_J] \cdot \B{E}_{\rmX_J}$ as a non-trivial linear combination of these basis elements. Observe that if $F = (S_1,S_2,\dots,S_k)$, then~$[\Delta_J] \cdot \B{H}_F = [\Delta_{J \cap S_i}]$ where~$i$ is the first index for which the intersection~$J \cap S_i$ is nonempty. Thus,
\[
[\Delta_J] \cdot \B{H}_F = \begin{cases}
[\Delta_J] & \text{if }\supp(F) \leq \rmX_J,\\
[\text{a proper face of } \Delta_J] & \text{otherwise}.
\end{cases}
\]
Using that the action of $\B{H}_F$ is an algebra morphism, we have $\log[\Delta_J] \cdot \B{H}_F = \log([\Delta_J] \cdot \B{H}_F)$.
Hence, the coefficient of~$\log[\Delta_J]$ in~$\log[\Delta_J] \cdot \B{E}_{\rmX_J}$ is
\[
\dfrac{1}{\dim(\rmX_J)!} \sum_{\supp(F) = \rmX_J} 1 = 1.
\]
The equality follows since for any flat~$\rmX$ of the braid arrangement,~$\arr_d^\rmX$ has~$\dim(\rmX)!$ chambers.
\end{proof}

Note that the element~$\log[\Delta_J] \cdot \B{E}_{\rmX_J}$ in the proposition is precisely the element~$x_\sigma$ for the unique permutation~$\sigma$ with~$\supp(\sigma) = \rmX_J$ and~$\exc(\sigma) = 1$.
Indeed, $t(\sigma)$ consists of a increasing path whose nodes are the elements in $J$ and isolated roots indexed by the elements in $[d] \setminus J$.

\begin{proposition}\label{p:basisAdams2}
For any~$\rmX = \{S_1,\dots,S_k\} \vdash [d]$, the space~$\Xi_{d-k}(\pi_d) \cdot \B{E}_\rmX$ is~$1$-dimensional.
Moreover,
\begin{equation}\label{eq:basis_segments}
x_\rmX = \prod_{i=1}^k \Big( \prod_{j \neq \min(S_i)} \log[\Delta_{\{\min(S_i),j\}}] \Big)
\end{equation}
is a nonzero element in~$\Xi_{d-k}(\pi_d) \cdot \B{E}_\rmX$.
\end{proposition}

\begin{proof}
Observe that any cyclic permutation on a set with~$s$ elements has at most~$s-1$ excedances, and only one cyclic permutation attains this maximum.
Namely, the only cyclic permutation in $\frakS_d$ having $d-1$ excedances is $(1 \, 2 \, \dots \, d-1 \, d)$.
Hence, for any~$\rmX = \{S_1,\dots,S_k\} \vdash [d]$ there is exactly one permutation with~$\supp(\sigma) = \rmX$ and~$d-k$ excedances.
Theorem~\ref{t:dims_simul_e-spaces_A} then implies that
$
\dim_\RR(\Xi_{d-k}(\pi_d) \cdot \B{E}_\rmX) = 1.
$

It follows from Example~\ref{ex:prod_intervals} that the element~$x_\rmX$ is nonzero, and counting the number of factors in \eqref{eq:basis_segments} shows that $x_\rmX \in \Xi_{d-k}(\pi_d)$. Thus, we are only left to prove that~$x_\rmX \in \Xi_{d-k}(\pi_d) \cdot \B{E}_\rmX$. That is, that $x_\rmX \cdot \B{E}_\rmX = x_\rmX$

Let~$G \in \face[\arr_d]$ with~$\supp(G) > \rmX$.
Then, for some block~$S_i \in \rmX$ and some~$a \in S_i$,~$a$ and~$\min(S_i)$ are not in the same block of~$\supp(G)$.
Hence~$[\Delta_{\{\min(S_i),a\}}] \cdot \B{H}_G$ is the class of a point, and~$\log[\Delta_{\{\min(S_i),a\}}]\cdot \B{H}_G = \log[\{{\bf 0}\}] = 0$.
Since the action of $\B{H}_G$ is an algebra morphism, we get that $x_\rmX \cdot \B{H}_G = 0$.
Therefore,
\[
x_\rmX \cdot \B{E}_\rmX = x_\rmX \cdot \Big( \dfrac{1}{\dim(\rmX)!} \sum_{\supp(F) = \rmX} \B{H}_F \Big) 
= \dfrac{1}{\dim(\rmX)!} \sum_{\supp(F) = \rmX} x_\rmX \cdot \B{H}_F
= x_\rmX,
\]
as we wanted to show.
\end{proof}

The element $x_\rmX$ in the previous result is $x_\sigma$ for the only permutation $\sigma$ with $\supp(\sigma) = \rmX$ and $\exc(\sigma) = d - |\rmX|$.
In this case, the forest $t(\sigma)$ consists of a $k$ trees with node sets $S_1,\dots,S_k$, respectively.
Each tree has $\min(S_i)$ as a root and every other element in $S_i$ as a child of the root.

\section{The module of type B generalized permutahedra}\label{s:typeB}

In this section, we study the algebra $\Pi(\pi^B_d)$ of \emph{type B generalized permutahedra} and its structure as a module over the Tits algebra of the Coxeter arrangement of type B.

\subsection{The type B Coxeter arrangement}

The \textbf{type B Coxeter arrangement}~$\arr^\pm_d$ in~$\RR^d$ consists of the hyperplanes
$x_i=x_j$,~$x_i=-x_j$ for~$1 \leq i<j \leq d$ and~$x_k=0$ for~$1 \leq k \leq d$.
Its central face is the trivial cone $\{{\bf 0}\}$.
The \emph{Coxeter complex of type~$B_d$}, obtained by intersecting the arrangement with a sphere around the origin in~$\RR^d$, is shown below for~$d=2$ and~$3$.
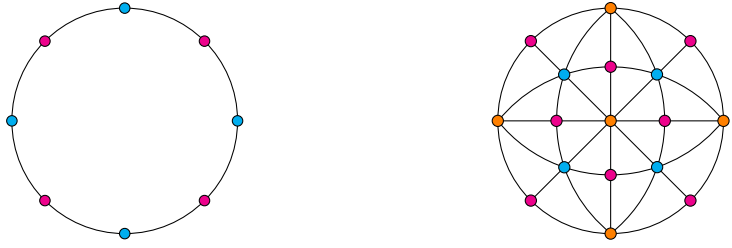
\[
\begin{gathered}
\begin{tikzpicture}[scale=1.5]
\draw circle (1);
\node [circle, inner sep=1.4pt,fill=cyan,draw] at (0:1) {};
\node [circle, inner sep=1.4pt,fill=magenta,draw] at (45:1) {};
\node [circle, inner sep=1.4pt,fill=cyan,draw] at (90:1) {};
\node [circle, inner sep=1.4pt,fill=magenta,draw] at (135:1) {};
\node [circle, inner sep=1.4pt,fill=cyan,draw] at (180:1) {};
\node [circle, inner sep=1.4pt,fill=magenta,draw] at (225:1) {};
\node [circle, inner sep=1.4pt,fill=cyan,draw] at (270:1) {};
\node [circle, inner sep=1.4pt,fill=magenta,draw] at (315:1) {};
\end{tikzpicture}
\end{gathered}
\hspace*{.2\linewidth}
\begin{gathered}
\begin{tikzpicture}[scale=0.6]
\newdimen\R
\R=2.5cm \coordinate (madhya) at (0,0);
\draw (madhya) circle (\R); \coordinate (P) at (-1.2,0);
\coordinate (Q) at (1.2,0);
\coordinate (R) at (0,1.2);
\coordinate (S) at (0,-1.2);
\coordinate (A1) at (0:\R);
\coordinate (B1) at (180:\R);
\draw (A1) -- (B1);
\arcThroughThreePoints{A1}{R}{B1};
\arcThroughThreePoints{B1}{S}{A1};
\coordinate (A2) at (45:\R);
\coordinate (B2) at (225:\R);
\draw (A2) -- (B2);
\coordinate (A3) at (90:\R);
\coordinate (B3) at (270:\R);
\draw (A3) -- (B3);
\arcThroughThreePoints{A3}{P}{B3};
\arcThroughThreePoints{B3}{Q}{A3};
\coordinate (A4) at (135:\R);
\coordinate (B4) at (315:\R);
\draw (A4) -- (B4);
\path[name intersections={of=kamaanA1 and kamaanB3, by=a1b3}];
\path[name intersections={of=kamaanB1 and kamaanA3, by=b1a3}];
\path[name intersections={of=kamaanA1 and kamaanA3, by=a1a3}];
\path[name intersections={of=kamaanB1 and kamaanB3, by=b1b3}];
\node [circle, inner sep=1.5pt,fill=orange,draw] at (madhya) {};
\node [circle, inner sep=1.5pt,fill=cyan,draw] at (a1b3) {};
\node [circle, inner sep=1.5pt,fill=cyan,draw] at (b1a3) {};
\node [circle, inner sep=1.5pt,fill=cyan,draw] at (a1a3) {};
\node [circle, inner sep=1.5pt,fill=cyan,draw] at (b1b3) {};
\node [circle, inner sep=1.5pt,fill=magenta,draw] at (P) {};
\node [circle, inner sep=1.5pt,fill=magenta,draw] at (Q) {};
\node [circle, inner sep=1.5pt,fill=magenta,draw] at (R) {};
\node [circle, inner sep=1.5pt,fill=magenta,draw] at (S) {};
\node [circle, inner sep=1.5pt,fill=orange,draw] at (A1) {};
\node [circle, inner sep=1.5pt,fill=orange,draw] at (B1) {};
\node [circle, inner sep=1.5pt,fill=magenta,draw] at (A2) {};
\node [circle, inner sep=1.5pt,fill=magenta,draw] at (B2) {};
\node [circle, inner sep=1.5pt,fill=orange,draw] at (A3) {};
\node [circle, inner sep=1.5pt,fill=orange,draw] at (B3) {};
\node [circle, inner sep=1.5pt,fill=magenta,draw] at (A4) {};
\node [circle, inner sep=1.5pt,fill=magenta,draw] at (B4) {};
\end{tikzpicture}
\end{gathered}
\]
Flats and faces of $\arr^\pm_d$ are in correspondence with signed set partitions and signed set compositions of $[\pm d] := \{-d,-d+1,\dots,-1,1,\dots,d-1,d\}$, as originally introduced by Reiner~\citep{reiner97noncrossing}.

Let~$I$ be a finite set with an fixed point free involution~$x \mapsto \opp{x}$.
For instance, $[\pm d]$ with involution $\opp{x} = -x$.
A subset $S \subseteq I$ is said to be \textbf{involution-exclusive} if $S \cap \opp{S} = \emptyset$, where $\opp{S} = \{\opp{x} \,:\, x \in S\}$.
In contrast, $S \subseteq I$ is said to be \textbf{involution-inclusive} if $S = \opp{S}$.
Given an arbitrary subset $S \subseteq I$, we let $\pm S$ be the involution-inclusive set $S \cup \opp{S}$.

A \textbf{signed set partition} of~$I$ is a weak set partition of the form~$\rmX = \{S_0,S_1,\opp{S_1},\dots,S_k,\opp{S_k}\}$, where $S_0$ is involution-inclusive and allowed to be empty, and each $S_i$ for $i \neq 0$ is nonempty and involution-exclusive. We call $S_0$ the \textbf{zero block} of $\rmX$.
We write~$\rmX \vdash^B I$ to denote that~$\rmX$ is a signed partition of~$I$.
Given a signed partition~$\rmX \vdash^B [\pm d]$, the corresponding flat of $\arr^\pm_d$ is the intersection of the hyperplanes $x_i = x_j$ for each $i,j$ in the same block of $\rmX$, where for $k \in [d]$, we let $x_{\opp{k}}$ denote $-x_k$.
In particular, if $k \in [d]$ is in the zero block of $\rmX$, the corresponding flat lies in the hyperplane $x_k = 0$.
For instance, consider the following two examples for $d = 7$:
\[
\begin{array}{CCC}
x_1 = x_3, \quad x_2 = -x_4 = x_5, \quad x_6 = x_7
& \hspace*{.05\linewidth} \longleftrightarrow \hspace*{.05\linewidth} & 
\{ \emptyset , 13, \bar{1}\bar{3}, 2\bar{4}5 , \bar{2}4\bar{5} , 67 , \bar{6}\bar{7} \}, \\
x_1 = x_3 = 0, \quad x_2 = -x_4 = x_5, \quad x_6 = x_7
& \longleftrightarrow & 
\{ 1\bar{1}3\bar{3}, 2\bar{4}5 , \bar{2}4\bar{5} , 67 , \bar{6}\bar{7} \}. \\
\end{array}
\]
The zero block in the first partition is empty since the corresponding flat is not contained in any coordinate hyperplane. We use $\rmX$ to denote both a flat of $\arr^\pm_d$ and the corresponding signed partition of $[\pm d]$.
Observe that the number of nonzero blocks of $\rmX$ is $2\dim(\rmX)$.
The partial order relation of $\fflat[\arr^\pm_d]$ becomes the ordering by refinement of signed partitions. For instance,
$\{ 1\bar{1}3\bar{3}, 2\bar{4}5 , \bar{2}4\bar{5} , 67 , \bar{6}\bar{7} \} \leq \{ \emptyset , 13, \bar{1}\bar{3}, 2\bar{4}5 , \bar{2}4\bar{5} , 67 , \bar{6}\bar{7} \} \leq \{ \emptyset , 13, \bar{1}\bar{3}, 25 , \bar{2}\bar{5}, 4, \bar{4} , 67 , \bar{6}\bar{7} \}$.

If $S$ is an involution-inclusive union of blocks of $\rmX \vdash^B I$, we let $\rmX|_S \vdash^B S$ denote the corresponding signed partition.
If, on the other hand, $S$ is an involution-exclusive union of blocks of $\rmX \vdash^B I$, we let $\rmX|_S \vdash S$ denote the corresponding (type A) partition.

Let~$\rmX = \{S_0,S_1,\opp{S_1},\dots,S_k,\opp{S_k}\}$.
A choice of $\rmY \geq \rmX$ is equivalent to the choice of a signed partition $\rmY|_{S_0} \vdash^B S_0$ of the zero block and of (type A) partitions $\rmY|_{S_i} \vdash S_i$ for $i = 1,\dots, k$.
Note that in this case, $\rmY|_{\opp{S_i}}$ is automatically determined by $\rmY|_{S_i}$.

With $\rmX$ and $\rmY$ as above, the Möbius function of~$\fflat[\arr^\pm_d]$ is determined by
\begin{equation}\label{eq:mu_prod_B}
\mu(\perp,\rmX) = (-1)^k(2k-1)!!
\qqand
\mu(\rmX,\rmY) = \mu(\perp,\rmY|_{S_0}) \mu(\perp,\rmY|_{S_1}) \dots \mu(\perp,\rmY|_{S_k}),
\end{equation}
where~$(2k-1)!!$ is the double factorial~$(2k-1)!! = (2k-1)(2k-3)\dots 1$,
and in each factor $\perp$ denotes the minimum (signed) partition of ($S_0$) $S_i$.

A \textbf{signed composition} is an ordered signed set partition with the property that $S_i$ precedes $S_j$ if and only if $\opp{S_j}$ precedes $\opp{S_i}$.
The following examples illustrate the identification between signed compositions of $\pm [d]$ and faces of $\arr^\pm_d$:
\[
\begin{array}{CCC}
x_6 = x_7 > -x_2 = x_4 = -x_5 > x_1 = x_3 > 0 
& \hspace*{.05\linewidth} \longleftrightarrow \hspace*{.05\linewidth} &
(67,\bar{2}4\bar{5},13,\emptyset,\bar{1}\bar{3},2\bar{4}5,\bar{6}\bar{7})\\
x_6 = x_7 > -x_2 = x_4 = -x_5 > x_1 = x_3 = 0 
& \longleftrightarrow &
(67,\bar{2}4\bar{5},1\bar{1}3\bar{3},2\bar{4}5,\bar{6}\bar{7})
\end{array}
\]
Note that we can alternatively describe the first face with the inequalities $0 > -x_1 = -x_3 > x_2 = -x_4 = x_5 > -x_6 = -x_7$.

\subsection{The hyperoctahedral group and the Type B Eulerian polynomial}

The hyperoctahedral group~$\frakB_d$ is the group of bijections~$\sigma : [\pm d] \rightarrow [\pm d]$ satisfying~$\sigma(\opp{i}) = \opp{\sigma(i)}$ for all $i \in [\pm d]$ under composition. Elements in~$\frakB_d$ are called \textbf{signed permutations}. 
The group~$\frakB_d$ acts on~$\RR^d$ by permutation and sign changes of coordinates:
\[
\sigma(x_1,x_2,\dots,x_d) = (x_{\sigma(1)},x_{\sigma(2)},\dots,x_{\sigma(d)}).
\]
Recall that, for instance, $x_{\opp{1}} = -x_1$.
For a signed permutation $\sigma \in \frakB_d$, we let~$\supp(\sigma)$ denote the subspace of points fixed by the action of~$\sigma$; it is a flat of~$\arr^\pm_d$.
Under the identification above,~$\supp(\sigma)$ is the signed partition of~$[\pm d]$ obtained from the underlying the cycle decomposition of~$\sigma$ by merging all the blocks that contain an element~$i$ and its negative~$\opp{i}$.
For example, if in cycle notation $\sigma = (1)(\bar{1})(2\bar{2})(34\bar{3}\bar{4})(5\bar{6})(\bar{5}6)$, then $\supp(\sigma) = \{2\bar{2}3\bar{3}4\bar{4},1,\bar{1},5\bar{6},\bar{5}6\}$.

Let~$\sigma \in \frakB_n$.
The restriction~$\sigma|_{S_0}$ to the zero block~$S_0 \in \supp(\sigma)$ is a signed permutation of~$S_0$.
Its action on~$\RR^{|S_0|/2}$ does not fix any nonzero vector, so~$\supp(\sigma|_{S_0}) = \perp$.
For a nonzero block~$S\in \supp(\sigma)$,~$\sigma|_S \in \frakC(S)$ is a cyclic permutation of the elements in~$S$.
The restriction~$\sigma|_{\pm S}$ is again a signed permutation, and it is completely determined by either~$\sigma|_S$ or~$\sigma|_{\opp{S}}$.

We present some statistics on signed permutations.
For~$\sigma \in \frakB_d$, let
\begin{align*}
\Des(\sigma) & = \{i \in [d-1] \cup \{0\} \,:\, \sigma(i) > \sigma(i+1)\} &
\des(\sigma) & = \big|\Des(\sigma)\big| \\
\Exc(\sigma) & = \{i \in [d-1] \,:\, \sigma(i) > i\} &
\exc(\sigma) & = \big|\Exc(\sigma)\big| \\
\Neg(\sigma) & = \{i \in [d] \,:\, \sigma(i) < 0\} &
\fneg(\sigma) & = \big|\Neg(\sigma)\big| \\
& &
\fexc(\sigma) & = 2 \exc(\sigma) + \fneg(\sigma),
\end{align*}
where we set $\sigma(0) = 0$.
Elements in the sets above are \textbf{descents}, \textbf{excedances} and \textbf{negations} of~$\sigma$, respectively.
The last statistic is called the \textbf{flag-excedance} of a signed permutation.
We define one last statistic, the B-excedance of~$\sigma$:
\begin{equation}\label{eq:def_B-exc}
\exc_B(\sigma) = \lfloor \tfrac{\fexc(\sigma)+1}{2} \rfloor = \exc(\sigma) + \lfloor \tfrac{\fneg(\sigma)+1}{2} \rfloor.
\end{equation}
Foata and Han~\citep[Section 9]{fh09signedV} show that descents and B-excedances are equidistributed.
That is,
\[
B_{d,k} := \big|\{ \sigma \in \frakB_d \,:\, \des(\sigma) = k \}\big|
=
\big|\{ \sigma \in \frakB_d \,:\, \exc_B(\sigma) = k \}\big|,
\]
for all possible values of~$k$.
The numbers~$B_{d,k}$ are the \textbf{Eulerian numbers of type B} (OEIS: \href{https://oeis.org/A060187}{A060187}).
The \textbf{type B Eulerian polynomial}~$B_d(z)$ is:
\[
B_d(z) = \sum_{k=0}^d B_{d,k}z^k = \sum_{\sigma \in \frakB_d} z^{\exc_B(\sigma)}.
\]
The exponential generating function of these polynomials is first due to Brenti~\citep[Theorem 3.4]{brenti94cox-eul}.
We will be interested in the \textbf{type B exponential generating function} of these polynomials:
\begin{equation}\label{eq:generating_eulerian_B}
B(z,x) = 1 + \sum_{d \geq 1} B_d(z) \dfrac{x^d}{(2d)!!} = \dfrac{(1-z)e^{x(1-z)/2}}{1-ze^{x(1-z)}},
\end{equation}
where~$(2d)!!$ is the double factorial~$(2d)!! = (2d)(2d-2)\dots 2 = 2^d d!$.
Substituting~$x$ by~$2x$ one recovers Brenti's original formula.

\subsection{The module of type B Generalized permutahedra}

The type B permutahedron~$\pi^B_d \subseteq \RR^d$ is the convex hull of the $\frakB_d$-orbit of the point~$(1,2,\dots,d)$.
It is full-dimensional and a zonotope of~$\arr^\pm_d$.
We now consider the module~$\Pi(\pi^B_d)$.
The main result of this section is the following.

\begin{theorem}\label{t:dims_simul_e-spaces_B}
For any flat~$\rmX \in \fflat[\arr^\pm_d]$ and~$r=0,1,\dots, d$,
\[
\eta_\rmX(\Xi_r(\pi^B_d)) = \big|\{ \sigma \in \frakB_d \,:\, \supp(\sigma) = \rmX,\ \exc_B(\sigma) = r \}\big|.
\]
\end{theorem}

\begin{figure}[ht]
\[
\begin{gathered}
\begin{tikzpicture}[scale=.6]
\draw [very thick, fill=gray!30!white] (2,1) node [right] {\scriptsize$(2,1)$}
-- (1,2) node [above right] {\scriptsize$(1,2)$}
-- (-1,2) node [above left] {\scriptsize$(-1,2)$}
-- (-2,1) node [left] {\scriptsize$(-2,1)$}
-- (-2,-1) node [left] {\scriptsize$(-2,-1)$}
-- (-1,-2) node [below left] {\scriptsize$(-1,-2)$}
-- (1,-2) node [below right] {\scriptsize$(1,-2)$}
-- (2,-1) node [right] {\scriptsize$(2,-1)$}
-- cycle;
\end{tikzpicture}
\end{gathered}
\hspace*{.2\linewidth}
\begin{gathered}
\tdplotsetmaincoords{60}{120}
\begin{tikzpicture}[tdplot_main_coords,scale=.6]
\draw [very thick] (-1, 2, 3) -- (1, 2, 3);
\draw [very thick] (1, -2, 3) -- (2, -1, 3);
\draw [very thick] (-1, -2, 3) -- (1, -2, 3);
\draw [very thick] (-2, 1, 3) -- (-1, 2, 3);
\draw [very thick] (-2, -1, 3) -- (-2, 1, 3);
\draw [very thick] (-2, -1, 3) -- (-1, -2, 3);
\draw [very thick] (2, -1, 3) -- (2, 1, 3);
\draw [very thick] (1, 2, 3) -- (2, 1, 3);
\draw [very thick] (2, 1, -3) -- (3, 1, -2);
\draw [very thick] (2, -1, -3) -- (3, -1, -2);
\draw [very thick] (2, -1, -3) -- (2, 1, -3);
\draw [very thick] (3, -1, -2) -- (3, 1, -2);
\draw [very thick] (2, -3, -1) -- (2, -3, 1);
\draw [very thick] (2, -3, 1) -- (3, -2, 1);
\draw [very thick] (2, -3, -1) -- (3, -2, -1);
\draw [very thick] (3, -2, -1) -- (3, -2, 1);
\draw [very thick] (1, 3, -2) -- (2, 3, -1);
\draw [very thick] (1, 2, -3) -- (2, 1, -3);
\draw [very thick] (1, 2, -3) -- (1, 3, -2);
\draw [very thick] (2, 3, -1) -- (3, 2, -1);
\draw [very thick] (3, 1, -2) -- (3, 2, -1);
\draw [very thick] (-1, 3, 2) -- (1, 3, 2);
\draw [very thick] (-2, 3, 1) -- (-1, 3, 2);
\draw [very thick] (-1, 3, -2) -- (1, 3, -2);
\draw [very thick] (-2, 3, -1) -- (-2, 3, 1);
\draw [very thick] (-2, 3, -1) -- (-1, 3, -2);
\draw [very thick] (2, 3, -1) -- (2, 3, 1);
\draw [very thick] (1, 3, 2) -- (2, 3, 1);
\draw [very thick] (-1, 2, 3) -- (-1, 3, 2);
\draw [very thick] (1, 2, 3) -- (1, 3, 2);
\draw [very thick] (1, -3, 2) -- (2, -3, 1);
\draw [dashed] (1, -3, -2) -- (2, -3, -1);
\draw [very thick] (-1, -3, 2) -- (1, -3, 2);
\draw [dashed] (-1, -3, -2) -- (1, -3, -2);
\draw [dashed] (-2, -3, -1) -- (-2, -3, 1);
\draw [dashed] (-2, -3, -1) -- (-1, -3, -2);
\draw [dashed] (-2, -3, 1) -- (-1, -3, 2);
\draw [very thick] (1, -3, 2) -- (1, -2, 3);
\draw [very thick] (2, -1, 3) -- (3, -1, 2);
\draw [very thick] (3, -2, 1) -- (3, -1, 2);
\draw [dashed] (1, -2, -3) -- (2, -1, -3);
\draw [dashed] (1, -3, -2) -- (1, -2, -3);
\draw [very thick] (3, -2, -1) -- (3, -1, -2);
\draw [very thick] (-1, -3, 2) -- (-1, -2, 3);
\draw [very thick] (-3, 2, 1) -- (-2, 3, 1);
\draw [very thick] (-3, 1, 2) -- (-3, 2, 1);
\draw [very thick] (-3, 1, 2) -- (-2, 1, 3);
\draw [dashed] (-1, -2, -3) -- (1, -2, -3);
\draw [dashed] (-1, 2, -3) -- (1, 2, -3);
\draw [dashed] (-2, 1, -3) -- (-1, 2, -3);
\draw [dashed] (-2, -1, -3) -- (-1, -2, -3);
\draw [dashed] (-2, -1, -3) -- (-2, 1, -3);
\draw [dashed] (-1, -3, -2) -- (-1, -2, -3);
\draw [dashed] (-1, 2, -3) -- (-1, 3, -2);
\draw [dashed] (-3, 2, -1) -- (-3, 2, 1);
\draw [dashed] (-3, 2, -1) -- (-2, 3, -1);
\draw [dashed] (-3, 1, -2) -- (-3, 2, -1);
\draw [dashed] (-3, -1, 2) -- (-3, 1, 2);
\draw [dashed] (-3, -2, -1) -- (-3, -2, 1);
\draw [dashed] (-3, -2, -1) -- (-3, -1, -2);
\draw [dashed] (-3, -1, -2) -- (-3, 1, -2);
\draw [dashed] (-3, -2, 1) -- (-3, -1, 2);
\draw [dashed] (-3, 1, -2) -- (-2, 1, -3);
\draw [dashed] (-3, -1, 2) -- (-2, -1, 3);
\draw [dashed] (-3, -2, -1) -- (-2, -3, -1);
\draw [dashed] (-3, -2, 1) -- (-2, -3, 1);
\draw [dashed] (-3, -1, -2) -- (-2, -1, -3);
\draw [very thick] (2, 1, 3) -- (3, 1, 2);
\draw [very thick] (3, -1, 2) -- (3, 1, 2);
\draw [very thick] (2, 3, 1) -- (3, 2, 1);
\draw [very thick] (3, 2, -1) -- (3, 2, 1);
\draw [very thick] (3, 1, 2) -- (3, 2, 1);
\end{tikzpicture}
\end{gathered}
\]
\caption{Type B permutahedron in $\RR^2$ and $\RR^3$.}
\end{figure}

As in type A, the relation between~$\Pi(\pi^B_d)$ and statistics on~$\frakB_d$ is due to Brenti's result showing that $h(\pi^B_d,z) = B_d(z)$.
For a flat~$\rmX = \{S_0,S_1,\opp{S_1},\dots,S_k,\opp{S_k}\}$ of~$\arr^\pm_d$,
the face $(\pi^B_d)_\rmX$ is a translate of $\pi^B_{|S_0|/2} \times \pi_{|S_1|} \times \dots \times \pi_{|S_k|}$, a product of lower-dimensional permutahedra of type A and B, where exactly one factor is of type B.
Thus,
\begin{equation}\label{eq:prod_perm_B}
h((\pi^B_d)_\rmX,z) = B_{|S_0|/2}(z) \cdot A_{|S_1|}(z) \cdot \ldots \cdot A_{|S_k|}(z).
\end{equation}

The following result is the analogous in type B of Lemma~\ref{l:cyc_exc}.

\begin{lemma}\label{l:cyc_exc_B}
For every $d \geq 1$,
\begin{equation}\label{eq:lem_cyc_exc_B}
\sum_{\rmX = \{S_0,\dots,S_k,\opp{S_k}\} \vdash^B [\pm d]} \mu(\perp,\rmX) B_{|S_0|/2}(z) \cdot A_{|S_1|}(z) \cdot \ldots \cdot A_{|S_k|}(z)
=
\sum_{\substack{\sigma \in \frakB_d \\ \supp(\sigma) = \perp}} z^{\exc_B(\sigma)}.
\end{equation}
\end{lemma}
In the same spirit as the proof of Lemma~\ref{l:cyc_exc}, we will establish~\eqref{eq:lem_cyc_exc_B} by comparing the type B exponential generating function of both sides of the equality.
An important tool in this proof is the following analog of the Compositional Formula (Theorem~\ref{t:compositional_formula}) for type B generating functions.

\begin{proposition}[Type B Compositional Formula]\label{p:typeB-compositional}
Let
\[
f(x) = 1 + \sum_{d \geq 1} f_d \dfrac{x^d}{(2d)!!} \qquad
g(x) = 1 + \sum_{d \geq 1} g_d \dfrac{x^d}{(2d)!!} \qquad
a(x) = \sum_{d \geq 1} a_d \dfrac{x^d}{d!}.
\]
If
\[
h(x) = 1 + \sum_{d \geq 1} h_d \dfrac{x^d}{(2d)!!}
\quad \text{where} \quad
h_d = \sum_{\{ S_0 , S_1,\opp{S_1}, \dots , S_k , \opp{S_k} \} \vdash^B [\pm d]} f_{|S_0|/2} g_k a_{|S_1|} \dots a_{|S_k|},
\]
then
\[
h(x) = f(x)g(a(x)).
\]
\end{proposition}

\begin{proof}
Using the usual Compositional Formula, the coefficient of~$\dfrac{x^d}{(2d)!!}$ in~$f(x)g(a(x))$ is
\[
\begin{array}{RL}
& 2^d \, d! \sum_{r = 0}^d \dfrac{f_r}{2^r \, r!} \left( \dfrac{1}{(d-r)!} \sum_{\{K_1,\dots,K_k\} \vdash [d-r]} \dfrac{g_k}{2^k} \, a_{|K_1|} \dots a_{|K_k|} \right) \\
= & \sum_{r = 0}^d {d \choose r} \left( \sum_{\{K_1,\dots,K_k\} \vdash [d-r]} 2^{d-r-k} \, f_r \, g_k \, a_{|K_1|} \dots a_{|K_k|} \right) \\
= & \sum_{\{ S_0 , S_1,\opp{S_1}, \dots , S_k , \opp{S_k} \} \vdash^B [\pm d]} f_{|S_0|/2} \, g_k \, a_{|S_1|} \dots a_{|S_k|},
\end{array}
\]
this is precisely the coefficient of~$\dfrac{x^d}{(2d)!!}$ in~$h(x)$.
To verify the last equality, note that choosing a type B partition~$\{ S_0 , S_1,\opp{S_1}, \dots , S_k , \opp{S_k} \} \vdash^B [\pm d]$ with~$|S_0| = 2r$ is equivalent to:
\begin{enumerate}
\item choosing a subset~$K_0 \in {[d] \choose r}$ and setting~$S_0 = \pm K_0$,
\item choosing a partition~$\{K_1,\dots,K_k\}$ of~$[d] \setminus K_0$, and
\item constructing blocks $\{S_i,\opp{S_i}\}$ from $K_i$ as follows: for each~$j \in K_i \setminus \{\max K_i\}$, choose whether~$j$ and $\max K_i$ will be in the same or in \emph{opposite} blocks.
\end{enumerate}
There are precisely $2^{d-r-k}$ possible choices in the last step.
\end{proof}

Taking~$g_d = 1$ in the Type B Compositional Formula we deduce the following.

\begin{corollary}[Type B Exponential Formula]
Let~$f(x)$ and~$a(x)$ be as before.
If
\[
h(x) = 1 + \sum_{d \geq 1} h_d \dfrac{x^d}{(2d)!!}
\qquad \text{where} \qquad
h_d = \sum_{\rmX \vdash^B [\pm d]} f_{|S_0|/2} a_{|S_1|} \dots a_{|S_k|},
\]
then
\[
h(x) = f(x)\, \exp(a(x)/2).
\]
\end{corollary}

In the proof of Theorem~\ref{t:dims_simul_e-spaces_A}, we used that for (type A) permutations~$\sigma \in \frakS_d$,~$\exc(\sigma)$ equals the sum of~$\exc(\sigma|_S)$ as~$S$ runs through the blocks of~$\supp(\sigma)$.
For signed permutations, one easily checks this also holds for the statistics~$\exc$ and~$\fneg$.
However, it is not obvious at all that the same is true for~$\exc_B$, since its definition uses the floor function.

Consider the order~$\prec$ of the elements of any involution-exclusive subset~$S \subseteq [\pm d]$ defined by:
\begin{equation}\label{eq:ord_exc}
i \prec j \Longleftrightarrow \begin{cases}
0 < i < j, \text{ or}\\
i < 0 < j, \text{ or}\\
j < i < 0.
\end{cases}
\end{equation}

\begin{proposition}\label{p:exc_b_additive}
Let~$\sigma \in \frakB_d$ and~$\supp(\sigma) = \{S_0,S_1,\opp{S_1},\dots,S_k,\opp{S_k}\}$.
Then,
\[
\exc_B(\sigma) = \exc_B(\sigma|_{S_0}) + \exc_\prec(\sigma|_{S_1}) + \dots + \exc_\prec(\sigma|_{S_k}),
\]
where~$\exc_\prec(\sigma|_{S_i})$ is the number of usual \textup{(}type A\textup{)} excedances of~$\sigma|_{S_i}$ with respect to the order~$\prec$.
\end{proposition}

\begin{proof}
For~$i \geq 1$, write~$\sigma|_{S_i} = (j_1 j_2 \dots j_\ell)$ in cycle notation.
Since~$\sigma|_{\opp{S_i}} = (\opp{j_1} \, \opp{j_2} \dots \opp{j_\ell})$,
negations of~$\sigma|_{\pm S_i}$ are in correspondence with changes of sign in the sequence
\[
j_1 \mapsto j_2 \mapsto \dots \mapsto j_\ell \mapsto j_1.
\]
It follows that
\[
\fneg(\sigma|_{\pm S_i}) = 2 \cdot \big|\{j \in S_i \,:\, j < 0 < \sigma(j) \}\big|
\]
is an even number.

Observe that according to the three cases of definition~\eqref{eq:ord_exc}, a~$\prec$-excedance of~$\sigma|_{S_i}$ corresponds to either
\[
\begin{cases}
\text{an excedance of } \sigma|_{\pm S_i} \text{ occurring in } S_i, \text{ or}\\
\text{a negation of } \sigma|_{\pm S_i} \text{ occurring in } \opp{S_i}, \text{ or}\\
\text{an excedance of } \sigma|_{\pm S_i} \text{ occurring in } \opp{S_i},
\end{cases}
\]
respectively.
Since exactly half of the negations of~$\sigma|_{\pm S_i}$ occur in~$\opp{S_i}$, we deduce that
\[
\exc_\prec(\sigma|_{S_i}) = \exc(\sigma|_{\pm S_i}) + \tfrac{\fneg(\sigma|_{\pm S_i})}{2}.
\]
Thus, in view of~\eqref{eq:def_B-exc} and using that $\tfrac{\fneg(\sigma|_{\pm S_i})}{2}$ is always an integer,
\[
\begin{array}{RL}
\exc_B(\sigma)
= & \exc(\sigma) + \lfloor \tfrac{\fneg(\sigma)+1}{2} \rfloor \\
= & \exc(\sigma|_{S_0}) + \sum_i \exc(\sigma|_{\pm S_i}) + \lfloor \tfrac{\fneg(\sigma|_{S_0})+\sum_i\fneg(\sigma|_{\pm S_i})+1}{2} \rfloor \\
= & \exc(\sigma|_{S_0}) + \sum_i \exc(\sigma|_{\pm S_i}) + \lfloor \tfrac{\fneg(\sigma|_{S_0})+1}{2} \rfloor +\sum_i \tfrac{\fneg(\sigma|_{\pm S_i})}{2} \\
= & \exc_B(\sigma|_{S_0}) + \exc_\prec(\sigma|_{S_1}) + \dots + \exc_\prec(\sigma|_{S_k}),
\end{array}
\]
as we wanted to show.
\end{proof}

\begin{proof}[Proof of Lemma~\ref{l:cyc_exc_B}]
Recall that~$\mu(\perp,\rmX) = (-1)^k (2k-1)!!$, where~$|\rmX| = 2k+1$.
Observe that
\[
1 + \sum_{d \geq 1} (-1)^d (2d-1)!! \dfrac{x^d}{(2d)!!} = \sum_{d \geq 0} {-1/2 \choose d} x^d = (1+x)^{-1/2}.
\]
Using the Type B Compositional formula, we conclude that the type B exponential generating function of the LHS of~\eqref{eq:lem_cyc_exc_B} is
\[
B(z,x)(1+(A(z,x)-1))^{-1/2} - 1= \dfrac{B(z,x)}{\sqrt{A(z,x)}} - 1,
\]
where~$A(z,x)$ and~$B(z,x)$ are the generating functions in~\eqref{eq:generating_eulerian} and~\eqref{eq:generating_eulerian_B}, respectively.
On the other hand, Proposition~\ref{p:exc_b_additive} shows that for each partition~$\rmX = \{S_0,S_1,\opp{S_1},\dots,S_k,\opp{S_k}\} \vdash^B [\pm d]$,
\begin{equation}\label{eq:exc_cycles_B}
\sum_{\substack{\sigma \in \frakB_d \\ \supp(\sigma) = \rmX}} z^{\exc_B(\sigma)}
= \Big( \sum_{\substack{\sigma_0 \in \frakB(S_0) \\ \supp(\sigma_0) = \perp}} z^{\exc_B(\sigma_0)} \Big) \prod_{i=1}^k \Big( \sum_{\sigma \in \frakC(|S_i|)} z^{\exc(\sigma)} \Big).
\end{equation}
In the proof of Lemma~\ref{l:cyc_exc}, we showed that the (usual) exponential generating function of the terms in the product is~$\log (A(z,x))$.
An application of the type B Exponential Formula and~\eqref{eq:exc_cycles_B} yields
\[
B(z,x) = \bigg( 1 + \sum_{d \geq 1} \Big( \sum_{\substack{\sigma \in \frakB_d \\ \supp(\sigma) = \perp}} z^{\exc_B(\sigma)} \Big) \dfrac{x^d}{(2d)!!} \bigg)\, \exp\Big(\dfrac{\log (A(z,x))}{2}\Big) 
\]
Dividing both sides by~$\exp\big(\tfrac{\log (A(z,x))}{2}\big) = \sqrt{A(z,x)}$ and subtracting $1$ yields the result.
\end{proof}

We are now ready to complete the proof of the main theorem in this section.
The steps of the proof mirror those of the type A result.

\begin{proof}[Proof of {\rm Theorem~\ref{t:dims_simul_e-spaces_B}}]
We will use formula \eqref{eq:eta_simplicial_arr} to compute the values $\eta_\rmX(\Xi_r(\pi^B_d))$.
Recall that the formula reads
\[
\sum_r \eta_\rmX(\Xi_r(\pi^B_d)) z^r \sum_{\rmY \geq \rmX} \mu(\rmX,\rmY)h((\pi^B_d)_\rmY,z).
\]
Using formulas~\eqref{eq:mu_prod_B} and~\eqref{eq:prod_perm_B}, we see that for a flat $\rmX = \{S_0,\dots,S_k,\opp{S_k}\}$ this expression equals the following product
\begin{multline*}
\bigg( \sum_{\substack{\rmY \vdash^B S_0 \\ \rmY = \{T_0,\dots,T_\ell,\opp{T_\ell}\}}} \mu(\perp,\rmY) B_{|T_0|/2} A_{|T_1|} \dots A_{|T_\ell|} \bigg) \cdot \\
\prod_{i=1}^k \bigg( \sum_{\substack{\rmY_i \vdash S_i \\ \rmY_i = \{T^i_1,\dots,T^i_\ell\}}} \mu(\perp,\rmY_i) A_{|T^i_1|}(z) \cdot \ldots \cdot A_{|T^i_\ell|}(z) \bigg).
\end{multline*}
Using Lemmas~\ref{l:cyc_exc} and~\ref{l:cyc_exc_B} in each factor, we deduce
\[
\sum_r \eta_\rmX(\Xi_r(\pi^B_d)) z^r
= \bigg( \sum_{\substack{\sigma \in \frakB(S_0) \\ \supp(\sigma) = \perp}} z^{\exc_B(\sigma)} \bigg)
\prod_{i=1}^k \Big( \sum_{\sigma \in \frakC(|S_i|)} z^{\exc(\sigma)} \Big)
= \sum_{\substack{\sigma \in \frakB_d \\ \supp(\sigma) = \rmX}} z^{\exc_B(\sigma)},
\]
where the last equality is~\eqref{eq:exc_cycles_B}.
Finally, taking the coefficient of~$z^r$ on both sides yields the result.
\end{proof}

Adding over all flats with the same dimension in Theorem~\ref{t:dims_simul_e-spaces_B}, we conclude the following.

\begin{corollary}
Let~$w \in \tits[\arr^\pm_d]$ be a characteristic element of non-critical parameter~$t$ and $\lambda > 0$.
The dimension of the simultaneous eigenspace for $w$ and $\delta_\lambda$ with eigenvalues $t^k$ and $\lambda^r$ is
\[
\big| \{ \sigma \in \frakB_d \,:\, \dim(\supp(\sigma)) = k,\ \exc_B(\sigma) = r \}\big|.
\]
\end{corollary}

As in the type A case, we can modify the proof of the previous Lemma to obtain the generating function for the bivariate polynomials $\sum_{\sigma \in \frakB_d} t^{\dim(\supp(\sigma))} z^{\exc_B(\sigma)}$.
To the best of our knowledge, this is a new result.

\begin{proposition}\label{p:bivariate-generating-B}
The following identity holds
\[
1 + \sum_{d \geq 1} \Big( \sum_{\sigma \in \frakB_d} t^{\dim(\supp(\sigma))} z^{\exc_B(\sigma)} \Big) \dfrac{x^n}{(2n)!!}
=
\dfrac{(1-z)e^{x(1-z)/2}}{1-ze^{x(1-z)}} \, \left( \dfrac{z-1}{z-e^{x(z-1)}} \right)^{\tfrac{t-1}{2}}.
\]
\end{proposition}

\begin{proof}
We can slightly modify~\eqref{eq:exc_cycles_B} to obtain
\[
\sum_{\substack{\sigma \in \frakB_d \\ \supp(\sigma) = \rmX}} t^{\dim(\rmX)} z^{\exc_B(\sigma)}
= \Big( \sum_{\substack{\sigma \in \frakB(S_0) \\ \supp(\sigma) = \perp}} z^{\exc_B(\sigma_0)} \Big) \prod_{i=1}^k \Big( t \sum_{\sigma \in \frakC(|S_i|)} z^{\exc(\sigma)} \Big).
\] 
Using the (type B) generating functions of the factors deduced in the proofs of Lemmas \ref{l:cyc_exc} and \ref{l:cyc_exc_B}, and the type B compositional formula, we deduce
\[
1 + \sum_{d \geq 1} \Big( \sum_{\sigma \in \frakB_d} t^{\dim(\supp(\sigma))} z^{\exc_B(\sigma)} \Big) \dfrac{x^n}{(2n)!!}
=
\dfrac{B(z,x)}{\sqrt{A(z,x)}}\exp\Big(\dfrac{t\log (A(z,x))}{2}\Big),
\]
which equals $B(z,x)A(z,x)^{\tfrac{t-1}{2}}$.
Substituting the expressions for $A(z,x)$ and $B(z,x)$ in~\eqref{eq:generating_eulerian} and~\eqref{eq:generating_eulerian_B} yields the result.
\end{proof}

Specializing $t := 0$ and $x := 2x$ gives an alternative expression for the exponential generating function of the OEIS sequence \href{https://oeis.org/A156919}{A156919}.
In our context, these coefficients count the number of signed permutations whose action on $\RR^d$ has no nonzero fixed point weighted by the statistic $\exc_B$.

\subsection{Two bases basis for type B generalized permutahedra}

Faces of the standard simplex $\Delta_{[d]}$ correspond to linearly independent rays of the cone of generalized permutahedra in $\RR^d$. In the language of the present work, this means that $\{ \log[\Delta_S] \,:\, S \subseteq [d],\, |S| \geq 2 \}$ forms a linear basis of $\Xi_1(\pi_d)$ and that $\log[\Delta_S]$ is not the sum of the log-classes of other generalized permutahedra (other than trivial dilations of itself).
The goal of this section is to establish an analogous result for the type B case.

The standard simplex coincides with the \emph{weight polytope} $P_{\frakS_d}(\lambda_1)$ in the sense of Ardila, Castillo, Eur, and Postnikov~\citep{acep20submodular}, where $\lambda_1$ is the fundamental weight $(1,0,\dots,0)$ of $\frakS_d$. In this manner, the cross-polytope $P_{\frakB_d}(\lambda_1) = \Conv\{\pm e_i \,:\, i \in [d]\}$ is the type~B analog of the standard simplex. 
However, in the same paper the authors point out that the faces of the cross-polytope span a space of roughly half the desired dimension.
This is intuitively clear once we notice that the collection of faces of $P_{\frakB_d}(\lambda_1)$ entirely contained in one orthant span the same space in $\Xi_1(\pi^B_d)$ as the faces contained in the opposite orthant.
The following result shows that it is not possible to find a single type B generalized permutahedron whose faces generate $\Xi_1(\pi^B_d)$.

\begin{proposition}\label{p:min_num_gens}
Let $\cP = \{ \p_\alpha \}_\alpha$ be a collection of type B generalized permutahedra such that $\{ \log[\p_\alpha] \}_\alpha$ spans $\Xi_1(\pi^B_d)$. Then, $\cP$ contains at least $2^{d-1}$ full dimensional polytopes.
\end{proposition}

\begin{proof}
Let $\{\B{E}_\rmX\}_\rmX \subseteq \tits[\arr^\pm_d]$ be an Eulerian family of $\arr^\pm_d$.
The result follows from the following two facts, which we justify below.
\begin{enumerate}
\item The projection of $\log[\p]$ to the subspace $\Xi_1(\pi^B_d) \cdot \B{E}_\perp$ is zero unless $\p$ is full-dimensional.
\item $\dim_\RR(\Xi_1(\pi^B_d) \cdot \B{E}_\perp) = 2^{d-1}$.
\end{enumerate}

\paragraph{(1)} Let $\p$ be a type B generalized permutahedron that is not full-dimensional.
Then we can choose a face $F \neq O$ of $\arr^\pm_d$ such that $\p_F = \p$, for instance any maximal face in $N(\p,\p)$.
It follows from \citep[Lemma 11.12]{aa17} that $\B{H}_F \cdot \B{E}_\perp = 0$.
Therefore,
\[
\log[\p] \cdot \B{E}_\perp = \log([\p] \cdot \B{H}_F) \cdot \B{E}_\perp  = \log[\p] \cdot (\B{H}_F \cdot \B{E}_\perp) = 0.
\]

\paragraph{(2)} Recall that by definition, $\dim_\RR(\Xi_1(\pi^B_d) \cdot \B{E}_\perp) = \eta_\perp(\Xi_1(\pi^B_d))$.
Using Theorem~\ref{t:dims_simul_e-spaces_B}, this is the number of signed permutations $\sigma \in \frakB_d$ with $\supp(\sigma) = \perp$ and $\exc_B(\sigma) = 1$.
A signed permutation $\sigma \in \frakB_d$ with $\supp(\sigma) = \perp$ is the product of cycles on involution-inclusive subsets $S \subseteq [\pm d]$.
Moreover, each such cycle adds at least $1$ to the number of negations of $\sigma$.
Thus, a signed permutation with $\supp(\sigma) = \perp$ and $\exc_B(\sigma) = 1$ must be the product of either $1$ or $2$ cycles and have no excedances.
Such cycles are necessarily of the form $(d\, d-1\, \dots 1 \opp{d}\, \opp{d-1} \dots \opp{1})$.
Thus, the permutations counted by $\eta_\perp(\Xi_1(\pi^B_d))$ are in correspondence with \emph{unordered} pairs $\{S,[\pm d] \setminus S\}$ of involution-inclusive subsets of $[\pm d]$, and there are precisely $2^{d-1}$ many of them.

Alternatively, one can manipulate the generating function in Proposition~\ref{p:bivariate-generating-B} ($\tfrac{\partial}{\partial z} \big|_{t,z=0}$) to deduce that the coefficient of $t^0z^1x^d$ is $2^{d-1}$.
\end{proof}

In~\citep{ppr20shard}, Padrol, Pilaud, and Ritter construct a family of type B generalized permutahedra called \emph{shard polytopes}.
They show that any type B generalized permutahedron can be written uniquely as a signed Minkowski sum of these polytopes (up to translation).
Already in $\RR^3$, there are $14$ full-dimensional shard polytopes.
We proceed to construct a family of generators that achieves the minimum imposed by the previous proposition.

For a nonempty involution-exclusive subset $S \subseteq [\pm d]$, define the simplices
\[
\Delta_S = \Conv\{ e_i \, \mid \, i \in S \}
\qquad\qquad
\Delta^0_S = \Conv(\{{\bf 0}\} \cup \{ e_i \, \,:\, \, i \in S \}),
\]
where for $i \in [d]$, $e_{\opp{i}} = - e_i$.
Observe that the only full-dimensional simplices in this collection are $\Delta^0_S$ with $|S| = d$.
We say that an involution-exclusive subset $S \subseteq [\pm d]$ is \textbf{special} if in addition $\min\{ |i| \,:\, i \in S \} \in S$.
From now on, we will only consider simplices $\Delta_S$ and $\Delta^0_S$ for special sets $S$.

\begin{theorem}\label{t:B-gens}
Every type B generalized permutahedron can be written uniquely as a signed Minkowski sum of the simplices $\Delta_S$ and $\Delta^0_S$ with $S$ as above.
\end{theorem}

\begin{figure}[ht]
\[
\begin{gathered}
\begin{tikzpicture}
\newdimen\sh
\sh = 1mm
\draw [shift={(\sh,\sh)}, fill=gray!30!white] (0,0) -- (1,0) -- (0,1) -- (0,0);
\draw [shift={(\sh,-\sh)}, fill=gray!30!white] (0,0) -- (1,0) -- (0,-1) -- (0,0);
\draw [shift={(\sh,0)}, thick] (0,0) -- (1,0);
\draw [shift={(0,\sh)}, thick] (0,0) -- (0,1);
\draw [shift={(1.71\sh,1.71\sh)}, thick] (1,0) -- (0,1);
\draw [shift={(2\sh,-2\sh)}, thick] (1,0) -- (0,-1);
\node [shift={(0,3\sh)}, circle, inner sep=1pt,fill=black,draw] at (0,1) {};
\node [shift={(3\sh,0)}, circle, inner sep=1pt,fill=black,draw] at (1,0) {};
\end{tikzpicture}
\end{gathered}
\hspace*{.15\linewidth}
\tdplotsetmaincoords{60}{120}
\begin{gathered}
\begin{tikzpicture}[tdplot_main_coords]
\draw [->] (0,0,0) -- (1,0,0) node [left] {$e_1$};
\draw [->] (0,0,0) -- (0,1,0) node [right] {$e_2$};
\draw [->] (0,0,0) -- (0,0,1) node [above] {$e_3$};
\end{tikzpicture}
\end{gathered}
\hspace*{.05\linewidth}
\begin{gathered}
\begin{tikzpicture}[tdplot_main_coords]
\draw [shift={((0,.5,.5)} , very thick] (1,0,0) -- (0,1,0) -- (0,0,1) -- cycle;
\draw [shift={((0,.5,.5)} , dashed ] (0,1,0) -- (0,0,0) -- (1,0,0);
\draw [shift={((0,.5,.5)} , dashed ] (0,0,0) -- (0,0,1);
\draw [shift={((0,.5,-.5)} , very thick] (0,1,0) -- (0,0,-1) -- (1,0,0) -- (0,0,0) -- cycle -- (1,0,0);
\draw [shift={((0,.5,-.5)} , dashed] (0,0,0) -- (0,0,-1);
\draw [shift={((0,-.5,-.5)} , very thick] (0,0,0) -- (0,0,-1) -- (1,0,0) -- (0,-1,0) -- cycle -- (1,0,0);
\draw [shift={((0,-.5,-.5)} , dashed] (0,-1,0) -- (0,0,-1);
\draw [shift={((0,-.5,.5)} , very thick] (0,0,1) -- (0,0,0) -- (1,0,0) -- (0,-1,0) -- cycle -- (1,0,0);
\draw [shift={((0,-.5,.5)} , dashed] (0,-1,0) -- (0,0,0);
\end{tikzpicture}
\end{gathered}
\]
\caption{The type B generalized permutahedra that are \emph{signed Minkowski generators} in $\RR^2$. In $\RR^3$, we only show the $4 = 2^{3-1}$ full-dimensional generators.}
\end{figure}
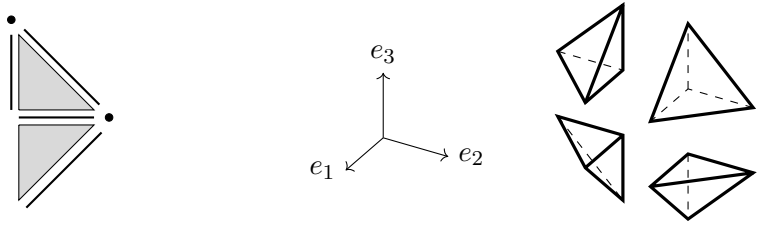

\begin{proof}
Observe that there are exactly $d$ zero-dimensional polytopes in this collection: $\Delta_S$ with $|S| = 1$.
Since they minimally generate all translations in $\RR^d$ and $h_1(\pi^B_d) = 3^d - d - 1$, the statement is equivalent to showing that the collection
\begin{equation}\label{eq:B-log-basis}
\{ \log[\Delta_S] \,:\, S \text{ special} ,\, |S| \geq 2 \} \cup \{ \log[\Delta^0_S] \,:\, S \text{ special} \}
\end{equation}
is linearly independent in $\Xi_1(\pi^B_d)$.
To prove this, we will use the following linear map, whose existence is guaranteed by \citep[Theorem 5]{mcmullen89}: 
\[
\begin{array}{RCL}
\Psi_1 : \Xi_1(\pi^B_d) & \longrightarrow & \RR \face^{d-1}[\arr^\pm_d] \\
\log[\p] & \longmapsto & \sum_F \vol_1(\p_F) \B{H}_F.
\end{array}
\]
Here, $\face^{d-1}[\arr^\pm_d]$ denotes the collection of $(d-1)$-dimensional faces of $\arr^\pm_d$ and $\vol_1$ denotes the $1$-dimensional normalized volume.
See Section~\ref{s:remarks} for more details.

The \textbf{edges} ($1$-dimensional faces) of $\Delta_S$ are of the form $\Delta_{ab} : = \Delta_{\{a,b\}}$ for distinct $a,b \in S$.
On the other hand, the edges of $\Delta^0_S$ are of the form $\Delta_{ab}$ for distinct $a,b \in S$ or of the form $\Delta^0_{a} : = \Delta^0_{\{a\}}$ for $a \in S$.
The normalized volume of all such edges is $1$.
Thus, $\log[\Delta_S]$ (resp. $\log[\Delta^0_S]$) is mapped to the sum of $\B{H}_F$ as $F$ ranges over the faces of $\arr^\pm_d$ that are maximal inside the normal cone of each edge of $\Delta_S$ (resp. $\Delta^0_S$).

The following is an explicit description of the normal cones for the edges of $\Delta_S$ and $\Delta^0_S$. Recall that if $j \in [d]$, then $x_{\opp{j}}$ denotes  $-x_j$.
\begin{equation}\label{eq:B-normal}
\begin{array}{RL}
N(\Delta_{ab},\Delta_S) = & \{ x \in \RR^d \,:\, x_a = x_b > x_s \text{ for all } s \in S \setminus \{a,b\} \} \\
N(\Delta_{ab},\Delta^0_S) = & \{ x \in \RR^d \,:\, x_a = x_b > 0 , \, x_a = x_b > x_s \text{ for all } s \in S \setminus \{a,b\} \} \\ 
N(\Delta^0_{a},\Delta^0_S) = & \{ x \in \RR^d \,:\, x_a = 0 > x_s \text{ for all } s \in S \setminus \{a\} \}
\end{array}
\end{equation}
Suppose
\[
\varphi = \sum \alpha_S \log[\Delta_S] + \sum \beta_S \log[\Delta^0_S] = 0.
\]
By studying the coefficient of different faces of $\arr^\pm_d$ in $\Psi_1(\varphi)$, we will conclude that all the coefficients $\alpha_S$ and $\beta_S$ are necessarily $0$, and therefore the collection~\eqref{eq:B-log-basis} is linearly independent.

We will start with the coefficients $\alpha$.
Fix $k \in [d-1]$, we prove by induction on $|S|$ that $\alpha_S = 0$ for all $S$ such that $k = \min\{|i| \,:\, i \in S\}$.
The base case is $|S| = 2$.
Let $S = \{k,j\}$ with $j \in \pm [k+1,d]$ and consider the following face of $\arr^\pm_d$
\[
x_1 > x_2 > \dots > x_{k-1} > x_{\opp{k}} = x_{\opp{j}} > x_{k+1} > \dots > \widehat{x_{|j|}} > \dots > x_d > 0,
\]
or equivalently,
\[
0 > x_{\opp{d}} > \dots > \widehat{x_{\opp{|j|}}} > \dots > x_{\opp{k+1}} > x_k = x_j > x_{\opp{k-1}} > \dots > x_{\opp{1}}.
\]
The hat over $x_{|j|}$ denotes that the corresponding inequality is missing (since we already imposed $x_{|j|} = \pm x_k$).
By the description of the normal cones in \eqref{eq:B-normal}, we observe that this face only appears in a normal cone of either:
\begin{itemize}
\item $\Delta_{S'}$ or $\Delta^0_{S'}$ with $\opp{k},\opp{j} \in S'$ and $S' \setminus \{\opp{k},\opp{j}\} \subseteq \pm [k+1,d] \cup \opp{[k-1]}$, or
\item $\Delta_{S'}$ with $k,j \in S'$ and $S' \setminus \{k,j\} \subseteq \opp{[k-1]}$.
\end{itemize}
Since we are only considering special sets $S'$ (that is, such that $\min\{ |i| \,:\, i \in S' \} \in S'$) and $k > 0$, the only possibility is $S' = \{k,j\} = S$. 
Thus, the coefficient of this face in $\Psi_1(\varphi)$ is $\alpha_{\{k,j\}} = 0$.

Now suppose $S = \{k,j_1,\dots,j_r\}$ with $j_1,\dots,j_r \in \pm[k+1,d]$.
If we proceed with the previous analysis for the face
\[
0 > x_{\opp{d}} > \dots > \widehat{x_{\opp{|j_i|}}} > \dots > x_{\opp{k+1}} > x_k = x_{j_1} > x_{j_2} > \dots > x_{j_r} > x_{\opp{k-1}} > \dots > x_{\opp{1}},
\]
we conclude that it appears in a normal cone of $\Delta_{S'}$ if and only if $\{k,j_1\} \subseteq S' \subseteq S$.
Thus, the coefficient of this face in $\Psi_1(\varphi)$ is $\sum_{\{k,j_1\} \subseteq S' \subseteq S} \alpha_{S'} = 0$.
By induction we have $\alpha_{S'} = 0$ for all $\{k,j_1\} \subseteq S' \subsetneq S$, and so we conclude $\alpha_S = 0$.

We proceed now with the coefficients $\beta$.
Fix $k \in [d-1]$ and $j \in \pm[k+1,d]$, and let $\cS = \cS(k,j)$ be the collection of special subsets $S$ such that $k,j \in S$ and $k = \min\{ |i| \,:\, i \in S\}$.
We first prove that $\beta_S = 0$ for any $S \in \cS$.
Let $S = \{k,j,s_1,\dots,s_r\} \in \cS$ and let $\{t_1,\dots,t_{r'}\} = [k+1,d] \setminus \{|j|,|s_1|,\dots,|s_r|\}$.
Consider the following $(d-1)$-dimensional face of $\arr^\pm_d$:
\[
x_1 > \dots > x_{k-1} > x_{\opp{s_1}} > \dots > x_{\opp{s_r}} > x_k = x_j > x_{t_1} > \dots > x_{t_{r'}} > 0.
\]
Its coefficient in $\Psi_1(\varphi)$ is
\[
y_S = \sum_{S' \in \cS,\, S' \cap \opp{S} = \emptyset} \beta_{S'} = 0,
\]
We claim that these relations, one for each $S \in \cS$, imply that $\beta_S = 0$ for all $S \in \cS$.
This follows since for any $S \in \cS$,
\begin{multline*}
\sum_{T \in \cS,\, T \cap \opp{S} = \emptyset} (-1)^{|T|} y_{\opp{T}}
= \sum_{T \in \cS,\, T \cap \opp{S} = \emptyset} \bigg( \sum_{S' \in \cS,\, S' \cap T = \emptyset} (-1)^{|T|} \beta_{S'} \bigg) \\
= \sum_{S' \in \cS} \bigg( \sum_{\substack{T \in \cS,\, T \cap \opp{S} = \emptyset \\ T \cap S' = \emptyset }} (-1)^{|T|} \bigg) \beta_{S'} = \pm \beta_S.
\end{multline*}
The last equality follows from the next observations:
\begin{itemize}
\item If $S' = S = \{k,j,s_1,\dots,s_r\}$, then the sum is over all $T \in \cS$ such that $\{k,j\} \subseteq T \subseteq \{k,j\} \cup \pm \big( [k+1,d] \setminus \{j,s_1,\dots,s_r\} \big)$.
Then, this sum is of the form
\[
\sum_{v \in \{-1,0,1\}^\ell} (-1)^{|v|}
\]
where $\ell = d-k-r-1 \geq 0$ and $|v|$ is the number of nonzero entries of $v$.
This sum is $(-1 + 1 - 1)^\ell = (-1)^\ell$.

\item If $S' \neq S$ and there is an element $i \in S \setminus S'$, then the subsets in the sum come in pairs $T,T'$ where $i \in T$ and $T' = T \setminus \{i\}$.
Since $(-1)^{|T|} + (-1)^{|T'|} = 0$, this shows that the whole sum is $0$ in this case.

\item If $S' \neq S$ and $S \subsetneq S'$, we can chose $i \in S' \setminus S$ and repeat the previous argument with pairs $T,T'$ where $\opp{i} \in T$ and $T' = T \setminus \{\opp{i}\}$.
\end{itemize}

Since each special subset $S$ with $|S| \geq 2$ is in $\cS(k,j)$ for some $k,j$, we are only left to prove that $\beta_{\{k\}} = 0$ for all $k \in [d]$.
For fixed $k \in [d]$, consider the following face of $\arr^\pm_d$:
\[
x_k = 0 > x_{d} > x_{d-1} > \dots > \widehat{x_{k}} > \dots > x_{1}.
\]
The coefficient of this face in $\Psi_1(\varphi)$ is $\sum_{\{k\} \subseteq S \subseteq [d]} \beta_S = 0$.
Since we have shown that $\beta_S = 0$ for any $S$ with $|S| \geq 2$, we conclude that $\beta_{\{k\}} = 0$.
\end{proof}

\begin{remark}\label{r:not_symmetric}
Observe that the generating collection $\{ \Delta_S \,:\, S \subseteq [d] \}$ for generalized permutahedra is invariant under the action of $\frakS_d$.
In contrast, the collection of generators for type B generalized permutahedra presented in the previous theorem fails to be invariant under the action of $\frakB_d$. This is not by accident. Already in $\RR^2$, we see that any collection of type B generalized permutahedra that contains a triangle $\p$ (full-dimensional simplex) and that is invariant under the action of $\frakB_d$, will contain the rotations of $\p$ by $90^\circ$, $180^\circ$, and $270^\circ$, all of which are necessarily different. Thus, such a collection will not attain the minimum number of full-dimensional polytopes required by Proposition~\ref{p:min_num_gens}.
In the following theorem, we present a different collection of generators for type B generalized permutahedra that is invariant under the action of $\frakB_d$.
\end{remark}

\begin{theorem}\label{t:B-sym-gens}
Every type B generalized permutahedron can be written uniquely as a signed Minkowski sum of the simplices $\{\Delta^0_S \,:\, S \subseteq [\pm d] $ \emph{involution-exclusive}$\}$.
\end{theorem}

\begin{proof}
First observe that for every $i \in [d]$, $\Delta^0_{\opp{i}}$ is a translation of $\Delta^0_i$; explicitly $\Delta^0_{\opp{i}} + \Delta_i = \Delta^0_i$.
Then, we could replace the pair $\Delta^0_i,\Delta^0_{\opp{i}}$ in the statement with the pair $\Delta^0_i,\Delta_i$.
Since $\{\Delta_i \,:\, i \in [d]\}$ minimally generates the translations in $\RR^d$, the statement is equivalent to showing that the collection
\begin{equation}\label{eq:B-log-basis-symmetric}
\{ \log[\Delta^0_S] \,:\, |S| \geq 2 \} \cup \{ \log[\Delta^0_i] \,:\, i \in [d] \}
\end{equation}
is a basis of $\Xi_1(\pi^B_d)$.
By comparing the number of elements, we see that it is enough to write the elements in the basis~\eqref{eq:B-log-basis} as a linear combination of the elements in the proposed basis above.
Since the proposed basis already contains all the elements of the form $\log[\Delta^0_S]$, we are only left to write $\log[\Delta_S]$ for $|S| \geq 2$ as a linear combination of the proposed basis. We do this by induction on $|S|$.

We use $\lambda = -1$ and $r = 1$ in Theorem~\ref{t:euler_dilations} to conclude that for any polytope $\p$:
\[
\log[-\p] = - \sum_{\q \leq \p} (-1)^{\dim(\q)} \log[\q].
\]
Applying this identity to $\p = \Delta^0_{S} = -\Delta^0_{\opp{S}}$ yields
\begin{equation}\label{eq:aux1}
\log[\Delta_S] = (-1)^{|S|} \log[\Delta^0_{\opp{S}}] + \sum_{T \subseteq S}(-1)^{|S \setminus T|}\log[\Delta^0_T] - \sum_{T \subsetneq S}(-1)^{|S \setminus T|}\log[\Delta_T].
\end{equation}
If $S = \{i,j\}$, this is
\[
\log[\Delta_S] = \log[\Delta^0_{\opp{S}}] + \log[\Delta^0_S] - \log[\Delta^0_i] - \log[\Delta^0_j].
\]
Observe that all of the terms in the right are in the proposed basis, so we have completed the base case of the induction.
If $|S| > 2$, we use~\eqref{eq:aux1} and observe that, by induction hypothesis, all the terms in the right can be written as a linear combination of the proposed basis.
\end{proof}

\section{Hopf monoid structure}\label{s:hopf}

Combinatorial species were originally introduced by Joyal~\citep{joyal} as a tool for studying generating power series from a combinatorial perspective.
A comprehensive introduction to the theory of species can be found in the work by Bergeron, Labelle, and Leroux~\citep{bll98species}.
The category of species possesses more than one monoidal structure.
Of central interest for the present work are the \emph{Cauchy} and \emph{Hadamard} product.
Aguiar and Mahajan~\citep{am10,am13} have explored these structures extensively, and have exploited this rich algebraic structure to obtain outstanding combinatorial results.
The first of these structures leads to the definition of Hopf monoids in species, a very active topic of research in recent years.

Aguiar and Ardila introduced the Hopf monoid of generalized permutahedra~${\sf GP}$ in~\citep{aa17}.
It contains many other interesting combinatorial Hopf monoids as submonoids.
In this section we show that the valuation~\eqref{eq:valuation_property} and translation invariance~\eqref{eq:translation_property} properties define a Hopf monoid quotient of~${\sf GP}$.

\subsection{Hopf monoids in a nutshell}

Let~${\sf set^\times}$ denote the category of finite sets with bijections as morphisms, and~${\sf Vec}$ the category of vector spaces and linear maps.
The category of species~${\sf Sp}$ is the functor category~$[{\sf set^\times} , {\sf Vec}]$.
It is a symmetric monoidal category under the \emph{Cauchy product}.
The Cauchy product of two species~$\sfp$ and~$\sfq$ is
\[
(\sfp \bdot \sfq)[I] = \bigoplus_{I = S \sqcup T} \sfp[S] \otimes \sfq[T].
\]
We say a species~$\sfh$ is a \emph{Hopf monoid} if it is a bimonoid with an \emph{antipode} in this monoidal category.

Let us make these definitions explicit.
A \textbf{species}~$\sfp$ consists of the following data:
\begin{enumerate}[i.]
\item For each finite set~$I$, a vector space~$\sfp[I]$.
\item For each bijection~$\sigma : I \rightarrow J$, a linear isomorphism~$\sfp[\sigma] : \sfp[I] \rightarrow \sfp[J]$.
These linear maps satisfy
\[
\sfp[\sigma \circ \tau] = \sfp[\sigma] \circ \sfp[\tau]
\qqand
\sfp[\Id] = \Id.
\]
\end{enumerate}
A morphism of species~$f : \sfp \rightarrow \sfq$ is a collection of linear maps
\[
f_I : \sfp[I] \rightarrow \sfq[I],
\]
one for each finite set~$I$, that commute with bijections.
That is,~$f_J \circ \sfp[\sigma] = \sfq[\sigma] \circ f_I$
for any bijection~$\sigma: I \rightarrow J$.

A \textbf{Hopf monoid} is a species~$\sfh$ together a collection of product, coproduct and antipode maps
\[
\begin{array}{rccccrccccrccc}
\mu_{S,T} : & \sfh[S] \otimes \sfh[T] & \rightarrow & \sfh[I] & \phantom{aaa} & \Delta_{S,T} : & \sfh[I] & \rightarrow & \sfh[S] \otimes \sfh[T] & \phantom{aaa} & \apode_I : & \sfh[I] & \rightarrow & \sfh[I] \\
& x \otimes y & \mapsto & x \bdot y & & & z & \mapsto & \sum z|_S \otimes z/_S & & & z & \mapsto & \apode_I(z)
\end{array}
\]
for all finite sets~$I$ and decompositions~$I = S \sqcup T$.
This morphisms satisfy certain naturality, (co)unitality, (co)associativity and compatibility axioms.
See~\citep[Section 2]{am13} and~\citep{aa17} for more details.

A Hopf monoid $\sfh$ is \textbf{commutative} if $x \bdot y = y \bdot x$ for all decompositions $I = S \sqcup T$ and structures $x \in \sfh[S]$, $y \in \sfh[T]$.
If $\sfh$ is a commutative Hopf monoid, then each space $\sfh[[d]]$ has the structure of a right $\tits[\arr_d]$-module, as follows.
Let $F \in \face[\arr_d]$ be the face associated to a composition $(S_1,\dots,S_k)$ of $[d]$. The associativity and coassociativity axioms imply that there are well defined maps
\[
\mu_F : \sfh[S_1] \otimes \dots \otimes \sfh[S_k] \rightarrow \sfh[[d]]
\qand
\Delta_F : \sfh[[d]] \rightarrow \sfh[S_1] \otimes \dots \otimes \sfh[S_k],
\]
obtained by iterating the product and coproduct maps in any meaningful way.
Then, $\sfh[[d]]$ is a right $\tits[\arr_d]$-module with the following structure:
\[
x \cdot \B{H}_F = \mu_F \circ \Delta_F (x).
\]
We will see that the module structure on generalized permutahedra in Section~\ref{s:braid} arises in this manner.

\subsection{Generalized permutahedra and the McMullen (co)ideal}\label{ss:mc}

As a species, ${\sf GP}[I]$ is the vector space with basis
\[
\GP[I] = \{ \p \subseteq \RR^I \,:\, \p \text{ is a generalized permutahedron} \}.
\]
The product~$\mu_{S,T}$ is defined by
\[
\mu_{S,T}(\p \otimes \q) = \p \times \q,
\]
for all permutahedra~$\p \in \GP[S]$ and~$\q \in \GP[T]$.
In particular,~${\sf GP}$ is a commutative monoid.
Let~$F$ be the face of the braid arrangement in~$\RR^I$ corresponding to the composition~$(S,T)$, and let~$v \in \relint(F)$.
Then, for any~$\p \in \GP[I]$, the face~$\p_v$ decomposes as a product of generalized permutahedra~$\p|_S \times \p/_S$, with~$\p|_S \in \GP[S]$ and~$\p/_S \in \GP[T]$, see \citep[Proposition 5.2]{aa17}.
The coproduct is defined by
\[
\Delta_{S,T}(\p) = \p|_S \otimes \p/_S.
\]
Aguiar and Ardila also give the following grouping-free and cancellation-free formula for its antipode.
For a generalized permutahedron~$\p \in \GP[I]$,
\[
\apode_I(\p) = (-1)^{|I|} \sum_{\q \leq \p} (-1)^{\dim(\q)} \q.
\]

We now introduce the subspecies~${\sf Mc}$ of~${\sf GP}$.
The space~${\sf Mc}[I] \subseteq {\sf GP}[I]$ is the subspace spanned by elements
\begin{equation}\label{eq:gensMc1}
\p \cup \q + \p \cap \q - \p - \q \qquad \text{for}\quad \p,\q  \in \GP[I] \text{ such that } \p \cup \q \text{ is convex},
\end{equation}
and
\begin{equation}\label{eq:gensMc2}
\p_{+t} - \p \qquad \text{for}\quad \p\in \GP[I] \text{ and } t \in \RR^I,
\end{equation}
where~$\p_{+t}$ denotes the Minkowski sum~$\p + \{t\}$.
The sums and differences in~\eqref{eq:gensMc1} and~\eqref{eq:gensMc2} correspond to the vector space structure of~${\sf GP}[I]$, not to Minkowski sum or difference.

\begin{remark}
Recall from Section~\ref{ss:subalgebra_polytope} that if $\p\cup\q$ is a polytope, then~$\p\cup\q$ and $\p \cap \q$ are necessarily generalized permutahedra.
Thus, the elements~\eqref{eq:gensMc1} are indeed in ${\sf GP}[I]$.
\end{remark}

The following result shows that ${\sf Mc}$ defines relations compatible with the Hopf monoid structure of ${\sf GP}$.
Ardila and Sanchez~\citep{as20valGP} prove a similar result for extended generalized permutahedra by realizing ${\sf Mc}$ as the kernel of a Hopf monoid morphism.

\begin{theorem}
The subspecies~${\sf Mc}$ is an ideal and a coideal of~${\sf GP}$.
That is,
\[
\mu_{S,T}\big( {\sf Mc}[S] \otimes {\sf GP}[T] \big)
\subseteq {\sf Mc}[I]
\qand
\Delta_{S,T} \big( {\sf Mc}[I] \big)
\subseteq {\sf Mc}[S] \otimes {\sf GP}[T] + {\sf GP}[S] \otimes {\sf Mc}[T],
\]
for any~$I = S \sqcup T$.
Therefore, the quotient species~$\widetilde{\Pi}$ defined by
\[
\widetilde{\Pi}[I] = {\sf GP}[I]/{\sf Mc}[I]
\]
inherits the Hopf monoid structure of~${\sf GP}$.
\end{theorem}

\begin{proof}
For generators of~${\sf Mc}$ of the form~\eqref{eq:gensMc2}, the result follows from the following two observations.
If~$\p \in \GP[S]$,~$\rr \in \GP[T]$ and~$t \in \RR^S$, then
\[
\p_{+t} \times \rr = (\p \times \rr)_{+(t,0)}.
\]
If~$\p \in \GP[I]$ and~$t \in \RR^T$, then
\[
\Delta_{S,T}(\p_{+t}) = (\p|_S)_{+t_S} \otimes (\p|_S)_{+t_T},
\]
where~$t_S$ and~$t_T$ denote the projections of~$t$ to~$\RR^S$ and~$\RR^T$, respectively.

We will now focus on generators of~${\sf Mc}$ of the form~\eqref{eq:gensMc1}.
Fix an arbitrary finite set~$I$ and a nontrivial decomposition~$I = S \sqcup T$.
Let~$v \in \RR^I$ be any vector in the interior of the corresponding face of the braid arrangement.

Suppose~$\p,\q,\p \cup \q \in {\sf GP}[S]$ and~$\rr \in {\sf GP}[T]$.
Then, 
\[
(\p \cup \q) \times \rr = (\p \times \rr) \cup (\q \times \rr),
\qquad\qquad
(\p \cap \q) \times \rr = (\p \times \rr) \cap (\q \times \rr),
\]
and~$(\p \cup \q) \times \rr = (\p \times \rr) \cup (\q \times \rr)$ is a polytope if and only if~$\p \cup \q$ is.
It follows that
\[
\mu_{S,T}\big( (\p \cup \q + \p \cap \q - \p - \q) \otimes \rr \big) = (\p \times \rr) \cup (\q \times \rr) + (\p \times \rr) \cap (\q \times \rr) - \p \times \rr - \q \times \rr \in {\sf Mc}[I].
\]
Since~${\sf GP}$ is commutative, this proves that~${\sf Mc}$ is an ideal.

Now, let~$\p,\q,\p \cup \q \in {\sf GP}[I]$.
There are two possibilities:
\begin{enumerate}[i.]
\item The face~$(\p \cup \q)_v$ of~$\p \cup \q$ is completely contained in~$\p$ or in~$\q$.
Without loss of generality, suppose the former.
Then~$(\p \cup \q)_v = \p_v$ and, necessarily,~$(\p \cap \q)_v = \q_v$.
Hence,
\[
\Delta_{S,T} \big( \p \cup \q + \p \cap \q - \p - \q \big) =
\Delta_{S,T}(\p) + \Delta_{S,T}(\q) - \Delta_{S,T}(\p) - \Delta_{S,T}(\q) = 0
\]

\item The face~$(\p \cup \q)_v$ is not contained in~$\p$ nor in~$\q$.
Hence,~$(\p \cup \q)_v = \p_v \cup \q_v$ and~$(\p \cap \q)_v = \p_v \cap \q_v$.
Expanding the first equality we have
\[
(\p \cup \q)|_S \times (\p \cup \q)/_S = (\p|_S \times \p/_S) \cup (\q|_S \times \q/_S).
\]
The union of two Cartesian products~$A \times B$ and~$C \times D$ is again a Cartesian product if and only if one contains the other or either~$A = C$ or~$B = D$.
By assumption, the is no containment between~$\p_v$ and~$\q_v$.
We can therefore assume without loss of generality that
\begin{equation}\label{eq:coideal1}
\p|_S = \q|_S.
\end{equation}
Projecting to~$\RR^S$ and~$\RR^T$, we further see that
\begin{equation}\label{eq:coideal2}
(\p \cup \q)|_S = \p|_S \cup \q|_S = \p|_S
\qqand
(\p \cup \q)/_S = \p/_S \cup \q/_S.
\end{equation}
In particular,~$\p/_S \cup \q/_S$ is a generalized permutahedron.
On the other hand, expanding~$(\p \cap \q)_v = \p_v \cap \q_v$, we have
\[
(\p \cap \q)|_S \times (\p \cap \q)/_S = (\p|_S \times \p/_S) \cap (\p|_S \times \q/_S) = \p|_S \times (\p/_S \cap \q/_S).
\]
Comparing factors, we deduce
\begin{equation}\label{eq:coideal3}
(\p \cap \q)|_S = \p|_S
\qqand
(\p \cap \q)/_S = \p/_S \cap \q/_S.
\end{equation}
Putting together~\eqref{eq:coideal1},~\eqref{eq:coideal2} and~\eqref{eq:coideal3}, we conclude
\begin{multline*}
\Delta_{S,T}\big(\p \cup \q + \p \cap \q - \p - \q\big)
= \Delta_{S,T}(\p\cup\q) + \Delta_{S,T}(\q\cap\q) - \Delta_{S,T}(\p) - \Delta_{S,T}(\q)\\
= \p|_S \otimes (\p/_S \cup \q/_S) + \p|_S \otimes (\p/_S \cap \q/_S) - \p|_S \otimes \p/_S - \p|_S \otimes \q/_S \\
= \p|_S \otimes \big( \p/_S\cup\q/_S + \p/_S\cap\q/_S - \p/_S - \q/_S \big)
\in {\sf GP}[S] \otimes {\sf Mc}[T].
\end{multline*}
\end{enumerate}
Thus, in either case we get~$\Delta_{S,T}\big(\p \cup \q + \p \cap \q - \p - \q\big) \in {\sf Mc}[S] \otimes {\sf GP}[T] + {\sf GP}[S] \otimes {\sf Mc}[T]$. That is,~${\sf Mc}$ is a coideal of~${\sf GP}$.
\end{proof}

Comparing the generators of the (co)ideal ${\sf Mc}$ with the relations defining McMullen's polytope algebra, it is natural to ask if $\widetilde{\Pi}[I]$ agrees with $\Pi(\pi_I)$, where $\pi_I \subseteq \RR^I$ is the standard permutahedron.
The answer is no.
For instance, in the polytope algebra, the structure of $\RR$-vector space is defined so that
\[
\alpha
\big( \big[
\begin{gathered}
\begin{tikzpicture}
\draw [very thick,fill=gray!30!white] (0,0) -- (.4,0) node [above] {\scriptsize$1$} -- (.8,0);
\end{tikzpicture}
\end{gathered}
\big]
-1 \big)
=
\big[
\begin{gathered}
\begin{tikzpicture}
\draw [very thick,fill=gray!30!white] (0,0) -- (.5,0) node [above] {\scriptsize$\alpha$} -- (1,0);
\end{tikzpicture}
\end{gathered}
\big]
-1,
\]
where the numbers over the segments denote their length.
However, if $\alpha$ is an irrational number, then
\[
\big( \alpha \begin{gathered}
\begin{tikzpicture}
\draw [very thick,fill=gray!30!white] (0,0) -- (.4,0) node [above] {\scriptsize$1$} -- (.8,0);
\end{tikzpicture}
\end{gathered} - \alpha \bullet \big)
-
\big( \begin{gathered}
\begin{tikzpicture}
\draw [very thick,fill=gray!30!white] (0,0) -- (.5,0) node [above] {\scriptsize$\alpha$} -- (1,0);
\end{tikzpicture}
\end{gathered} - \bullet \big) \notin {\sf Mc}[I].
\]
Nevertheless, the proof of the previous theorem works verbatim to show that the Hopf monoid operations are well-defined in $\Pi(\pi_I)$.

\begin{theorem}\label{t:McGP}
The species $\Pi$ defined by $\Pi[I] = \Pi(\pi_I)$ is a Hopf monoid, with product and coproduct defined for any decomposition $I = S \sqcup T$ by
\[
\mu_{S,T}([\q_1] \otimes [\q_2]) = [\q_1 \times \q_2]
\qqand
\Delta_{S,T}([\p]) = [\p|_S] \otimes [\p/_S]
\]
for all classes of generalized permutahedra $[\q_1]\in \Pi[S]$, $[\q_2]\in \Pi[T]$, and $[\p]\in \Pi[I]$.
Moreover, $\Pi$ is a Hopf monoid quotient of ${\sf GP}$ via the morphism $\p \mapsto [\p]$.
\end{theorem}

It immediately follows from the definitions above that, if $F \in \face[\arr_d]$ is the face corresponding to a composition $(S,T)$ of $[d]$, then
\[
[\p] \cdot \B{H}_F = \mu_{S,T} \circ \Delta_{S,T}([\p])
\]
for all generalized permutahedra $\p \subseteq \RR^d$. Thus, in the case of generalized permutahedra, the module structure of Section~\ref{s:braid} is precisely the one induced from the Hopf monoid structure above.

The antipode formula of~${\sf GP}$ descends to the quotient~$\Pi$, but it is no longer grouping-free in general.
The Euler map~\eqref{eq:Euler_map} allows us to write the antipode formula of~$\Pi$ in a very compact form:
\[
\apode_I([\p]) = (-1)^{|I|} [\p]^*.
\]

\subsection{Higher monoidal structures}

We have just proved that~$\Pi$ is a Hopf monoid in the symmetric monoidal category~$({\sf Sp},\bdot)$.
The algebra structure of each space~$\Pi[I]$ defined by McMullen can also be defined for~${\sf GP}$.
In both cases, this endows the species with the structure of a monoid in the symmetric monoidal category~$({\sf Sp},\times)$ of species with the \emph{Hadamard product}.
The Hadamard product of two species~$\sfp$ and~$\sfq$ is defined by
\[
(\sfp \times \sfq)[I] = \sfp[I] \otimes \sfq[I].
\]
Hence, a monoid in~$({\sf Sp},\times)$ consists of a species~$\sfp$ with an algebra structure on each space~$\sfp[I]$.
For generalized permutahedra, these structures are compatible in a very special way.

\begin{theorem}
The species of generalized permutahedra~${\sf GP}$ and its quotient $\Pi$ are~$(2,1)$-monoids in the 3-monoidal category~$({\sf Sp},\bdot,\times,\bdot)$.
\end{theorem}

See~\citep[Chapter 7]{am10} for the definition of higher monoidal categories and of monoids in such categories.
The notation~$(2,1)$ indicates that~${\sf GP}$ is a monoid with respect to the first two monoidal structures (Cartesian product and Minkowski sum, respectively) and a comonoid with respect to the last (coproduct maps~$\Delta_{S,T}$).

\begin{proof}
We only discuss the remaining compatibility axioms:
the compatibility between Cartesian product and Minkowski sum, and the compatibility between Minkowski sum and the coproduct.

The compatibility between Cartesian product and Minkowski sum boils down to the identity
\[
(\p_1 + \p_2) \times (\q_1 + \q_2) = (\p_1 \times \q_1) + (\p_2 \times \q_2)
\]
for~$\p_1,\p_2 \in \GP[S]$ and~$\q_1,\q_2 \in \GP[T]$, which one easily verifies for arbitrary sets $\p_1,\p_2 \subseteq \RR^S$ and~$\q_1,\q_2 \subseteq \RR^T$.

On the other hand, the compatibility between Minkowski sum and the coproduct is equivalent to the following identity for generalized permutahedra~$\p,\q \in \GP[I]$:
\[
(\p + \q)|_S \otimes (\p + \q)/_S = (\p|_S + \q|_S) \otimes (\p/_S + \q/_S).
\]
This follows by projecting the identity
$
(\p + \q)_v = \p_v + \q_v
$
to~$\RR^S$ and~$\RR^T$, respectively, where~$v$ is any vector in the interior of the face of the braid arrangement corresponding to the composition~$(S,T)$.
\end{proof}

The compatibility between Minkowski sum and the Hopf monoid operations refines the last statement in Theorem~\ref{t:mod_str}; which, in the language of this section, states that the maps $\mu_{S,T} \circ \Delta_{S,T}$ are compatible with the Minkowski sum operation.

\section{Final remarks and questions}\label{s:remarks}

\subsection*{1} Eulerian numbers are defined for any Coxeter group~$W$ in terms of~$W$-descents.
For the Coxeter groups of type A and B, descents and (B-)excedances are equally distributed, so we can interpret the~$W$-Eulerian polynomials as the generating functions for (B-)excedances.
However, the joint distributions of~$\big(|\supp(\cdot)|,\des(\cdot)\big)$ and~$\big(|\supp(\cdot)|,\exc(\cdot)\big)$ do not agree.
Therefore, Theorems~\ref{t:dims_simul_e-spaces_A} and~\ref{t:dims_simul_e-spaces_B} cannot immediately be expressed in terms of descents.

Extending the results of this section to other Coxeter groups~$W$ requires to find the correct notion of~$W$-excedance for other types.
Is there a nice analog of the result of Ardila, Benedetti, and Doker, and of Theorem~\ref{t:B-gens} in type D?

\subsection*{2} McMullen~\citep{mcmullen89} also studied valuation relations for the collection of polyhedral cones in~$V$.
The \emph{full cone group} of~$V$ is generated by the classes~$[C]$, one for each polyhedral cone~$C \subseteq V$.
They satisfy the following relation:
\begin{equation}\label{eq:val_cone}
[C_1 \cup C_2] = [C_1] + [C_2]
\end{equation}
whenever~$C_1 \cup C_2$ is a cone and~$C_1 \cap C_2$ is a proper face of~$C_1$ and of~$C_2$.
Note that this is not to say that the class~$[C_1 \cap C_2]$ is zero.

A cone of an arrangement~$\arr$ is any convex cone obtained as the union of faces of~$\arr$.
The space of formal linear combinations of cones~$\Omega[\arr]$ is a right~$\tits[\arr]$ module under the following operation.
If~$C$ is a cone of~$\arr$ and~$F \in \face[\arr]$, then
\[
C \cdot \B{H}_F = \begin{cases}
T_{\tilde{F}} C & \text{if } F \subseteq C,\\
0 & \text{otherwise},
\end{cases}
\]
where~$\tilde{F} \leq C$ is the minimum face of~$C$ containing~$F$ and~$T_{\tilde{F}} C$ denotes the tangent cone of~$C$ at~$\tilde{F}$. The relation~\eqref{eq:val_cone} is compatible with this action, and defines a quotient module~$\opp{\Omega}[\arr]$.

Restricting to the case of all the braid arrangements,~$\Omega$ defines a Hopf monoid in species. The product is defined by means of the Cartesian product. Let~$C$ be a cone of the braid arrangement in~$\RR^I$ and~$F$ the face corresponding to the composition~$(S,T) \vDash I$. If~$F \subseteq C$, the tangent cone~$T_{\tilde{F}} C$ decomposes as a product~$C|_S \times C/_S$ of cones in~$\RR^S$ and~$\RR^T$. The coproduct of~$\Omega$ is defined as follows:
\[
\Delta_{S,T}(C) = \begin{cases}
C|_S \otimes C/_S & \text{if } F \subseteq C \\
0 & \text{otherwise}.
\end{cases}
\]
With these operations,~$\Omega$ is isomorphic to the Hopf monoid of preposets~${\bf Q}$ considered in~\citep{aa17}.
Relation~\eqref{eq:val_cone} defines a Hopf monoid quotient~$\opp{\Omega}$.
Under a suitable change of basis,~$\opp{\Omega}$ is isomorphic to the dual Hopf monoid of faces~${\bf \face}^*$ defined in~\citep[Chapter 12]{am10}.

\subsection*{3} There is a Hopf monoid morphism~${\sf GP} \rightarrow \Omega$, whose components~${\sf GP}[I] \rightarrow \Omega[I]$ are defined as follows:
\begin{equation}\label{eq:GP-Omega}
\p \longmapsto \sum_{\q \leq \p} N(\q,\p) = \sum_{C \in \face_\p} C.
\end{equation}
Moreover, this is a morphism of~$(2,1)$-monoids in the 3-monoidal category~$({\sf Sp},\bdot,\times,\bdot)$, where the monoidal structure of~$\Omega$ under the Hadamard product is given by
\[
C_1 \cdot C_2 = \begin{cases}
C_1 \cap C_2 & \text{ if } \relint(C_1) \cap \relint(C_2) \neq \emptyset,\\
0 & \text{otherwise}.
\end{cases}
\]
That this map defines a morphism of monoids under the Hadamard product is equivalent to the following fact: the normal fan of~$\p + \q$ is the common refinement of~$\face_\p$ and~$\face_q$.

The map~\eqref{eq:GP-Omega} does \textbf{not} induce a well defined morphism~$\Pi \rightarrow \opp{\Omega}$.
In~\citep[Theorem 5]{mcmullen89}, McMullen shows that
\[
\p \longmapsto \sum_{\q \leq \p} \vol(\q) N(\q,\p)
\]
induces an injective map~$\Pi[I] \rightarrow \opp{\Omega}[I]$, where~$\vol(\q)$ is the \emph{normalized} volume of~$\q$ in the affine space spanned by~$\p$.
Moreover, one can verify that the induced morphism~$\Pi \rightarrow \opp{\Omega}$ is a morphism of Hopf monoids.
Is it possible to endow~$\opp{\Omega}$ with the structure of a~$(2,1)$-monoid so that the morphism above is a morphism of~$(2,1)$-monoids?

Such a structure on $\bar{\Omega}[I]$ would contain a subalgebra isomorphic to the Möbius algebra~$B^*(M)$ introduced by Huh and Wang in~\citep[Definition 5]{hw17topheavy}, where~$M$ is the matroid associated with the braid arrangement~$\arr_I$ in $\RR^I$.

\section*{Acknowledgments}

I thank my advisor Marcelo Aguiar for his suggestion to explore the interplay between dilations of polytopes and the characteristic operations in the Hopf monoid of generalized permutahedra. The presentation of this paper has substantially improved thanks to our conversations. I also thank Federico Ardila, Federico Castillo, and Mario Sanchez for their feedback and helpful discussions. I am also grateful to the anonymous referees for their valuable suggestions; which in particular led to Remark~\ref{r:not_symmetric} and the subsequent theorem.

\bibliographystyle{plain}
\bibliography{bib}

\end{document}